\documentclass[12pt,oneside, reqno]{amsart}
\usepackage{geometry}
\usepackage{amsmath,amssymb,amsfonts}
\usepackage[mathscr]{eucal}
\usepackage{mathdots}
\usepackage{color}
\usepackage [all,2cell,color]{xy}

\usepackage{enumerate}

\usepackage{tikz}
\usepackage{tikz-3dplot}
\usepackage{tikz-cd}
\usetikzlibrary{decorations.markings,decorations.pathmorphing,shapes}
\usepackage{xcolor}
\usepackage{lipsum}

\usepackage[textwidth=3cm, textsize=small, colorinlistoftodos]
{todonotes}

\usepackage[utf8]{inputenc}
\DeclareFontFamily{U}{min}{}
\DeclareFontShape{U}{min}{m}{n}{<-> udmj30}{}

\usetikzlibrary{positioning} 
\usetikzlibrary{decorations.pathreplacing}
\usetikzlibrary{arrows, decorations.markings} 
\pgfarrowsdeclarecombine{twotriang}{twotriang}{stealth}{stealth}
{stealth}{stealth}
\tikzset{mono/.style={>-stealth}} 
\tikzset{epi/.style={-twotriang}} 
\tikzset{arrow/.style={->}}
\tikzset{arrowshorter/.style={->, shorten <=2pt, shorten >=2pt}}
\tikzset{twoarrowlonger/.style={double,double distance=1.5pt,
shorten <=5pt,shorten >=6pt,
decoration={markings,mark=at position -4pt with {\arrow[scale=1.75]{>}}},
preaction={decorate}}} 

\tikzset{mapstikz/.style={-stealth, 
decoration={markings,mark=at position 0pt with {\arrow[scale=0.5]{|}}}, preaction={decorate}}}
\tikzset{dot/.style={circle,draw,fill,inner sep=1pt}}

\tikzset{ 
    vnode/.style={circle, radius=2pt, minimum size=4pt, draw, fill, inner sep=0, label={[below,text height=5mm]:#1}}}
\tikzset{boxy/.style={baseline={([yshift=0.5ex]current bounding box.center)}}}

\usetikzlibrary{shapes}
\tikzset{vellipsegreenone/.style={draw, ellipse, minimum width=0.6*#1, minimum height=1.5*#1, green, thick}}
\tikzset{vellipsepurpleone/.style={draw, ellipse, minimum width=0.6*#1, minimum height=1.5*#1, purple, thick}}

\newdir{ >}{{}*!/-5pt/@{>}}

\theoremstyle{plain}

\newtheorem{prop}[equation]{Proposition}
\newtheorem{lemma}[equation]{Lemma}

\theoremstyle{definition}
\newtheorem{definition}[equation]{Definition}
\newtheorem{example}[equation]{Example}
\newtheorem{remark}[equation]{Remark}

\numberwithin{equation}{section}

\def\edge#1#2{\;{}_{#1\!\!}\bullet\!\!\rule[0.53ex]{1.5ex}{0.15ex}\!\!\bullet_{#2}}

\newcommand{\aug}{\operatorname{aug}}

\newcommand{\Cat}{\mathcal Cat}

\newcommand{\Deltaop}{\Delta^{\op}}
\newcommand{\Hom}{\operatorname{Hom}}
\newcommand{\Hor}{\mathcal Hor}
\newcommand{\hor}{\operatorname{hor}}
\newcommand{\id}{\operatorname{id}}
\newcommand{\ob}{\operatorname{ob}}
\newcommand{\op}{\operatorname{op}}
\newcommand{\Seg}{\mathcal Seg}
\newcommand{\Set}{\mathcal{S}\!\operatorname{et}}
\newcommand{\SSet}{s\mathcal{S}\!\operatorname{et}}
\newcommand{\sq}{\operatorname{sq}}
\newcommand{\st}{\operatorname{st}}
\newcommand{\Ver}{\mathcal Ver}
\newcommand{\ver}{\operatorname{ver}}

\newcommand{\union}{\sqcup}
\newcommand{\bigunion}{\bigsqcup}

\definecolor{frenchblue}{rgb}{0.0, 0.45, 0.73}
\usepackage[colorlinks,citecolor=frenchblue, linkcolor=frenchblue]{hyperref}
\usepackage{float}
\usepackage{adjustbox}
\usepackage{xcolor}

\begin{document}

\title{2-Segal sets from cuts of rooted trees}

\author[J.E.\ Bergner]{Julia E.\ Bergner}
\author[O.\ Borghi]{Olivia Borghi}
\author[P.\ Dey]{Pinka Dey}
\author[I.\ G\'alvez]{Imma G\'alvez-Carrillo}
\author[T.\ Hoekstra]{Teresa Hoekstra-Mendoza}

\address{Department of Mathematics, University of Virginia, Charlottesville, VA 22904}
\email{jeb2md@virginia.edu}

\address{School Of Mathematics and Statistics, The University of Melbourne, Melbourne, Victoria, Australia}
\email{oborghi@student.unimelb.edu.au}

\address{Department of Mathematics, National Institute of Calicut, Kozhikode, Kerala, India - 673601}
\email{pinkadey11@gmail.com}

\address{Departamento de Álgebra, Geometría y Topología, Universidad de M\'alaga; IMTECH, Universitat Polit\`ecnica de Catalunya;
Centre de Recerca Matem\`atica.}
\email{imma.galvez@uma.es}

\address{Centro de Investigacion en Matem\'aticas A. C. Jalisco S/N, Col. Valenciana, 36023, Guanajuato, Gto, México}
\email{maria.idskjen@cimat.mx}

\date{\today}

\renewcommand{\subjclassname}{\textup{2020} Mathematics Subject Classification}


\begin{abstract}
	The theory of 2-Segal sets has connections to various important constructions such as the Waldhausen $S_\bullet$-construction in algebraic $K$-theory, Hall algebras, and (co)operads. In this paper, we construct 2-Segal sets from rooted trees and explore how these applications are illustrated by this example.  
\end{abstract}

\maketitle

\section{Introduction}

The notion of decomposition space was introduced by G\'alvez-Carrillo, Kock, and Tonks in \cite{dec}, and independently by Dyckerhoff and Kapranov in \cite{dk} under the name of 2-Segal space.  A 2-Segal space is a simplicial object that satisfies a condition weaker than the Segal condition.  Recall that, given a simplicial set $K$, the Segal condition requires all maps
\[ \varphi_n \colon  K_n\to \underbrace{K_1\times_{K_0}  \dots \times_{K_0} K_1}_n \]
to be isomorphisms for $n\geq 2$.  A $1$-\emph{Segal set}, or simply a \emph{Segal set}, is a simplicial set that satisfies the Segal condition.  Viewing $0$-simplices of a Segal set as objects and $1$-simplices as morphisms,  the fact that $\varphi_2$ is an isomorphism allows us to define a unique composition law
\[ K_1 \times_{K_0} K_1 \xleftarrow{\varphi_2} K_2 \overset{d_1}{\rightarrow} K_1. \]
A 2-Segal set is also a simplicial set, but it encodes a weaker algebraic structure. In particular, composites may or may not exist, and even when they do exist, they need not be unique. However, composition is associative in an appropriate sense. 

Analogously to Segal spaces, 2-Segal spaces are homotopy-theoretic analogues of 2-Segal sets, wherein simplicial sets are replaced by simplicial spaces, fibre products with their homotopy counterparts, and isomorphisms by weak equivalences.
    
The output of the Waldhausen's $S_{\bullet}$-construction constitutes an important family of examples of 2-Segal spaces.  In \cite{2s}, it was proven that the category of augmented stable double categories is equivalent to the category of 2-Segal sets via a generalized version of the discrete $S_{\bullet}$-construction, and a homotopical version for 2-Segal
spaces was proved in \cite{2ss}.  Another key feature of 2-Segal sets is that, under some finiteness assumptions, they give rise to Hall algebras or more general Hall categories.  When this construction is generalized to more general 2-Segal objects, it can be used to recover well-known Hall algebras in representation theory \cite{dk}. 

Dyckerhoff and Kapranov show that the category of 2-Segal sets is also equivalent to the category of invertible (co)operads in $\Set$ by \cite[Theorem 3.6.7]{dk}. This comparison uses a $X_1$-coloured cooperad $\mathcal{Q}_X$ corresponding to any simplicial set $X$;   this theory was further developed for 2-Segal spaces by Walde \cite{Wal21}. 

In this paper, we give a construction of a 2-Segal set $X^T$ arising from a rooted tree $T$, for which the face maps are described in terms of admissible cuts as defined in \cite{dec}, where a homotopical version is given for a category of all rooted trees.  This construction can furthermore be viewed as a specialization of the construction of a 2-Segal set $X^G$ associated to a graph $G$ as described in \cite{2s}.  Both of these families of examples give 2-Segal sets that are almost never 1-Segal.

By the results of \cite{2s}, we can associate to $X^T$ a pointed stable double category, and for certain families of trees we give an explicit description of this construction.  We also investigate the Hall algebra associated to $X^T$, and determine a condition under which it is commutative. 

Given a tree $T$, we can also forget its tree structure and think of its underlying graph $G=U(T)$, and thus compare the 2-Segal sets $X^T$ and $X^G$.   Although the naturally arising map $U \colon X^T \to X^G$ is a simplicial map, it is neither CULF nor relatively Segal; see Definitions \ref{culfdef} and \ref{relSegal}.  As a consequence, the inclusion $U \colon X^T \to X^G$ does not necessarily induce a homomorphism between the associated Hall algebras $\mathcal H^T$ and $\mathcal H^G$. Intuitively, these negative results are a consequence of the fact that decompositions of underlying graphs cannot always be lifted to decompositions of trees.

We conclude with the study of the (co)operads $\mathcal{Q}_{X^T}$ and $\mathcal{Q}_{X^G}$ associated with the 2-Segal sets $X^T$ and $X^G$.

\subsection*{Organization}
In Section \ref{background} we discuss the basic notions of 2-Segal sets and in Section \ref{rootedtrees} we define the construction of 2-Segal sets arising from rooted trees. 
In Section \ref{double} we describe the stable double category associated with the 2-Segal sets.  The relation between the 2-Segal set arising from a rooted tree and that from the underlying graph is studied in Section \ref{relation}.  The associated Hall algebra structure is analysed in Section \ref{hall}.  Finally, in Section \ref{operad}, we study the (co)operad associated with the 2-Segal sets.

\subsection*{Acknowledgements}
This paper is part of the authors’ Women in Topology IV project. We would like to thank the organizers of the Women in Topology IV workshop, as well as the Hausdorff Research Institute for Mathematics, where the workshop was held. The first-named author was partially supported by NSF grant DMS-1906281. The third-named author was supported by the NBHM grant no.\ 16(21)/2020/11. The fourth-named author was partially supported by Spanish Ministry of Science and Catalan Government Grants PID2019-103849GB-I00, PID2020-117971GB-C22, PID2023-149804NB-I00, CEX2020-001084-M and 2021SGR-00603, and by the Danish National Research Foundation through the Copenhagen Centre for Geometry and Topology (DNRF151).

\section{Background on 2-Segal sets} \label{background}

Recall that a \emph{simplicial set} is a functor $K \colon \Deltaop \rightarrow \Set$, where $\Set$ is the category of sets and $\Delta$ is the category of finite ordered sets $[n]=\{0 \leq 1 \leq \cdots \leq n\}$, for $n \geq 0$, and order-preserving functions between them.  We denote the set $K([n])$ by $K_n$.  A morphism between simplicial sets is a natural transformation of functors, and we denote the category of simplicial sets by $\SSet$.  In other words, a simplicial set $K$ is a collection of sets $K_n$ for each $n \geq 0$, together with face maps $d_i \colon K_n\rightarrow K_{n-1}$ for $0 \leq i \leq n$ and degeneracy maps $s_i \colon K_i \rightarrow K_{i+1}$ for $0 \leq i \leq n$ that satisfy the relations $d_id_j=d_{j-1}d_i$ and $d_is_j=s_{j-1}d_i$ for $i<j,$ $d_js_j=d_{j+1}s_j=id$, $d_is_j=s_jd_{i-1}$ for $i>j+1 $ and $s_is_j =s_{j+1}s_i$ for $ i \leq j$.  

A key example is the $n$-\emph{simplex} for any $n \geq 0$, which is the representable simplicial set $\Delta[n]$ defined by 
\[ \Delta[n]_k = \Hom_\Delta([k], [n]). \]
By the Yoneda Lemma, we have that
\[ \Hom_{\SSet}(\Delta[n], K) = K_n \]
for any $n \geq 0$.

An important sub-simplicial set of $\Delta[n]$ is the \emph{spine} $G[n]$, which is a union of 1-simplices, given by the colimit of the diagram
\[ \Delta[1] \overset{d^0}{\leftarrow} \Delta[0] \overset{d^1}{\rightarrow} \Delta[1] \leftarrow \cdots \leftarrow \Delta[0] \overset{d^1}{\rightarrow} \Delta[1], \]
with $n$ copies of $\Delta[1]$.  The inclusion $G[n] \hookrightarrow \Delta[n]$ induces, for any simplicial set $K$, a map of sets
\[ \Hom_{\SSet}(\Delta[n], K) \rightarrow \Hom_{\SSet}(G[n], K). \]
Using the definition of $G[n]$, one can check that
\[ \Hom(G[n], K) \cong \underbrace{K_1 \times_{K_0} \cdots \times_{K_0} K_1}_n. \]
We thus obtain, for each $n$, a \emph{Segal map} 
\[ K_n \rightarrow \underbrace{K_1 \times_{K_0} \cdots \times_{K_0} K_1}_n. \]

\begin{definition} 
A \emph{1-Segal set}, or simply \emph{Segal set}, is a simplicial set $K$ for which the Segal maps are isomorphisms for $n \geq 2$.
\end{definition}

We often refer to the condition that the Segal maps be isomorphisms as the \emph{Segal condition}.  Note that the Segal condition for $n=0,1$ is automatically satisfied, since in those cases $G[n]=\Delta[n]$.

If we think of the 0-simplices of a Segal set $K$ as ``objects" and its 1-simplices as ``morphisms", then the Segal condition tells us that $K$ behaves something like a category, in that any finite string of composable arrows has a unique composite.  More specifically, for any Segal map $K$, consider the diagram
\begin{equation} \label{composition}
K_1 \times_{K_0} K_1 \overset{(d_0,d_2)}{\longleftarrow} K_2 \overset{d_1}{\rightarrow} K_1. 
\end{equation}
If $K$ is a Segal set, then the map $(d_0,d_2)$ is an isomorphism, so we can define composition by $d_1 \circ (d_0, d_2)^{-1}$.

To formalize this idea further, let us recall the definition of the nerve of a category.

\begin{definition}
    Let $\mathcal C$ be a small category.  Its \emph{nerve} $N \mathcal C$ is the simplicial set given by
    \[ (N \mathcal C)_n = \Hom_{\Cat}([n], \mathcal C). \]
\end{definition}

Using the composition law defined above, and verifying the required identity conditions, we obtain the following result.

\begin{prop}
    A simplicial set is a Segal set if and only if it is isomorphic to the nerve of a category.
\end{prop}

Now we can generalize this idea to the notion of a 2-Segal set.  Given any triangulation $\mathcal T$ of a regular $(n+1)$-gon by its vertices, which are assumed to be labeled cyclically by $\{0,1, \ldots, n\}$, we can define a simplicial set $\Delta[\mathcal T]$ with 0-, 1-, and 2-simplices as indicated by the triangulation.  

\begin{example} \label{squares}
    There are two triangulations of the square:
    \[ \xymatrix{\mathcal T_1: & 3 & 2 \ar[l]  & \mathcal T_2:& 3 & 2 \ar[l] \\
    & 0 \ar[u] \ar[r] \ar[ur] & 1 \ar[u] && 0 \ar[u] \ar[r] & 1. \ar[u] \ar[ul]} \]
    Each of the associated simplicial sets includes naturally into the 3-simplex $\Delta[3]$.
\end{example}

More generally, for a triangulation $\mathcal T$ of an $(n+1)$-gon, there is a natural inclusion $\Delta[\mathcal T] \rightarrow \Delta[n]$ that in turn induces a map
\[ \Hom_{\SSet}(\Delta[n], K) \rightarrow \Hom_{\SSet}(\Delta[\mathcal T], K). \]
Similarly to the case of 1-Segal maps above, we can write these \emph{2-Segal maps} as
\[ K_n \rightarrow \underbrace{K_2 \times_{K_1} \cdots \times_{K_1} K_2}_{n-1}. \]

\begin{definition}
    A \emph{2-Segal set} is a simplicial set $K$ such that the 2-Segal maps are isomorphisms for every $n \geq 3$ and every triangulation $\mathcal T$ of a regular $(n+1)$-gon by its vertices.
\end{definition}

Now, a 2-Segal map behaves something like a category, since it still has ``objects" and ``morphisms", but composition need not exist, and when it does it need not be unique.  Looking back at \eqref{composition}, there is now no reason for the map $(d_0, d_2)$ to be an isomorphism, so all we can do is take the preimage of $X_1 \times_{X_0} X_1$ under this map and apply $d_1$.  If the preimage over some point is empty, then there is no composite at all; if it has more than one element, and those elements are not identified by $d_1$, then we obtain multiple composites.

However, this composition is associative when the notion makes sense.  Returning to Example \ref{squares}, we can think of the arrow $0 \rightarrow 2$ in the left-hand square as a composite of the arrows $0 \rightarrow 1$ and $1 \rightarrow 2$, and then $0 \rightarrow 3$ is obtained by taking the composite of $0 \rightarrow 2$ and $2 \rightarrow 3$.  We can similarly think of the right-hand square as encoding the triple composite in a different order.  The 2-Segal condition essentially says that either of these two configurations in a simplicial set $K$ uniquely determines a 3-simplex, namely, that the triple composite $0 \rightarrow 3$ must be the same for both.

\section{Rooted trees and their associated 2-Segal sets} \label{rootedtrees}

In this section, we recall some terminology for trees and then describe the construction of a 2-Segal set arising from a rooted tree.
	
Recall that a \emph{graph} $G$ consists of a pair $(V(G), E(G))$, where $V(G)$ is a set of vertices and $E(G)$ is a set of edges between vertices.  We assume all our graphs are finite, in the sense that both of these sets are finite. A \emph{tree} is a graph with no cycles.  In particular, a tree has at most one edge between any two vertices, and it has no loops, or edges that start and end at the same vertex.   
 
A \emph{rooted tree} $T$ has a distinguished vertex $v_0$, called the \emph{root}. We can regard the set $V(T)=\{v_0, v_1, \ldots, v_n\}$ of vertices of $T$ as a partially ordered set with $v_0 \leq v_j$ for every $v_j \in V(T)$; each non-root vertex $v_j$ has exactly one immediate predecessor $v_i$, and the set of the pairs $(v_i, v_j)$ corresponds exactly to the set of edges $E(T)$.  Thus for each vertex $v_j$, the vertices $v_i$ with $v_i\leq v_j$ form a totally ordered set $v_0<\dots<v_j$. 

We visually depict our trees with the root at the bottom and with the edges moving upward as indicated by these total orders; we can hence speak of edges ``above" or ``below" a given vertex.  We also typically depict the root by a filled black vertex.  See Figure \ref{figtr1} for an example. 
\addtocounter{equation}{1}
	{\scriptsize
		\begin{figure}
			\begin{tikzpicture}[scale=.70,line cap=round,line join=round,>=triangle 45,x=1cm,y=1cm, dot/.style = {circle, gray, minimum size=#1,
					inner sep=0pt, outer sep=0pt}, 
				dot/.default = 3pt]
				\draw [] (0,-1)-- (0,0);
				\draw [] (0,0)-- (-1,1);
				\draw [] (0,0)-- (0,1);
				\draw [] (0,0)-- (1,1);
				\begin{scriptsize}
					\draw [fill=black] (0,-1) circle (3pt);
					\draw[color=black] (-.5,-1) node {$v_0$};
					\draw [fill=white] (0,0) circle (3pt);
					\draw[color=black] (-.5,0) node {$v_1$};
					\draw [fill=white] (-1,1) circle (3pt);
					\draw[color=black] (-1.2,1.43) node {$v_{2}$};
					\draw [fill=white] (0,1) circle (3pt);
					\draw[color=black] (0.16,1.43) node {$v_{3}$};
					\draw [fill=white] (1,1) circle (3pt);
					\draw[color=black] (1.2,1.43) node {$v_{4}$};
				\end{scriptsize}
			\end{tikzpicture}
			\caption{Example of a rooted tree.} 
			\label{figtr1}
		\end{figure}
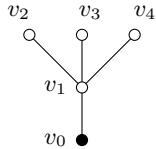 }
	
We sometimes impose the additional condition that our trees be planar.  A \emph{planar rooted tree} is a rooted tree together with, for each vertex $v_i$, a total order on the set of edges $(v_i,v_j)$ above $v_i $.  We typically depict the order from left to right. In Figure \ref{figtr1}, we order the edges as $(v_1,v_2)<(v_1,v_3)<(v_1,v_4)$. 

An isomorphism of graphs is a bijection between the vertex sets and a compatible bijection between the edge sets. An isomorphism of rooted trees is given by a bijection $f:V\to V'$ of the vertex sets such that $f(v_1)\leq f(v_2)$ if and only if $v_1\leq v_2$. An isomorphism of planar rooted trees is further required to preserve the ordering of the edges at each vertex.

We can also consider collections of rooted trees.  A \emph{rooted forest} $F$ is a disjoint union of rooted trees.  It is \emph{planar} if it consists of an ordered disjoint union of planar trees.  Again, we typically depict the order of the trees from left to right. A isomorphism of (planar) rooted forests is given by (ordered) disjoint union of isomorphisms of (planar) rooted trees.
	
Any rooted tree or forest is \emph{labelled} by a set $\Lambda$ if it is equipped with an injective function from the vertex set to $\Lambda$.  The only isomorphisms in this context are the identity functions.
	
Our aim is to associate a simplicial set to any rooted tree and show that it has the property of being a 2-Segal set.  To produce the simplicial structure, we use the notion of a cut.  For simplicity, we make our definitions for labelled planar trees, but the same construction can be applied to unlabelled planar trees or to trees that are neither labelled nor planar.

Consider $S\subseteq V(T)$, with the induced partial order. The subforest of $T$ spanned by  $S$ consists of the vertex set $S$ and includes an edge $(v_i,v_j)$ whenever $v_i, v_j \in S$ and $(v_i,v_j)$ is an edge of $T$.  We use the notation $S$ for both the subset of vertices and the subforest that it spans.

A subset $L$ of the vertex set of $T$ defines a \emph{lower subtree} if, whenever $v_j \in L$ is the upper vertex of an edge $(v_i,v_j)$ of the tree, then $v_i\in L$. Equivalently, whenever $L$ contains $v_j$ it also contains the chain $v_0<\dots<v_j$.  Lower subtrees of a tree are either empty or connected; in the latter case they contain the root.  If $L$ defines a lower subtree, then the complement $V(T)\setminus L$ is called an \emph{upper subforest} of $T$.

\begin{example}
The following labelled rooted  tree (considered as planar or not)
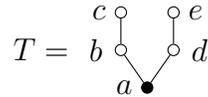
\[ T=\begin{tikzpicture}[baseline=2ex,level distance=5mm, sibling distance=7mm, grow'=up, every node/.style={circle,draw,inner sep=1.5pt}]
\node[fill=black,"$a$" left ]{} 
      child{node["$b$" left ]{} child{node["$c$" left ]{}}}
      child{node["$d$" right]{} child{node["$e$" right]{}}};
\end{tikzpicture} \]
has 10 possible lower subtrees: 
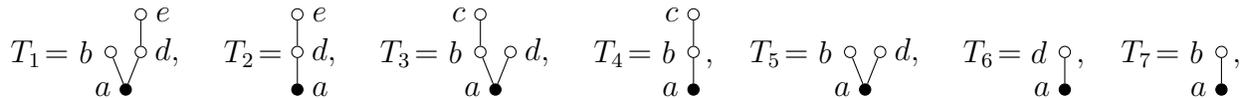
\[T_1\!=\!\!\begin{tikzpicture}[baseline=2ex,level distance=5mm, sibling distance=4mm, grow'=up, every node/.style={circle,draw,inner sep=1.5pt}]
\node[fill=black,"$a$" left ]{} 
      child{node["$b$" left ]{} }
      child{node["{$\!d,$}" right]{} child{node["$e$" right]{}}};
\end{tikzpicture}
\quad T_2\!=\begin{tikzpicture}[baseline=2ex,level distance=5mm, sibling distance=4mm, grow'=up, every node/.style={circle,draw,inner sep=1.5pt}]
\node[fill=black,"$a$" right ]{} 
      child{node["{$\!d,$}" right]{} child{node["$e$" right]{}}};
\end{tikzpicture}
\quad T_3\!=\!\!\begin{tikzpicture}[baseline=2ex,level distance=5mm, sibling distance=4mm, grow'=up, every node/.style={circle,draw,inner sep=1.5pt}]
\node[fill=black,"$a$" left ]{} 
      child{node["$b$" left ]{} child{node["$c$" left ]{}}}
      child{node["{$\!d,$}" right]{} };
\end{tikzpicture}
\quad T_4\!=\!\!\begin{tikzpicture}[baseline=2ex,level distance=5mm, sibling distance=4mm, grow'=up, every node/.style={circle,draw,inner sep=1.5pt}]
\node[fill=black,"$a$" left ]{} 
      child{node["$b$" left ]{} child{node["$c$" left ]{}}};
\end{tikzpicture}\,,
\quad T_5\!=\!\!\begin{tikzpicture}[baseline=2ex,level distance=5mm, sibling distance=4mm, grow'=up, every node/.style={circle,draw,inner sep=1.5pt}]
\node[fill=black,"$a$" left ]{} 
      child{node["$b$" left ]{} }
      child{node["{$\!d,$}" right]{} };
\end{tikzpicture}
\quad T_6\!=\!\!\begin{tikzpicture}[baseline=2ex,level distance=5mm, sibling distance=4mm, grow'=up, every node/.style={circle,draw,inner sep=1.5pt}]
\node[fill=black,"$a$" left ]{} 
      child{node["$d$" left ]{} };
\end{tikzpicture}\,,
\quad T_7\!=\!\!\begin{tikzpicture}[baseline=2ex,level distance=5mm, sibling distance=4mm, grow'=up, every node/.style={circle,draw,inner sep=1.5pt}]
\node[fill=black,"$a$" left ]{} 
      child{node["$b$" left ]{} };
\end{tikzpicture}\,, \] 
together with the trees $T_8=\bullet_a$, $T_9=T$ and $T_{10}=\varnothing$. 

If we consider $T$ as unlabelled but planar, then $T_2$ and $T_4$ are indistinguishable, in that they are isomorphic, as are $T_6$ are $T_7$.  Thus, only 8 trees occur as lower subtrees of $T$:
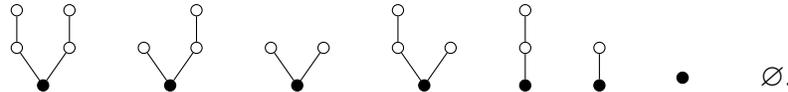
\[ \begin{tikzpicture}[baseline,level distance=5mm, sibling distance=7mm, grow'=up, every node/.style={circle,draw,inner sep=1.5pt}]
\node[fill=black]{} 
      child{node[]{} child{node[]{}}}
      child{node[]{} child{node[]{}}};
\end{tikzpicture}
\qquad
\begin{tikzpicture}[baseline,level distance=5mm, sibling distance=7mm, grow'=up, every node/.style={circle,draw,inner sep=1.5pt}]
\node[fill=black]{} 
      child{node[]{} }
      child{node[]{} child{node[]{}}};
\end{tikzpicture}
\qquad
\begin{tikzpicture}[baseline,level distance=5mm, sibling distance=7mm, grow'=up, every node/.style={circle,draw,inner sep=1.5pt}]
\node[fill=black]{} 
      child{node[]{} }
      child{node[]{} };
\end{tikzpicture}
\qquad
\begin{tikzpicture}[baseline,level distance=5mm, sibling distance=7mm, grow'=up, every node/.style={circle,draw,inner sep=1.5pt}]
\node[fill=black]{} 
      child{node[]{} child{node[]{}}}
      child{node[]{} };
\end{tikzpicture}
\qquad
\begin{tikzpicture}[baseline,level distance=5mm, sibling distance=7mm, grow'=up, every node/.style={circle,draw,inner sep=1.5pt}]
\node[fill=black]{} 
      child{node[]{} child{node[]{}}};
\end{tikzpicture}
\qquad
\begin{tikzpicture}[baseline,level distance=5mm, sibling distance=7mm, grow'=up, every node/.style={circle,draw,inner sep=1.5pt}]
\node[fill=black]{} 
      child{node[]{}};
\end{tikzpicture}
\qquad
\bullet
\qquad
\varnothing. \]
If we consider $T$ as neither labelled nor planar then $T_1$ and $T_3$ are also indistinguishable, and now only 7 trees occur as lower subtrees of $T$:
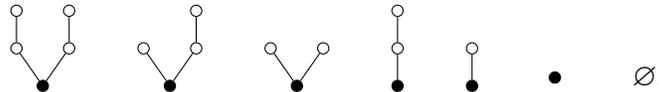
\[ \begin{tikzpicture}[baseline,level distance=5mm, sibling distance=7mm, grow'=up, every node/.style={circle,draw,inner sep=1.5pt}]
\node[fill=black]{} 
      child{node[]{} child{node[]{}}}
      child{node[]{} child{node[]{}}};
\end{tikzpicture}
\qquad
\begin{tikzpicture}[baseline,level distance=5mm, sibling distance=7mm, grow'=up, every node/.style={circle,draw,inner sep=1.5pt}]
\node[fill=black]{} 
      child{node[]{} }
      child{node[]{} child{node[]{}}};
\end{tikzpicture}
\qquad
\begin{tikzpicture}[baseline,level distance=5mm, sibling distance=7mm, grow'=up, every node/.style={circle,draw,inner sep=1.5pt}]
\node[fill=black]{} 
      child{node[]{} }
      child{node[]{} };
\end{tikzpicture}
\qquad
\begin{tikzpicture}[baseline,level distance=5mm, sibling distance=7mm, grow'=up, every node/.style={circle,draw,inner sep=1.5pt}]
\node[fill=black]{} 
      child{node[]{} child{node[]{}}};
\end{tikzpicture}
\qquad
\begin{tikzpicture}[baseline,level distance=5mm, sibling distance=7mm, grow'=up, every node/.style={circle,draw,inner sep=1.5pt}]
\node[fill=black]{} 
      child{node[]{}};
\end{tikzpicture}
\qquad
\bullet
\qquad
\varnothing. \]
\end{example}
	
\begin{definition} 
    Let $T$ be a rooted tree. An \emph{admissible cut}, or simply a \emph{cut}, on $T$ is a partition of the vertices $V(T)= L \union U$  such that $L$ defines a lower subtree of $T$.
\end{definition}

Observe that this process determines a partition of the edge set $E(T)$ into three parts: the edges of $L$, the edges of $U$, and the edges of the form $(v,w)$ with $v\in L$ and $w\in U$.  Indeed, any of these subsets uniquely determines the cut.

\begin{definition}
    A \emph{layering of $n-1$ cuts} of $T$ is a sequence of partially ordered subsets 
    \[V(T)=L_0\supseteq L_1 \supseteq \cdots\supseteq L_n=\varnothing,\] where each $L_k$ defines a lower subtree of $T$. 
\end{definition}

A layering defines a sequence of $n-1$ cuts ordered from leaves to the root.

\begin{example}
	For $n=0$ we have, by convention, a unique `layering of $-1$ cuts' $V(T) = L_0 = \varnothing$ of the empty tree $T=\varnothing$.
		
	For $n=1$ any tree $T$ has a unique `layering of 0 cuts' $V(T)=L_0 \supseteq L_1=\varnothing$.
 
	For $n=2$ there is a layering $V(T) \supseteq L\supseteq\varnothing$ associated to any cut of $T$, where $L$ defines the lower subtree of the cut.

    For $n=3$, consider the tree $T$ with two cuts as given on the left-hand side of Figure \ref{figtr2}. Then the associated chain is $V(T)= L_0\supseteq L_1\supseteq L_2\supseteq L_3=\varnothing$, where $L_1=\{e,d,f,g,h\}$ and $L_2=\{h\}$.
    \addtocounter{equation}{1}
	{\scriptsize
		\begin{figure}
			\begin{tikzpicture}[scale=.60,line cap=round,line join=round,>=triangle 45,x=1cm,y=1cm, dot/.style = {circle, gray, minimum size=#1,
					inner sep=0pt, outer sep=0pt}, 
				dot/.default = 3pt]
				\draw [-] (0,0)-- (1,1);
				\draw [] (0,0)-- (-1,1);
				\draw [] (0,-1)-- (0,0);
				\draw [] (-1,1)-- (1,3);
				\draw [] (0,2)-- (-1,3);
				\draw [] (-1,1)-- (-2,2);
				\draw [-, dashed] (-2.12,1.43)-- (1.20,2.8);
				\draw [-, dashed] (-1.05,-.5)-- (1.20,-.5);
				\draw [fill=black] (0,-1) circle (3pt);
				\draw[color=black] (-.4,-.8) node {$h$};
				\draw [fill=white] (-1,1) circle (3pt);
				\draw[color=black] (-1,1.43) node {$e$};
				\draw [fill=white] (1,1) circle (3pt);
				\draw[color=black] (1.16,1.43) node {$f$};
				\draw [fill=white] (0,0) circle (3pt);
				\draw[color=black] (0.3,-0.24) node {$g$};
				\draw [fill=white] (0,2) circle (3pt);
				\draw [fill=white] (-2,2) circle (3pt);
				\draw[color=black] (-1.90,2.43) node {$a$};
				\draw [fill=white] (-1,3) circle (3pt);
				\draw[color=black] (-0.84,3.43) node {$b$};
				\draw [fill=white] (1,3) circle (3pt);
				\draw[color=black] (1.16,3.43) node {$c$};
				\draw[color=black] (0.3,2.01) node {$d$};
			\end{tikzpicture}
			\hspace{2.5cm}
			\begin{tikzpicture}[scale=.60,line cap=round,line join=round,>=triangle 45,x=1cm,y=1cm, dot/.style = {circle, gray, minimum size=#1,
					inner sep=0pt, outer sep=0pt}, 
				dot/.default = 3pt]
				\draw [] (-2,-1)-- (-2,0);
				\draw [] (-2,0)-- (-1,1);
				\draw [] (-2,0)-- (-3,1);
				\draw [] (-3,1)-- (-2,2);
				\draw [-, dashed] (-3.05,-.6)-- (-1.20,-.6);
				\draw [fill=white] (-2,2) circle (3pt);
				\draw [fill=white] (-2,0) circle (3pt);
				\draw[color=black] (-1.70,-.20) node {$g$};
				\draw [fill=white] (-3,1) circle (3pt);
				\draw[color=black] (-1.7,2.1) node {$d$};
				\draw[color=black] (-2.98,1.43) node {$e$};
				\draw [fill=white] (-1,1) circle (3pt);
				\draw[color=black] (-0.84,1.43) node {$f$};
				\draw [fill=black] (-2,-1) circle (3pt);
				\draw[color=black] (-2.3,-1) node {$h$};
			\end{tikzpicture}
			\caption{An example of a tree $T$ with two cuts (left), and its image under the map $d_0$ (right).}
			\label{figtr2}
		\end{figure}
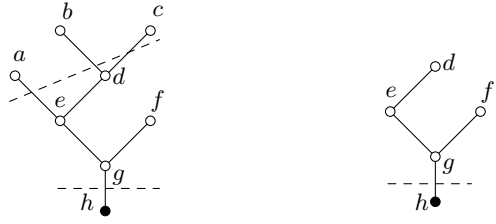 } 
\end{example}

We now generalize some of the previous definitions to forests.

\begin{definition}
A \emph{rooted forest} $F$ is a disjoint union of rooted trees $F=\bigunion T_{\alpha}$.
\begin{itemize}
\item  A subset $L=\bigunion L_{\alpha}$ of the vertex set defines a \emph{lower subforest} of $F$ if each $L_{\alpha}$ defines a lower subtree of $T_\alpha$.

\item Its complement $V(F)\setminus L$ defines an upper subforest $U=\bigunion U_{\alpha}$ of $F$. 

\item An \emph{(admissible) cut} of $F$ is a partition of the vertices
$L\union U$ such that $L$ defines a lower subforest of $F$. 

\item A \emph{layering of $n-1$} cuts of a rooted forest $F$ is a sequence of nested subsets $V(L)=L_0\supseteq L_1 \supseteq \cdots \supseteq L_n=\varnothing$, where each $L_k$ defines a lower subforest of $F$.
\end{itemize}

The subforest defined by $L_{i-1}\setminus L_{i}$ is referred to as the \emph{$i$-th layer} of the layering and the subforests defined by $L_i\setminus L_j$ for $0\leq i\leq j\leq n$ are \emph{admissible subforests}.
\end{definition}

Admissible subforests are sometimes also called \emph{subforests obtained via iterated cuts} or \emph{convex subforests} \cite[Example 7.12(1)]{restr}.

\begin{example}\label{convex}
Consider the labelled trees with three vertices, with root vertex of degree 1 or degree 2. Then the nonempty admissible subforests for these two examples are as follows:
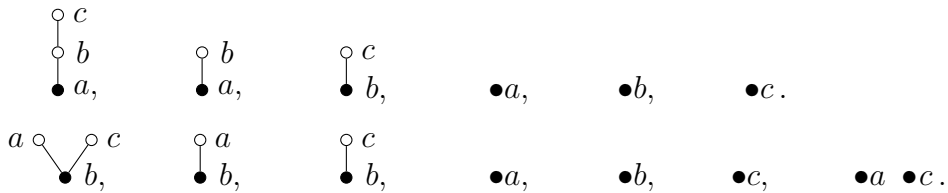
\[ \begin{array}{ccccccc}\qquad
\begin{tikzpicture}[baseline=-1mm,level distance=5mm, grow'=up, every node/.style={circle,draw,inner sep=1.5pt}]
\node[fill=black,"{$a$,}" right ]{} 
      child{node["$b$" right ]{} child{node["$c$" right ]{}}};
\end{tikzpicture}
\quad&\quad
\begin{tikzpicture}[baseline=-1mm,level distance=5mm, grow'=up, every node/.style={circle,draw,inner sep=1.5pt}]
\node[fill=black,"{$a$,}" right ]{} child{node["$b$" right ]{} };
\end{tikzpicture}
\quad&\quad
\begin{tikzpicture}[baseline=-1mm,level distance=5mm, grow'=up, every node/.style={circle,draw,inner sep=1.5pt}]
\node[fill=black,"{$b$,}" right ]{} child{node["$c$" right ]{} };
\end{tikzpicture}
\quad&\quad\bullet a,\quad&\quad\bullet b,\quad&\quad\bullet c\,.&\\[2mm]
\begin{tikzpicture}[baseline=-1mm,level distance=5mm, sibling distance=7mm, grow'=up, every node/.style={circle,draw,inner sep=1.5pt}]
\node[fill=black,"{$b$,}" right ]{} child{node["$a$" left ]{} }
                                    child{node["$c$" right]{} };
\end{tikzpicture}
&
\begin{tikzpicture}[baseline=-1mm,level distance=5mm, grow'=up, every node/.style={circle,draw,inner sep=1.5pt}]
\node[fill=black,"{$b$,}" right ]{} child{node["$a$" right ]{} };
\end{tikzpicture}
&
\begin{tikzpicture}[baseline=-1mm,level distance=5mm, grow'=up, every node/.style={circle,draw,inner sep=1.5pt}]
\node[fill=black,"{$b$,}" right ]{} child{node["$c$" right ]{} };
\end{tikzpicture}
&\bullet a,&\bullet b,&\bullet c,&\quad\;\bullet a\;\bullet\! c\,.
\end{array} \]
\end{example}

\begin{definition}\label{simpdef} 
	Let $(T,v_0)$ be a labelled rooted tree. We define a simplicial set $X^T$ associated to $T$ as follows.
	\begin{enumerate}
		\item The set $X^T_0$ contains just the `layering of $-1$ cuts' of the empty tree.
  
		\item The set $X^T_1$ is the set of admissible subforests of $T$. 
  
        \item For $n \geq 2$, the set $X^T_n$ is the set of all admissible subforests  $H\in X^T_1$ together with a layering of $n-1$ cuts $V(H)\supseteq L_1\supseteq\dots\supseteq L_n=\varnothing$.  Therefore, an element of $X^T_n$ can be considered as a sequence of nested subforests $H\supseteq L_1\supseteq\dots\supseteq L_n=\varnothing$.
			
        \item The face maps $d_i \colon X^T_n\to X^T_{n-1}$ are defined as follows. 
			\begin{enumerate}
				\item The map $d_0$ removes the first cut and the top layer, replacing the subforest $H$ by its restriction $H|L_1$ to the vertices $L_1$: 
                \[ d_0(H \supseteq \dots \supseteq L_n)=(L_1 \supseteq \dots \supseteq L_n).\]
				
                \item The map $d_n$ removes the last cut and the bottom layer:
				\[ d_n(H \supseteq \dots \supseteq L_n)=(H\setminus L_{n-1}\supseteq \dots \supseteq L_{n-1}\setminus L_{n-1}).\]
    
				\item For $0<i<n$ the map $d_i$ removes the $i$-th cut and merges the $i$-th and $(i+1)$-st layer:
                \[d_i(H\supseteq\dots\supseteq L_n)=(H\supseteq \dots \supseteq L_{i-1}\supseteq L_{i+1}\supseteq\dots \supseteq L_{n}).\]
			\end{enumerate}
		
        \item The degeneracy map $s_i \colon X^T_{n}\to X^T_{n+1}$ repeats the $i$-th cut, adding an empty $i$-th layer:
        \[s_i(H \supseteq \dots\supseteq L_n)=(H\supseteq \dots \supseteq L_{i}\supseteq L_i\supseteq \dots \supseteq L_{n}).\]
	\end{enumerate}
\end{definition}

We leave it to the reader to verify that these face and degeneracy maps satisfy the simplicial relations.  Observe that the modification for the face map $d_n$ is necessary to ensure that the last subset in the layering is $\varnothing$, as required.  Example \ref{convex} describes the set $X_1^T$ for two particular examples, namely the labelled trees on 3 vertices, and Figure \ref{figtr2} depicts the action of the face map $d_0$ on a particular element of $X_3^T$.
	
\begin{example}
    For any admissible subforest $H$, when $n=2$ we see that $d_0(H\supseteq L_1 \supseteq \varnothing)$ is just the lower subforest of the cut, $d_2(H\supseteq L_1\supseteq\varnothing)$ is just the upper subforest of the cut, and $d_1(H\supseteq L_1\supseteq\varnothing)$ recovers the subforest $H$ but forgets the cut.
\end{example}

\begin{definition}\label{unlabelled-X^T}
Let $T$ be an unlabelled (planar or non-planar) rooted tree.  Two layerings of $n-1$ cuts of admissible subforests of $T$ are \emph{isomorphic},
\[(H=L_0\supseteq L_1\supseteq\dots\supseteq L_n=\varnothing) \cong (H'=L_0'\supseteq L_1'\supseteq\dots\supseteq L_n'=\varnothing), \]
if there is an isomorphism $f\colon H\stackrel\cong\to H'$ of (planar or non-planar) forests inducing bijections $L_i\cong L'_i$ for each $i$.

The simplicial set $X^T$ associated to $T$ consists of the sets $X_n^T$ of isomorphism classes of layerings of $n-1$ cuts of admissible subforests  of $T$, with the simplicial face and degeneracy maps induced by those of Definition \ref{simpdef}.
\end{definition}

The simplicial sets described in Definitions \ref{simpdef} and \ref{unlabelled-X^T} are $2$-Segal but not $1$-Segal in general, as we now show.

\begin{prop}
If $T$ is a tree, which may be labelled or unlabelled, planar or non-planar, then $X^T$ is a 2-Segal set.  
\end{prop}

\begin{proof}
By \cite[Proposition 2.3.2(4)]{dk} it is enough to check that the following commutative diagrams are pullback squares, for each $n\geq3$ and each $0<i<n-1$:
\[ \begin{tikzcd}
    X^T_n\ar[rr,"d_1\dots d_i"] \ar[d,"d_{i+2}\dots d_{n  }"'] && X^T_{n-i}
   \ar[d,"d_{  2}\dots d_{n-i}" ] &&&
    X^T_n\ar[rr,"d_{i+1}\cdots d_{n-1}"] \ar[d,"d_{0}\dots d_{i-1}"'] && X^T_{i+1}\ar[d,"d_{0}\dots d_{i-1}" ] \\      
    X^T_{i+1}\ar[rr,"d_1\dots d_i"] && X^T_1 &&& X^T_{n-i}\ar[rr,"d_1\dots d_{n-i-1}"]&&X_1^T.
\end{tikzcd} \]
For example, the case $n=3$, $i=1$ says that the 2-Segal maps $X_3^T \to X_2^T \times_{X_1^T}X_2^T$ induced by the triangulations of Example \ref{squares} are isomorphisms, as illustrated in Figure \ref{2segfig}. 
\begin{figure}[H]
    \centering
   \begin{tikzpicture}[scale=1.3,thick]
		\tikzstyle{every node}=[minimum width=0pt, inner sep=2pt, circle]
			\draw (-4.2,2.55) node[draw] (0) {};
			\draw (-3.89,2.2) node[draw] (1) {};
			\draw (-3.57,2.59) node[draw] (2) {};
			\draw (-4.76,2.29) node[draw] (3) {};
			\draw (-4.4,1.47) node[draw] (4) {};
			\draw (-4.4,1.86) node[draw] (5) {};
			\draw (-4.86,2.01) node (6) {};
			\draw (-3.86,1.99) node (7) {};
			\draw (-4.08,0.98) node[draw] (8) {};
			\draw (-3.7,1.45) node[draw] (9) {};
			\draw (-4.58,1.16) node (10) {};
			\draw (-3.59,1.15) node (11) {};
			\draw (-4.1,0.46) node[draw, fill=black] (12) {};
			\draw (-4.51,0.3) node (13) {};
			\draw (-4.51,-0.4) node (14) {};
			\draw (-4.36,-0.76) node[draw] (15) {};
			\draw (-4.06,-1.11) node[draw] (16) {};
			\draw (-3.79,-0.74) node[draw] (17) {};
			\draw (-4.82,-1.17) node[draw] (18) {};
			\draw (-4.5,-1.91) node[draw, fill=black] (19) {};
			\draw (-4.5,-1.51) node[draw] (20) {};
			\draw (-5,-1.4) node (21) {};
			\draw (-3.8,-1.4) node (22) {};
			\draw (-3.8,-1.9) node[draw, fill=black]{};
            \draw (-4.7,1) node (X3){$X_3$};
            \draw (-4.7,-0.8) node (X2){$X_2$};
            \draw (-4.3,0) node (d3){$d_3$};
            \draw (-3.3,2) node (2d){};
            \draw (-2.3,2) node (1d){};
            \draw (-2.8,2.15) node (d1){$d_1$};
            \draw (-3.3,-1.5) node (3d){};
            \draw (-2.3,-1.5) node (4d){};
            \draw (-2.8,-1.7) node (d1){$d_1$};
            
			\draw  (1) edge (2);
			\draw  (0) edge (1);
			\draw  (3) edge (4);
			\draw  (1) edge (4);
			\draw  (4) edge (5);
			\draw[dashed]  (6) edge (7);
			\draw  (4) edge (8);
			\draw  (8) edge (9);
			\draw[dashed]  (10) edge (11);
			\draw  (8) edge (12);
			\draw  (16) edge (17);
			\draw  (15) edge (16);
			\draw  (16) edge (19);
			\draw  (19) edge (20);
			\draw  (18) edge (19);
			\draw[dashed]  (21) edge (22);
			\draw[->]  (13) edge (14);
            \draw[->] (2d) edge (1d);
            \draw[->] (3d) edge (4d);
		\tikzstyle{every node}=[minimum width=0pt, inner sep=2pt, circle]
			\draw (-1.2,2.55) node[draw] (0x) {};
			\draw (-0.89,2.2) node[draw] (1x) {};
			\draw (-0.57,2.59) node[draw] (2x) {};
			\draw (-1.76,2.29) node[draw] (3x) {};
			\draw (-1.4,1.47) node[draw] (4x) {};
			\draw (-1.4,1.86) node[draw] (5x) {};
			\draw (-1.86,2.01) node (6x) {};
			\draw (-0.86,1.99) node (7x) {};
			\draw (-1.08,0.98) node[draw] (8x) {};
			\draw (-0.7,1.45) node[draw] (9x) {};
			\draw (-1.58,1.16) node (10x) {};
			\draw (-0.59,1.15) node (11x) {};
			\draw (-1.1,0.46) node[draw, fill=black] (12x) {};
			\draw (-1.51,0.3) node (13x) {};
			\draw (-1.51,-0.4) node (14x) {};
			\draw (-1.36,-0.76) node[draw] (15x) {};
			\draw (-1.06,-1.11) node[draw] (16x) {};
			\draw (-0.79,-0.74) node[draw] (17x) {};
			\draw (-1.82,-1.17) node[draw] (18x) {};
			\draw (-1.5,-1.91) node[draw, fill=black] (19x) {};
			\draw (-1.5,-1.51) node[draw] (20x) {};
			\draw (-2,-1.38) node (21x) {};
			\draw (-0.8,-1.42) node (22x) {};
			\draw (-0.8,-1.9) node[draw, fill=black]{};
            \draw (-1.7,1) node (X3){$X_2$};
            \draw (-1.7,-0.8) node (X2){$X_1$};
            \draw (-1.3,0) node (d2){$d_2$};
			\draw  (1x) edge (2x);
			\draw  (0x) edge (1x);
			\draw  (3x) edge (4x);
			\draw  (1x) edge (4x);
			\draw  (4x) edge (5x);		
			\draw  (4x) edge (8x);
			\draw  (8x) edge (9x);
			\draw  (8x) edge (12x);
			\draw  (16x) edge (17x);
			\draw  (15x) edge (16x);
			\draw  (16x) edge (19x);
			\draw  (19x) edge (20x);
			\draw  (18x) edge (19x);
            \draw[dashed]  (10x) edge (11x);
			\draw[->]  (13x) edge (14x);

\tikzstyle{every node}=[minimum width=0pt, inner sep=2pt, circle]
			\draw (-4.2+6,2.55) node[draw] (0y) {};
			\draw (-3.89+6,2.2) node[draw] (1y) {};
			\draw (-3.57+6,2.59) node[draw] (2y) {};
			\draw (-4.76+6,2.29) node[draw] (3y) {};
			\draw (-4.4+6,1.47) node[draw] (4y) {};
			\draw (-4.4+6,1.86) node[draw] (5y) {};
			\draw (-4.86+6,2.01) node (6y) {};
			\draw (-3.86+6,1.99) node (7y) {};
			\draw (-4.08+6,0.98) node[draw] (8y) {};
			\draw (-3.7+6,1.45) node[draw] (9y) {};
			\draw (-4.58+6,1.16) node (10y) {};
			\draw (-3.59+6,1.15) node (11y) {};
			\draw (-4.1+6,0.46) node[draw, fill=black] (12y) {};
			\draw (-4.5+6,0.3) node (13y) {};
			\draw (-4.5+6,-0.4) node (14y) {};
			\draw (-4.06+6,-1.11) node[draw] (16y) {};
			\draw (-4.82+6,-1.17) node[draw] (18y) {};
            \draw (-4.82+6,-0.7) node[draw] (17y) {};
			\draw (-4.5+6,-1.61) node[draw] (19y) {};
			\draw (-4.5+6,-2.11) node[draw, fill=black] (20y) {};
			\draw (-5+6,-1.4) node (21y) {};
			\draw (-3.8+6,-1.4) node (22y) {};
            \draw (-4.7+6,1) node (X3y){$X_3$};
            \draw (-4.7+7,-0.8) node (X2y){$X_2$};
            \draw (-4.3+6,0) node (d3y){$d_0$};
            \draw (-3.3+6,2) node (2dy){};
            \draw (-2.3+6,2) node (1dy){};
            \draw (-2.8+6,2.15) node (d1y){$d_2$};
            \draw (-3.3+6,-1.5) node (3dy){};
            \draw (-2.3+6,-1.5) node (4dy){};
            \draw (-2.8+6,-1.7) node (d1y){$d_1$};

            \draw (17y) edge (18y);
			\draw  (1y) edge (2y);
			\draw  (0y) edge (1y);
			\draw  (3y) edge (4y);
			\draw  (1y) edge (4y);
			\draw  (4y) edge (5y);
			\draw  (4y) edge (8y);
			\draw  (8y) edge (9y);
			\draw  (8y) edge (12y);
			\draw  (16y) edge (19y);
			\draw  (19y) edge (20y);
			\draw  (18y) edge (19y);
            \draw[dashed]  (6y) edge (7y);
            \draw[dashed]  (10y) edge (11y);
			\draw[dashed]  (21y) edge (22y);
			\draw[->] (13y) edge (14y);
            \draw[->] (2dy) edge (1dy);
            \draw[->] (3dy) edge (4dy);
		
			\draw (-1.2+6,2.55) node[draw]  (0xy) {};
			\draw (-0.89+6,2.2) node[draw]  (1xy) {};
			\draw (-0.57+6,2.59) node[draw] (2xy) {};
			\draw (-1.76+6,2.29) node[draw] (3xy) {};
			\draw (-1.4+6,1.47) node[draw] (4xy) {};
			\draw (-1.4+6,1.86) node[draw] (5xy) {};
			\draw (-1.86+6,1.99) node (7xy) {};
			\draw (-1.08+6,0.98) node[draw](8xy) {};
			\draw (-0.7+6,1.45) node[draw] (9xy) {};
			\draw (-1.58+6,1.16) node (10xy) {};
			\draw (-0.59+6,1.15) node (11xy) {};
			\draw (-1.1+6,0.46) node[draw, fill=black] (12xy) {};
			\draw (-1.51+6,0.3) node  (13xy) {};
			\draw (-1.51+6,-0.4) node (14xy) {};
			\draw (-1.06+6,-1.11) node[draw] (16xy) {};
			\draw (-1.82+6,-0.7) node[draw] (17xy) {};
			\draw (-1.82+6,-1.17) node[draw] (18xy) {};
			\draw (-1.5+6,-1.61) node[draw] (19xy) {};
			\draw (-1.5+6,-2.11) node[draw, fill=black] (20xy) {};
			\draw (-2+6,-1.38) node  (21xy) {};
			\draw (-0.8+6,-1.42) node (22xy) {};
            \draw (-1.7+6,1) node (X3y){$X_2$};
            \draw (-1.7+7,-0.8) node (X2y){$X_1$};
            \draw (-1.3+6,0) node (d2y){$d_0$};
			\draw  (1xy) edge (2xy);
			\draw  (0xy) edge (1xy);
			\draw  (3xy) edge (4xy);
			\draw  (1xy) edge (4xy);
			\draw  (4xy) edge (5xy);		
			\draw  (4xy) edge (8xy);
			\draw  (8xy) edge (9xy);
			\draw  (8xy) edge (12xy);
			\draw (17xy) edge (18xy);
			\draw  (16xy) edge (19xy);
			\draw  (19xy) edge (20xy);
			\draw  (18xy) edge (19xy);
            \draw[dashed]  (10xy) edge (11xy);
			\draw[->]  (13xy) edge (14xy);            
		\end{tikzpicture}
    \caption{Example of the 2-Segal property. To recover the trees with two cuts in the upper left corners, observe that the trees in the top right corners contain all the necessary data except one cut. The information of this missing cut is present in the bottom left corners. Their intersection, the tree in the bottom right corner, informs us how to glue these two trees, hence determines the missing cut.}
    \label{2segfig}
\end{figure}
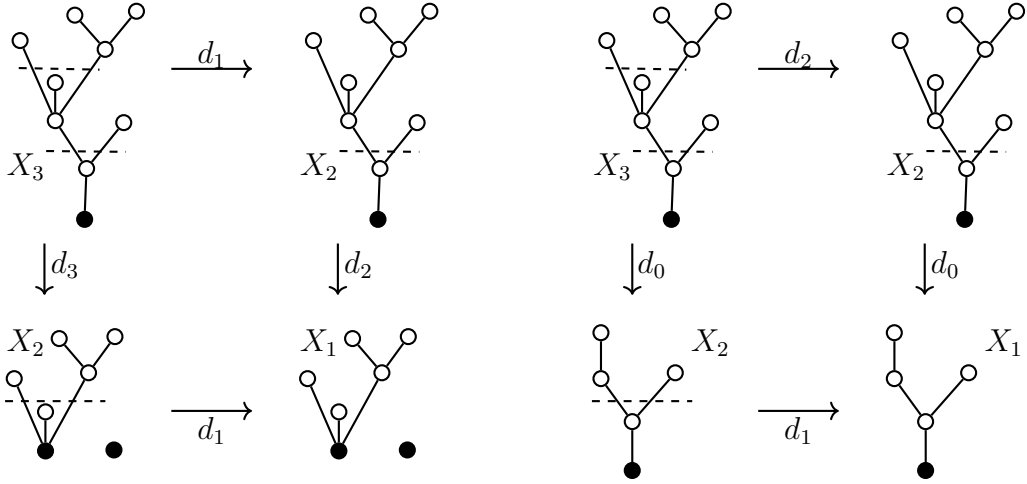

For the first diagram, the pullback is the set of triples
\[ \left( (H'\supseteq L_1'\supseteq\dots\supseteq L_{i+1}'=\varnothing), (H'''\supseteq\varnothing), (H''\supseteq L_1''\supseteq\dots\supseteq L_{n-i}''=\varnothing) \right)\in 
X^T_{i+1}\times X_1^T\times X^T_{n-i} \]
such that $H'=H'''=H''\setminus L_1''$. We can therefore construct an element
\[ (H''\supseteq L_1' \! \sqcup\!L_1''\supseteq\dots\supseteq L_i'\!\sqcup\!L_1'' \; \supseteq\; L_1''\supseteq\dots\supseteq L_{n-i}''=\varnothing) \in X^T_n,  \]
providing the inverse of the canonical map from $X^T_n$ to the pullback.
For the second diagram the elements of the pullback are the triples
\[ \left( (H'\supseteq L_1'\supseteq\dots\supseteq L_{n-i}'=\varnothing), (H'''\supseteq\varnothing), (H''\supseteq L_1''\supseteq\dots\supseteq L_{i+1}''=\varnothing) \right)\in 
X^T_{n-i}\times X_1^T\times X^T_{i+1} \]
such that $H'=H'''=H''\backslash L_i''$.  Again, the map sending such an element to
\[ (H''\supseteq L_1''\supseteq\dots\supseteq L_i'' \;\supseteq\; L_1' \supseteq \dots \supseteq L_{n-i}'=\varnothing) \in X^T_n \]
provides the inverse to the canonical map $X^T_n\to X^T_{n-i}\times _{X_1^T} X^T_{i+1}$.
\end{proof}

The simplicial set $X^T$ is never $1$-Segal, unless $T$ is the empty tree. Indeed, if $T \neq\varnothing$ then the Segal map
\[ (d_0,d_2) \colon X_2^T\to X_1^T\times X_1^T \]
is not surjective as $(T,T)$ is not in the image.

In the unlabelled case the Segal map can also fail to be injective, as depicted in Figure \ref{fignot1seg}.
\begin{figure}[H]
    \centering
    \begin{tikzpicture}[scale=1.5,thick]
		\tikzstyle{every node}=[minimum width=0pt, inner sep=2pt, circle]
			\draw (-2.95,2.44) node[draw] (0) { };
			\draw (-2.67,2.06) node[draw] (1) {};
			\draw (-2.39,2.51) node[draw] (2) {};
			\draw (-3.28,2.1) node[draw] (3) {};
			\draw (-2.97,1.5) node[draw, fill=black] (4) {};
			\draw (-2.97,1.8) node[draw] (5) {};
			\draw (-3.5,1.9) node (6) {};
			\draw (-2.5,1.9) node (7) {};
			\draw (-2.6,1.5) node[draw, fill=black] (8) {};
			\draw (-2.41,1.34) node (9) {};
			\draw (-3.73,1.77) node (10) {$X_2$};
			\draw (-1.8,0.42) node (11) {$X_1$};
            \draw (-3.8,1.2) node {$d_2$};
            \draw (-2.1,1.2) node {$d_0$};
			\draw (-1.94,0.67) node (12) {};
			\draw (-3.39,1.36) node (13) {};
			\draw (-3.93,0.7) node (14) {};
			\draw (-4.1,0.5) node[draw] (15) {};
			\draw (-3.78,-0.1) node[draw, fill=black] (16) {};
			\draw (-3.45,0.5) node[draw] (17) {};
			\draw (-4.1,-0.1) node (18) {$X_1$};
			\draw (-2.27,-0.1) node[draw, fill=black] (19) {};
			\draw (-2.25,0.5) node[draw] (20) {};
			\draw (-1.84,-0.1) node[draw, fill=black] (21) {};
			\draw (-3.3,-0.1) node[draw, fill=black] (22) {};
			\draw  (0) edge (1);
			\draw  (3) edge (4);
			\draw  (1) edge (4);
			\draw  (4) edge (5);
			\draw[dashed]  (6) edge (7);
			\draw  (16) edge (17);
			\draw  (15) edge (16);
			\draw  (19) edge (20);
			\draw[->]  (13) edge (14);
			\draw  (1) edge (2);
			\draw[->]  (9) edge (12);

            \draw (-2.95+4,2.44) node[draw] (0x) { };
			\draw (-2.67+4,2.1) node[draw] (1x) {};
			\draw (-2.39+4,2.51) node[draw] (2x) {};
			\draw (-3.28+4,2.1) node[draw] (3x) {};
			\draw (-3.28+4,1.5) node[draw, fill=black] (4x) {};
			\draw (-2.97+4,1.8) node[draw] (5x) {};
			\draw (-3.5+4,1.9) node (6x) {};
			\draw (-2.2+4,1.9) node (7x) {};
			\draw (-2.67+4,1.5) node[draw, fill=black] (8x) {};
			\draw (-2.41+4,1.34) node (9x) {};
			\draw (-3.73+4,1.77) node (10x) {$X_2$};
			\draw (-1.8+4,0.42) node (11x) {$X_1$};
            \draw (-3.8+4,1.2) node {$d_2$};
            \draw (-2.1+4,1.2) node {$d_0$};
			\draw (-1.94+4,0.67) node (12x) {};
			\draw (-3.39+4,1.36) node (13x) {};
			\draw (-3.93+4,0.7) node (14x) {};
			\draw (-4.1+4,0.5) node[draw] (15x) {};
			\draw (-3.78+4,-0.1) node[draw, fill=black] (16x) {};
			\draw (-3.45+4,0.5) node[draw] (17x) {};
			\draw (-4.1+4,-0.1) node (18) {$X_1$};
			\draw (-2.27+4,-0.1) node[draw, fill=black] (19x) {};
			\draw (-2.25+4,0.5) node[draw] (20x) {};
			\draw (-1.84+4,-0.1) node[draw, fill=black] (21x) {};
			\draw (-3.3+4,-0.1) node[draw, fill=black] (22x) {};
			\draw  (0x) edge (1x);
			\draw  (3x) edge (4x);
			\draw  (1x) edge (8x);
			\draw  (4x) edge (5x);
			\draw  (1x) edge (2x);
			\draw  (16x) edge (17x);
			\draw  (15x) edge (16x);
			\draw  (19x) edge (20x);
			\draw[->] (13x) edge (14x);
			\draw[->]  (9x) edge (12x);
			\draw[dashed](6x) edge (7x);
		\end{tikzpicture}
    \caption{The failure of injectivity for the 1-Segal condition in the unlabelled case.} \label{fignot1seg}
\end{figure}
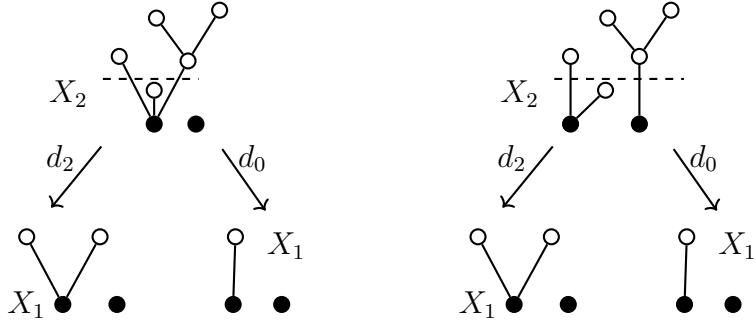

However, in the labelled case, injectivity always holds.

\begin{prop}\label{prop:comp} 
If $T$ is a labelled tree with vertex set $V$ then the Segal map
\[ (d_0,d_2)\colon X_2^T\to X_1^T\times X_1^T \]
is injective. Its fibre over an arbitrary element $(H_1,H_2)\in X_1^T\times X_1^T$ is either empty or consists of a single element $(H\supseteq V(H_1)\supseteq\varnothing)$, where the vertex set of $H$ is the disjoint union of the vertex sets of $H_1$ and $H_2$ in $V$.
\end{prop}

\begin{proof}
    In the labelled case, the vertex sets $L$ and $U$ of $H_1$ and $H_2$ are specific subsets of $V$, and if there is an element in the fibre it is given by $(H\supseteq L\supseteq\varnothing)$ with $V(H)\setminus L=U$.  There is no such element if $L$, $U$ are not disjoint, if their union does not define an admissible subforest $H$ of $T$, or if $L$ does not define a lower subforest of $H$. 
\end{proof}

A natural question is whether we can identify which subforests of $T$ are obtained from admissible cuts, and hence define elements of $X_1$.  For example, for a rooted tree with at least one edge, we do not obtain a subforest consisting of both the root and any other vertex of $T$ but no edges.  We do have the following result for trees, however.

\begin{prop} \label{all} 
    Let $(T, v_0)$ be a rooted tree and $S$ a subtree of $T$. Then $S$ can be obtained from $T$ by a sequence of admissible cuts.
\end{prop}

\begin{proof}
    Note that any single edge $e$ of $T$ determines an admissible cut, since $T\setminus e$ always has two connected components and, in particular, one of them must contain the root. We consider the case $v_0\not\in S$; the other case is similar. Let $w_0\in S$ be the closest vertex to the root $v_0$. Given a vertex $v$, denote by $e_r(v)=(u,v)$ the edge that lies on the unique path between $v$ and the root, and has $v$ as one of its endpoints. By cutting $T$ along the edge $e_r(w_0)$, we obtain two trees, $T_1$ and $T_2$ where $S \subseteq T_1$ and $v_0\in T_2$.  Given a subtree $H$, let $d_H(w)$ denote the degree of the vertex $w$ in $H$. If $d_S(w_0) \neq d_{T_1}(w_0)$, then we may successively carry out cuts along the set of edges $\{(w_0,v) \mid v \notin V(S)\}$. Denote the resulting subtree containing $w_0$ by $T_3$. Let $w_1, \dots, w_m$ be the vertices of $S$ that are connected by edges to $T_3\setminus S$, and let $c(w_i)=\{(w_i,v)\in E(T_3) \mid v \notin V(S)\}$ for $1 \leq i \leq m$.  Then the set of edges $c(w_i)$ determines an admissible cut $c_i$ of $T_3$ and, after taking the cuts $c_1, \dots, c_m$ we obtain the subtree $S$. 
\end{proof}

\begin{remark}\label{rem1}  
    The only tree that produces only subtrees, rather than subforests, from cuts is the linear tree $P_n$, in which the root and one other vertex have degree one, while the rest have degree two.  The only way to obtain a disconnected top layer is by taking a cut through edges with the same lower vertex, or through parallel edges with incomparable lower vertices, and neither case is possible for a linear tree.
\end{remark}

\section{The associated pointed stable double category} \label{double}

In this section, we begin by recalling the relationship between 2-Segal sets and augmented stable double categories.  We then consider certain families of 2-Segal sets associated to rooted trees and give explicit descriptions of their associated double categories.

Let us begin by recalling the definition of a double category.

\begin{definition}
    A (small) \emph{double category} is an internal category in the category of small categories, and a \emph{double functor} between double categories is an internal functor.
\end{definition}

In other words, a small double category $\mathcal{D}$ consists of a set of objects $\ob(\mathcal{D})$, a set of horizontal morphisms $\hor(\mathcal{D})$, a set of vertical morphisms $\ver(\mathcal{D})$, and a set of squares $\sq(\mathcal{D})$, and these sets are related by various source, target, identity and composition maps. In particular, $\Hor\mathcal{D} :=(\ob(\mathcal{D}), \hor(\mathcal{D))}$ and $\Ver\mathcal{D}:= (\ob(\mathcal{D}), \ver(\mathcal{D}))$ form categories. We typically denote horizontal morphisms by $\rightarrowtail$ and vertical morphisms by $\twoheadrightarrow$. 

As described in \cite{2s}, there is an equivalence of categories between 2-Segal sets and a particular subcategory of the category of small double categories.  Let us now set up the definitions we need to describe this subcategory.

\begin{definition}
    A double category $\mathcal{D}$ is \emph{pointed} if it is equipped with a distinguished object $\star$ that is initial in $\Hor\mathcal{D}$ and terminal in $\Ver\mathcal{D}$. A double functor $F \colon \mathcal{D}\rightarrow \mathcal{D'}$ between pointed double categories is \emph{pointed} if it sends the distinguished object of $\mathcal{D}$ to the distinguished object of $\mathcal{D}'$.
\end{definition}

Although most of the examples that we consider in this paper are pointed, we include the following more general definition.

\begin{definition}
    An \emph{augmentation} of a double category $\mathcal{D}$ consists of a set of objects $A$ such that, for every object $d$ of $\mathcal{D}$, there exist unique morphisms $a \rightarrowtail d$ in $\Hor\mathcal{D}$ and $d \twoheadrightarrow a'$ in $\Ver\mathcal{D}$ with $a, a' \in A.$  An \emph{augmented double category} is a double category equipped with an augmentation, and a double functor between augmented double categories is \emph{augmented} if it preserves the augmentation.
\end{definition}

Observe that if an augmented double category has its augmentation consisting of a single point, then we recover the definition of a pointed double category.

The next definition is meant to mimic the notion of bicartesian squares, which are both pullback and pushout squares in an ordinary category, to the setting of double categories, in which horizontal and vertical morphisms need not live in a common ambient category.

\begin{definition}
    A double category is \emph{stable} if every square is uniquely determined by its span of source morphisms and, independently, by its cospan of target morphisms. That is, the maps 
    \begin{align*}
        (s_h, s_v) \colon \sq(\mathcal{D}) \to \ver(\mathcal{D}) \times_{\ob(\mathcal{D})} \hor(\mathcal{D})  \\
        (t_h, t_v) \colon \sq(\mathcal{D}) \to \ver(\mathcal{D}) \times_{\ob(\mathcal{D})} \hor(\mathcal{D})
    \end{align*}
    given by 
    \[\adjustbox{scale=.8,center}{%
	\begin{tikzcd}
		\bullet \arrow[rr, "s_v", tail] \arrow[dd, "s_h"', two heads] && \bullet \arrow[dd, "t_h"', two heads] & \bullet \arrow[dd, two heads] \arrow[rr, tail] && \bullet && \bullet \arrow[rr, "s_v", tail] \arrow[dd, "s_h"', two heads] && \bullet \arrow[dd, "t_h"'] &&& \bullet \arrow[dd, two heads] \\
		&& {} \arrow[r, maps to] & {} &&& \textup{\large and } &&& {} \arrow[r, maps to] & {} && \\
		\bullet \arrow[rr, "t_v", tail] && \bullet & \bullet &&&& \bullet \arrow[rr, "t_v", tail] && \bullet & \bullet \arrow[rr, maps to] && \bullet                      
	\end{tikzcd} }\]
    are required to be bijections.  A double functor between stable double categories is \emph{stable} if it preserves stability.
\end{definition}

Let $\mathcal{D}\Cat_{\aug}^{\st}$ denote the category of augmented stable double categories and augmented stable double functors between them.  Let $2\Seg$ denote the category of 2-Segal sets and simplicial maps between them.  The main result of \cite{2s} is that there is an equivalence of categories
\[ S_\bullet \colon \mathcal{D}\Cat_{\aug}^{\st} \leftrightarrows 2\Seg \colon \mathcal P. \]
The notation $S_\bullet$ is suggestive of the fact that this functor generalizes the $S_\bullet$-construction in algebraic $K$-theory as first defined by Waldhausen \cite{wald}. Here, we are more interested in the functor $\mathcal P$, since we want to start with known examples of 2-Segal sets and produce augmented stable double categories from them.

Roughly speaking, the double category $\mathcal{P}X$ has as its set of objects $X_1$, and both the horizontal and vertical morphisms are given by the set $X_2$, while the set of squares is $X_3$. Although the horizontal and vertical morphisms are both given by the same set abstractly, they do not have the same behaviour.  Given $x \in X_2$, the corresponding  vertical morphism in $\mathcal P(X)$ corresponds to the map $d_1(x) \rightarrow d_0(x)$, while the corresponding horizontal morphism corresponds to the map $d_2(x)\rightarrow d_1(x)$.  The squares can be described analogously.  Furthermore, the inclusion $s_0X_0\subseteq X_1$ defines an augmentation for $\mathcal{P}X$. We refer the reader to \cite{2s} for more details, including the proof that $\mathcal PX$ is stable.
  
Recall that a 2-Segal set $X$ is \emph{reduced} if the set $X_0$ consists of a single point. Observe that, by the above equivalence, reduced 2-Segal sets correspond exactly to pointed stable double categories.  In particular, if $X^T$ is the 2-Segal set associated to a rooted tree $T$, then $\mathcal PX^T$ is a pointed stable double category.
 
If we consider the 2-Segal set $X^T$ associated to a tree $T$, the set of objects of $\mathcal{P}X^T$ are all subforests of $T$ obtainable by a sequence of admissible cuts, which by Proposition \ref{all} includes all subtrees of $T$.  Both the horizontal and the vertical morphisms are given by rooted subforests of $T$ together with one admissible cut, and the squares correspond to the rooted subforests of $T$ with two admissible cuts.

Denote by $Y_n$ the tree that has one vertex of degree $n$ and $n$ vertices of degree $1$, one of which is the root.  For specificity, let us consider the tree $T=Y_3$ as a nonplanar tree without labelling.  Then $\mathcal{P}X^T$ has seven objects: $T$, $K_1$, $2K_1$, $K_2$, $P_3^1$, $P_3^2$ and the empty tree, where $K_i$ is the complete graph on $i$ vertices, $P_3^i$ is the path on three vertices having as root a vertex of degree $i$ for $i=1,2$, and $2K_1$ denotes the forest on two vertices and no edges. To visualize the morphisms, assign to each element in $X_2$, which corresponds to a subtree with a nontrivial cut, a different colour. Then on the left of Figure \ref{color}, 
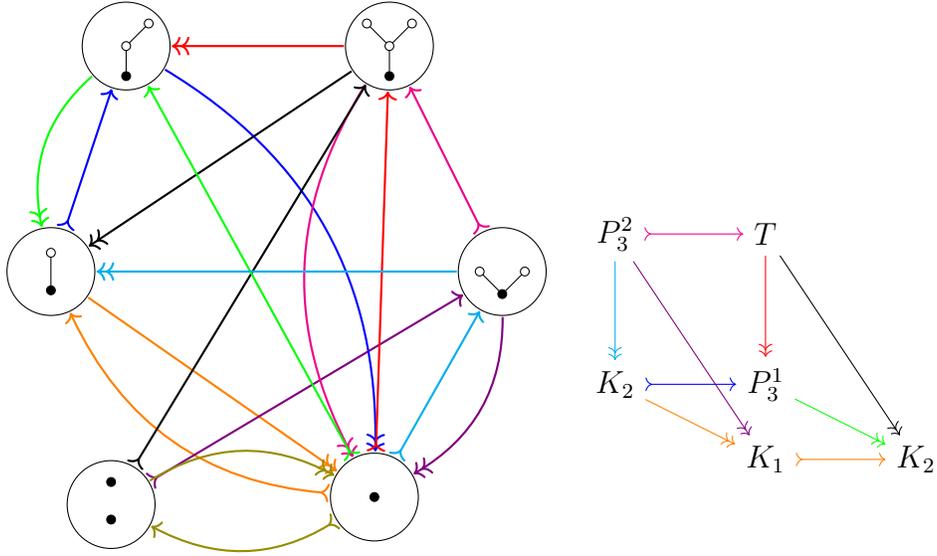
\begin{figure}[ht]
	\centering
	\begin{tikzpicture}
    \node[draw, circle, scale=3] (01) at (4.8, 1.5){};
	\node[draw, circle, scale=.3, fill=black] (1) at ( 4.8, 1.5) {};
  
    \node[draw, circle, scale=3] (0) at (1.3, 1.4){};
	\node[draw, circle, scale=.3, fill=black] (30) at ( 1.3, 1.7) {};
	\node[draw, circle, scale=.3, fill=black] (20) at ( 1.3, 1.2) {};

    \node[draw, circle, scale=3] (23) at (0.5, 4.5){};
	\node[draw, circle, scale=.3] (3) at ( 0.5, 4.75) {};
    \node[draw, circle, scale=.3, fill=black] (2) at ( 0.5, 4.25) {}; 

    \node[draw, circle, scale=3] (456) at (6.5, 4.5){};
	\node[draw, circle, scale=.3] (4) at ( 6.2, 4.5) {};
	\node[draw, circle, scale=.3, fill=black] (5) at ( 6.5, 4.2) {};
	\node[draw, circle, scale=.3] (6) at ( 6.8, 4.5) {};
  
	\node[draw, circle, scale=.3] (7) at ( 1.5, 7.5) {};
	\node[draw, circle, scale=.3] (8) at ( 1.8, 7.8) {};
	\node[draw, circle, scale=.3, fill=black] (9) at ( 1.5, 7.1) {};
    \node[draw, circle, scale=3] (78) at (1.5, 7.5){};
  
	\node[draw, circle, scale=.3] (10) at ( 4.7, 7.8) {};
	\node[draw, circle, scale=.3] (11) at ( 5, 7.5) {};
	\node[draw, circle, scale=.3] (12) at ( 5.3, 7.8) {};
	\node[draw, circle, scale=.3, fill=black] (13) at ( 5, 7.1) {};
    \node[draw, circle, scale=3] (T) at (5,7.5){};
		\draw (3) -- (2);
		\draw (7) -- (8);
		\draw (9) -- (7);
		\draw (11) -- (10);
		\draw (12) -- (11);
		\draw (13) -- (11);
		\draw (5) -- (4);
		\draw (6) -- (5);
        \draw[orange, thick, ->>](23) to (01);
        \draw[orange,bend left, thick, >->](01) to (23);
        \draw[violet, bend left, thick, ->>](456) to (01);
        \draw[violet, thick, >->](0) to (456);
        \draw[blue, bend left, thick, ->>](78) to (01);
        \draw[blue, thick, >->](23) to (78);
        \draw[red, thick, ->>](T) to (78);
        \draw[red, thick, >->](01) to (T);
        \draw[magenta,bend right, thick, ->>](T) to (01);
        \draw[magenta, thick, >->](456) to (T);
        \draw[thick, ->>] (T) to (23);
        \draw[thick, >->] (0) to (T);
        \draw[green, bend right, thick, ->>](78) to (23);
        \draw[green, thick, >->](01) to (78);
        \draw[cyan,  thick, ->>](456) to (23);
        \draw[cyan, thick, >->](01) to (456);
        \draw[olive, thick, >->, bend left](01) to (0);
        \draw[olive, thick, ->>, bend left](0) to (01);
    \node (V) at (8,5){$P_3^2$};
    \node (Y) at (10,5){$T$};
    \node (I) at (8,3){$K_2$};
    \node (0) at (10,2){$K_1$};
    \node (l) at (12,2){$K_2$};
    \node (L) at (10,3){$P_3^1$};
        \draw[magenta, >->] (V) to (Y);
        \draw[cyan, ->>] (V) to (I);
        \draw[blue, >->] (I) to (L);
        \draw[red, ->>] (Y) to (L);
        \draw[orange, ->>] (I) to (0);
        \draw[orange, >->] (0) to (l);
        \draw[green, ->>] (L) to (l);
        \draw[violet,->>] (V) to (0);
        \draw[->>] (Y) to (l);
    \end{tikzpicture}
	\caption{The objects, morphisms and three squares of $\mathcal{P}X^T$. }
	\label{color}
\end{figure}
we can see the corresponding morphisms in $\mathcal{P}X^T$; the root of each subtree is the black vertex.  Each element in $X^T_2$ corresponds to one vertical and one horizontal morphism. Both trivial cuts give us the vertical and horizontal identity morphisms, and vertical and horizontal morphisms to the empty tree on each object. 

We have three squares corresponding to elements in $X_3$ which are subtrees with two nontrivial cuts, which we can see on the right-hand side of Figure \ref{color}. Every other square corresponds to a tree with either one or two trivial cuts.

If we instead regard the tree $T$ as planar, then $\mathcal{P}X^T$ has eight objects and twelve non-identity horizontal and vertical morphisms, together with the morphisms to and from the empty tree. Moreover, there are four squares that correspond to subtrees with two nontrivial cuts. Thus, in this case, the associated double category is slightly larger than in the nonplanar case. To visualize it, one can take the picture in Figure \ref{color} and duplicate the object corresponding to $P_2^2$ together with the three vertical and two horizontal morphisms adjacent to that object. 

Finally, let us assign a labelling to the tree $T$ and consider the associated double category.  It has thirteen objects, fourteen nontrivial horizontal and vertical morphisms, and seven squares corresponding to subtrees with two nontrivial cuts. In all cases, we have more squares that correspond to subtrees with one trivial cut and one nontrivial cut or two trivial cuts.

So, assume that $T$ has vertices labelled by $1,2,3,4$, with root labelled 1, as follows:
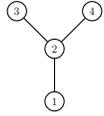
\[ \begin{tikzpicture} 
    \node (x) at (-0.3,0.5){};
    \node[draw, circle, scale=.35] (0) at (0,0){1};
    \node[draw, circle, scale=.35] (1) at (0,.7){2};
    \node[draw, circle, scale=.35] (2) at (0.5,1.2){4};
    \node[draw, circle, scale=.35] (3) at (-0.5,1.2){3};
    \draw (1)--(3);
    \draw (1)--(2);
    \draw (0)--(1);
\end{tikzpicture}. \] 
We can denote each object by the labels of its vertices as shown in Figure \ref{oblabel}.
\begin{figure}[ht]
    \centering
    \begin{tabular}{ccccccccccccc}
    $\varnothing$ & $ \bullet_1$ & $\bullet_2$ & $\bullet_3$ & $\bullet_4$ & \begin{tikzpicture}
        \node (x) at (-0.3,0.5){{\footnotesize{$123$}}};
        \node[draw, circle, scale=.35, fill=black] (0) at (0,0){};
        \node[draw, circle, scale=.35
        ] (1) at (0,.7){};
        \node[draw, circle, scale=.35
        ] (2) at (-0.5,1.2){};
        \draw (1)--(2);
        \draw (0)--(1);
    \end{tikzpicture}
    & 
    \begin{tikzpicture}
        \node (x) at (-0.3,0.5){{\footnotesize{$124$}}};
        \node[draw, circle, scale=.35, fill=black] (0) at (0,0){};
        \node[draw, circle, scale=.35
        ] (1) at (0,.7){};
        \node[draw, circle, scale=.35
        ] (2) at (0.5,1.2){};
        \draw (1)--(2);
        \draw (0)--(1);
    \end{tikzpicture}
    & 
    \begin{tikzpicture}
        \node (x) at (0,1.2){{\footnotesize{$24$}}};
        \node[draw, circle, scale=.35, fill=black] (1) at (0,.7){};
        \node[draw, circle, scale=.35
        ] (2) at (0.5,1.4){};
        \draw (2)--(1);
    \end{tikzpicture}
    &
    \begin{tikzpicture}
        \node (x) at (0,1.2){{\footnotesize{$23$}}};
        \node[draw, circle, scale=.35, fill=black] (1) at (0,.7){};
        \node[draw, circle, scale=.35
        ] (2) at (-0.5,1.4){};
        \draw (2)--(1);
    \end{tikzpicture}
    &
    \begin{tikzpicture}
        \node (x) at (-0.3,0.5){{\footnotesize{$12$}}};
        \node[draw, circle, scale=.35, fill=black] (0) at (0,0){};
        \node[draw, circle, scale=.35
        ] (1) at (0,.7){};
        \draw (0)--(1);
    \end{tikzpicture}
    & 
    \begin{tikzpicture}
        \node (x) at (0,1.2){{\footnotesize{$234$}}};
        \node[draw, circle, scale=.35, fill=black] (1) at (0,.7){};
        \node[draw, circle, scale=.35
        ] (2) at (-0.5,1.4){}; 
        \node[draw, circle, scale=.35
        ] (3) at (0.5,1.4){};
        \draw (1)--(2);
        \draw (1)--(3);
    \end{tikzpicture}
    &
    $\bullet \bullet_{34}$ & $T$\\
    \end{tabular}
    \caption{The thirteen objects in $\mathcal{P}X^T$.} \label{oblabel}
\end{figure}
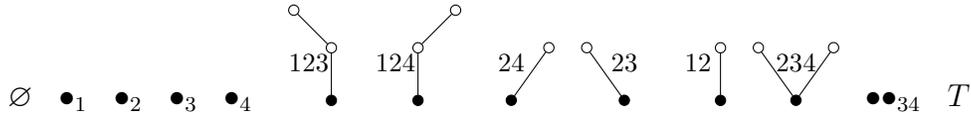

Then the morphisms corresponding to nontrivial cuts are shown in Figure \ref{14m}, 
\begin{figure}
    \centering
    \begin{tikzpicture}
        \node (3) at (0,-0.5){$3$};
        \node (34) at (2,-0.5){$34$};
        \node (4) at (4,-0.5){$4$};
        \node (123) at (-1,1){$123$};
        \node (t) at (1,1){$T$};
        \node (234) at (3,1){$234$};
        \node (24) at (5,1){$24$};
        \node (1) at (0,2.5){$1$};
        \node (124) at (4,2.5){$124$};
        \node (12) at (0,4){$12$};
        \node (23) at (2,3.5){$23$};
        \node (2) at (4,4){$2$};
        \draw[->>] (24)--(4);
        \draw[->>] (23)--(3);
        \draw[->>] (124)--(4);
        \draw[->>] (124)--(24);
        \draw[->>] (12)--(2);
        \draw[->>] (123)--(23);
        \draw[->>] (123)--(3);
        \draw[->>] (t)--(234);
        \draw[->>] (t)--(3);
        \draw[->>] (t)--(4);
        \draw[->>] (t)--(34);
        \draw[->>] (234)--(4);
        \draw[->>] (234)--(3);
        \draw[->>] (234)--(34);

        \draw[>->, bend left] (3) to (34);
        \draw[>->, bend left] (4) to (34);
        \draw[->>, bend left] (34) to (3);
        \draw[->>, bend left] (34) to (4);
        \draw[>->] (1)--(12);
        \draw[>->] (1)--(123);
        \draw[>->] (1)--(124);
        \draw[>->] (1)--(t);
        \draw[>->] (12)--(123);
        \draw[>->] (12)--(124);
        \draw[>->] (12)--(t);
        \draw[>->] (2)--(234);
        \draw[>->] (2)--(23);
        \draw[>->] (2)--(24);
        \draw[>->] (24)--(234);
        \draw[>->] (23)--(234);
        \draw[>->] (123)--(t);
        \draw[>->] (124)--(t);
    \end{tikzpicture}
    \caption{The fourteen morphisms of $\mathcal{P}X^T$ that correspond to a subtree with a nonempty cut.} \label{14m}
\end{figure}
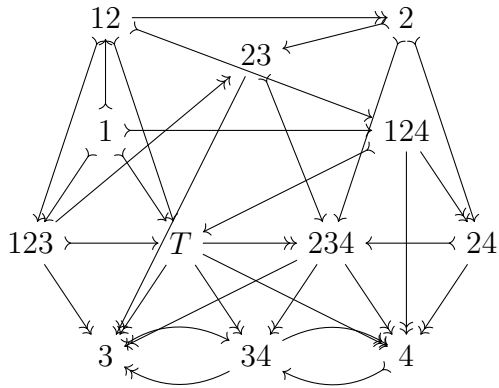
and the seven squares corresponding to two nontrivial cuts are shown in Figure \ref{7sq}.
\begin{figure}[ht]
    \centering
    \begin{tikzpicture}
    \node (23) at (0,1){$23$};
    \node (13) at (1.5,1){$123$};
    \node (02) at (0,0){$2$};
    \node (12) at (1.5,0){$12$};
    \draw[>->] (23)--(13);
    \draw[->>] (23)--(02);
    \draw[->>] (13)--(12);
    \draw[>->] (02)--(12);
    \node (231) at (3,1){$24$};
    \node (131) at (4.5,1){$124$};
    \node (021) at (3,0){$2$};
    \node (121) at (4.5,0){$12$};
    \draw[>->] (231)--(131);
    \draw[->>] (231)--(021);
    \draw[->>] (131)--(121);
    \draw[>->] (021)--(121);
    \node (232) at (6,1){$234$};
    \node (132) at (7.5,1){$T$};
    \node (022) at (6,0){$23$};
    \node (122) at (7.5,0){$123$};
    \draw[>->] (232)--(132);
    \draw[->>] (232)--(022);
    \draw[->>] (132)--(122);
    \draw[>->] (022)--(122);
    \node (233) at (9,1){$234$};
    \node (133) at (10.5,1){$T$};
    \node (023) at (9,0){$24$};
    \node (123) at (10.5,0){$124$};
    \draw[>->] (233)--(133);
    \draw[->>] (233)--(023);
    \draw[->>] (133)--(123);
    \draw[>->] (023)--(123);
    \node (23x) at (1,3){$123$};
    \node (13x) at (2.5,3){$T$};
    \node (02x) at (1,2){$12$};
    \node (12x) at (2.5,2){$124$};
    \draw[>->] (23x)--(13x);
    \draw[->>] (23x)--(02x);
    \draw[->>] (13x)--(12x);
    \draw[>->] (02x)--(12x);
    \node (231x) at (4,3){$23$};
    \node (131x) at (5.5,3){$234$};
    \node (021x) at (4,2){$23$};
    \node (121x) at (5.5,2){$123$};
    \draw[>->] (231x)--(131x);
    \draw[->>] (231x)--(021x);
    \draw[->>] (131x)--(121x);
    \draw[>->] (021x)--(121x);
    \node (232x) at (7,3){$234$};
    \node (132x) at (8.5,3){$T$};
    \node (022x) at (7,2){$2$};
    \node (122x) at (8.5,2){$12$};
    \draw[>->] (232x)--(132x);
    \draw[->>] (232x)--(022x);
    \draw[->>] (132x)--(122x);
    \draw[>->] (022x)--(122x);
    \end{tikzpicture}
    \caption{Squares from nontrivial cuts.} \label{7sq}
\end{figure}
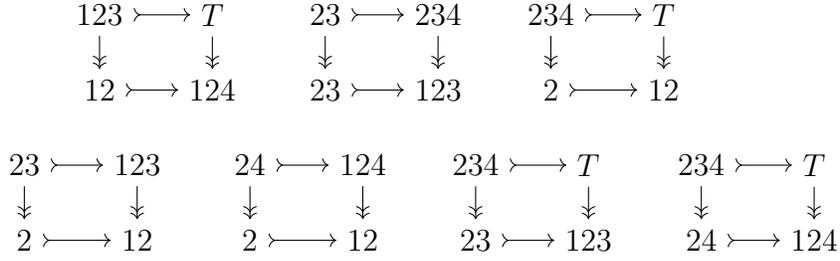 

Now, let us consider the family of what are sometimes called linear trees.  Let $TP_n$ denote the tree defined by the path graph on $n$ vertices having as a root a vertex of degree one:
\[ \begin{tikzpicture}
    \node[draw, circle, scale=.5, fill=black]  (0) at (0,0){};
    \node(r) at (0, -0.3){{}};
    \node[draw, circle, scale=.5]  (1) at (1,0){};
    \node[draw, circle, scale=.5]  (2) at (2,0){};
    \node[draw, circle, scale=.5]  (3) at (3,0){};
    \node[draw, circle, scale=.5]  (4) at (4,0){};
    \node[draw, circle, scale=.5]  (5) at (5,0){};\node[draw, circle, scale=.5]  (6) at (6,0){};
    \draw (0)--(1);
    \draw (1)--(2);
    \draw (2)--(3);
    \draw (3)--(4);
    \draw[dotted] (4)--(5);
    \draw (5)--(6);
\end{tikzpicture}.\]
Then the subtrees obtained by admissible cuts are the paths $P_k$ with $1 \leq k \leq n$, together with the empty tree. Thus we have an object of $\mathcal PX^{P_n}$ for every positive integer $k<n$ which, together with the empty tree, gives us $n+1$ objects. As we previously noticed, for every object $P_k$ there exists a vertical morphism from the empty tree to $P_k$ and from $P_k$ to the empty tree. Furthermore, we have a horizontal morphism from $P_i$ to $P_j$ if $i \leq j$ and a vertical morphism from $P_j$ to $P_i$ if $j \geq i$.

Finally, given integers $i,j,k$ such that $ i<j<k<n$, there exists a square with horizontal morphisms $k-i \rightarrowtail k$ and $j-i \rightarrowtail j$ and with vertical morphisms $k-1 \twoheadrightarrow j-i$ and $k \twoheadrightarrow j$, which means that there are $\binom{n}{3}$ squares arising from a subtree with two nontrivial cuts.

\begin{example}
    For example, consider the tree $P_4$. Its associated double category has five objects, thirty morphisms including identities, and four squares corresponding to subtrees with two nontrivial cuts.  A depiction is given in Figure \ref{P4}.
    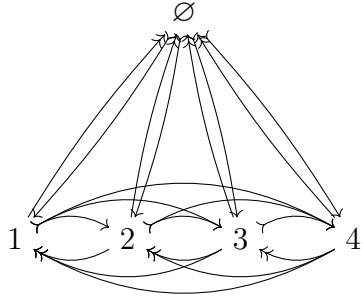
\begin{figure}[h]
        \centering
       \begin{tikzpicture}
        \node (1) at (  0,0){1};
        \node (2) at (1.5,0){2};
        \node (3) at (  3,0){3};
        \node (4) at (4.5,0){4};
        \node (0) at (2.25, 3){$\varnothing$};
        \draw[>->, bend left=5] (0)to(2);
        \draw[>->, bend left=5] (0)to(3);
        \draw[>->, bend left=5] (0)to(4);
        \draw[>->, bend left=5] (0)to(1);
        \draw[->>, bend left=5] (1)to(0);
        \draw[->>, bend left=5] (2)to(0);
        \draw[->>, bend left=5] (3)to(0);
        \draw[->>, bend left=5] (4)to(0);
        \draw[>->, bend left] (1)to(2);
        \draw[>->, bend left] (2)to(3);
        \draw[>->, bend left] (3)to(4);
        \draw[>->, bend left] (1)to(3);
        \draw[>->, bend left] (2)to(4);
        \draw[>->, bend left] (1)to(4);
        \draw[->>, bend left] (2)to(1);
        \draw[->>, bend left] (3)to(2);
        \draw[->>, bend left] (4)to(3);
        \draw[->>, bend left] (3)to(1);
        \draw[->>, bend left] (4)to(2);
        \draw[->>, bend left] (4)to(1);
       \end{tikzpicture}
        \caption{The objects and arrows of the double category associated to the tree $P_4$.} \label{P4}
    \end{figure}
\end{example}

Now let $K_{1,n}$ denote the tree with $n$ vertices of degree one and one vertex of degree $n$, with the latter taken to be the root vertex. Then the subtrees obtainable by admissible cuts are the trees $K_{1,i}$ for $0 \leq i \leq n$, and the subforests we obtain are $F_i$ for $1 \leq i \leq n$, where $F_i$ denotes the forest on $i$ vertices and no edges. Notice that $F_1=K_{1,0}$, thus together with the empty tree, we have $2n$ objects in the associated double category. The horizontal morphisms are $F_i \rightarrowtail K_{1,j}$ and $F_i \rightarrowtail F_j$ for $0 \leq  i \leq j \leq n$, and the vertical morphisms are $K_{1,j} \twoheadrightarrow K_{1,i}$ and $F_j \twoheadrightarrow F_i$ for $0 \leq i \leq j \leq n$, together with the morphisms from and to the empty tree.  We thus have $2\binom{n+1}{2}+2n$ horizontal and vertical morphisms besides the identity morphisms. Finally, the squares are of the form 
 \[  \begin{tikzpicture}
    \centering
    \node (23) at (0,1){$F_j$};
    \node (13) at (1.8,1){$K_{1,k}$};
    \node (02) at (0,-0.5){$F_j$};
    \node (12) at (1.8,-0.5){$K_{k-i}$};
    \draw[>->] (23)--(13);
    \draw[->>] (23)--(02);
    \draw[->>] (13)--(12);
    \draw[>->] (02)--(12); 
\end{tikzpicture}  \begin{tikzpicture}
    \centering
    \node (23) at (0,1){$F_{k-i}$};
    \node (13) at (1.8,1){$F_{k}$};
    \node (02) at (0,-0.5){$F_{j-i}$};
    \node (12) at (1.8,-0.5){$F_{j}$};
    \draw[>->] (23)--(13);
    \draw[->>] (23)--(02);
    \draw[->>] (13)--(12);
    \draw[>->] (02)--(12);  \end{tikzpicture} \]
where the right square corresponds to the subtree $K_{1,k}$, where the first cut passes through $i$ edges and the second cut passes through $j$ edges. The left square corresponds to the forest $F_k$ with two cuts where one cut separates $i$ vertices, and the other cut separates $j$ vertices for $i<j$.

\begin{example}
    Consider the tree $K_{1,3}$ having as root the vertex of degree three. Then the associated double category has seven objects and fifteen nonidentity horizontal and vertical morphisms. In Figure \ref{dk13}, 
    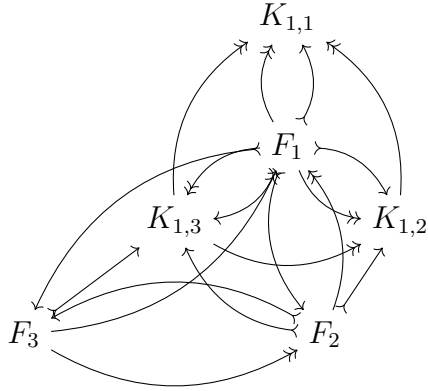
\begin{figure}[h]
    \centering
    \begin{tikzpicture}
        \node (x3) at (-2,-0.5){$F_3$};
        \node (y3) at (0,1){$K_{1,3}$};
        \node (x1) at (1.5,2){$F_1$}; 
        \node (y1) at (1.5,3.7){$K_{1,1}$};
        \node (y2) at (3,1){$K_{1,2}$};
        \node (x2) at (2,-0.5){$F_2$};
        \draw[>->] (x3)to(y3);
        \draw[>->] (x2)to(y2);
        \draw[>->, bend left] (x2)to(y3);
        \draw[>->, bend right] (x1)to(y1);
        \draw[>->, bend left] (x1)to(y2);
        \draw[>->, bend left] (x1)to(y3);
        \draw[->>, bend right] (y3)to(y2);
        \draw[->>,bend left] (y3)to(y1);
        \draw[->>,bend right] (y2)to(y1);
        \draw[->>,bend right] (x1)to(y2);
        \draw[->>,bend left]  (x1)to(y1);
        \draw[->>,bend right] (x1)to(y3);
        
        \draw[->>,bend right] (x2)to(x1);
        \draw[>->,bend right] (x1)to(x2);
        \draw[>->,bend right] (x2)to(x3);
        \draw[->>,bend right] (x3)to(x1);
        \draw[>->,bend right] (x1)to(x3);
        \draw[->>,bend right] (x3)to(x2);
    \end{tikzpicture}
    \caption{The objects and morphisms corresponding to nonempty subforests in the double category associated to the tree $K_{1,3}$.}
    \label{dk13}
    \end{figure}
    we see every object corresponding to a nonempty subtree and the morphisms between them. Finally, there are five squares that correspond to subtrees with two nontrivial cuts.
\end{example}

Finally, we consider the trees $Z_{m,n}$ obtained by subdividing one edge of the tree $K_{1,m+1}$ into $n$ edges. Alternatively, it can be obtained by identifying the root vertices of $P_{n+1}$, with the root a vertex of degree one, and $K_{1,m}$, with root the vertex of degree $m$: 
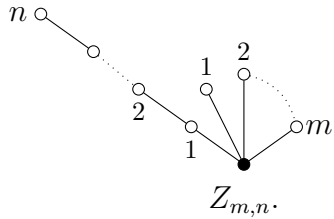
\[ \begin{tikzpicture}
        \node[draw, circle, scale=.4,fill] (0) at (0,0) {};
		\node[draw, circle, scale=.4] (1) at (-0.5,1) {};
        \node at (-0.5, 1.3){\footnotesize{$1$}};
		\node[draw, circle, scale=.4] (2) at (0,1.2) {};
		\node at (0, 1.5){\footnotesize{$2$}};
		\node[draw, circle, scale=.4] (4) at (0.7,0.5) {};
        \node at (1, 0.5){$m$};
        \node[draw, circle, scale=.4] (5) at (-0.7,0.5) {};
        \node at (-0.7, 0.2){\footnotesize{$1$}};
		\node[draw, circle, scale=.4] (7) at (-1.4, 1) {};
        \node at (-1.4, 0.7){\footnotesize{$2$}};
		\node[draw, circle, scale=.4] (6) at (-2, 1.5) {};
        \node[draw, circle, scale=.4] (8) at (-2.7, 2){};
        \node at (-3, 2){$n$};
        \node (t) at (0,-0.5){$Z_{m,n}$.};

        \draw (5)--(7);
        \draw[dotted] (6)--(7);
        \draw (0)--(5);
		\draw (0) -- (2);
		\draw[bend left, dotted] (2)to(4);
		\draw (0) -- (4);
		\draw (1) -- (0);
        \draw (6)--(8);
        \end{tikzpicture}\]
Then the nonempty forests we obtain by taking cuts are:
\begin{itemize}
    \item $K_{1,\ell}$ for $2 \leq \ell \leq m+1$;
    
    \item the forest $F_i$ on $i$ vertices and no edges for $1 \leq i \leq m+1$;
    
    \item the path tree $P_j$ for $2 \leq j \leq n$;
    
    \item the disjoint union 
    $W_{i,j}:= F_i \union P_j$ for $1 \leq i \leq m$ and $2 \leq j \leq n-1$; and
    
    \item the subtrees $Z_{i,j}$ for $1 \leq i \leq m $ and $2\leq j \leq n$.
\end{itemize}
See Figure \ref{obZ} for examples.
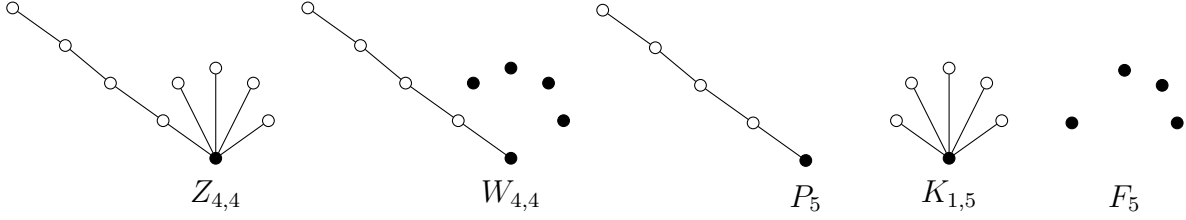
\begin{figure}[h]
        \centering
        \begin{tabular}{ccccc}
         \begin{tikzpicture}
        \node[draw, circle, scale=.4,fill] (0) at (0,0) {};
		\node[draw, circle, scale=.4] (1) at (-0.5,1) {};
		\node[draw, circle, scale=.4] (2) at (0,1.2) {};
		\node[draw, circle, scale=.4] (3) at (0.5,1) {};
		\node[draw, circle, scale=.4] (4) at (0.7,0.5) {};
        \node[draw, circle, scale=.4] (5) at (-0.7,0.5) {};
		\node[draw, circle, scale=.4] (7) at (-1.4, 1) {};
		\node[draw, circle, scale=.4] (6) at (-2, 1.5) {};
        \node[draw, circle, scale=.4] (8) at (-2.7, 2){};
        \node (t) at (0,-0.5){$Z_{4,4}$};

        \draw (5)--(7);
        \draw (6)--(7);
        \draw (0)--(5);
		\draw (0) -- (2);
		\draw (0) -- (3);
		\draw (0) -- (4);
		\draw (1) -- (0);
        \draw (6)--(8);
        \end{tikzpicture}
        &
        \begin{tikzpicture}
        \node[draw, circle, scale=.4,fill] (0) at (0,0) {};
		\node[draw, circle, scale=.4,fill] (1) at (-0.5,1) {};
		\node[draw, circle, scale=.4,fill] (2) at (0,1.2) {};
		\node[draw, circle, scale=.4,fill] (3) at (0.5,1) {};
		\node[draw, circle, scale=.4,fill] (4) at (0.7,0.5) {};
        \node[draw, circle, scale=.4] (5) at (-0.7,0.5) {};
		\node[draw, circle, scale=.4] (7) at (-1.4, 1) {};
		\node[draw, circle, scale=.4] (6) at (-2, 1.5) {};
        \node[draw, circle, scale=.4] (8) at (-2.7, 2){};
        \node (t) at (0,-0.5){$W_{4,4}$};

        \draw (5)--(7);
        \draw (0)--(5);
		\draw (6)--(7);
        \draw (6)--(8);
        \end{tikzpicture}
        &
        \begin{tikzpicture}
        \node[draw, circle, scale=.4,fill] (0) at (0,0) {};
        \node[draw, circle, scale=.4] (5) at (-0.7,0.5) {};
		\node[draw, circle, scale=.4] (7) at (-1.4, 1) {};
		\node[draw, circle, scale=.4] (6) at (-2, 1.5) {};
        \node[draw, circle, scale=.4] (8) at (-2.7, 2){};
        \node (t) at (0,-0.5){$P_{5}$};
        \draw(0)--(5);
        \draw (5)--(7);
        \draw (6)--(7);
        \draw (6)--(8);
        \end{tikzpicture}
         & \quad
        \begin{tikzpicture}
        \node[draw, circle, scale=.4,fill] (0) at (0,0) {};
		\node[draw, circle, scale=.4] (1) at (-0.5,1) {};
		\node[draw, circle, scale=.4] (2) at (0,1.2) {};
		\node[draw, circle, scale=.4] (3) at (0.5,1) {};
		\node[draw, circle, scale=.4] (4) at (0.7,0.5) {};
        \node[draw, circle, scale=.4] (5) at (-0.7,0.5) {};
        \node (t) at (0,-0.5){$K_{1,5}$};

        \draw (0)--(5);
		\draw (0) -- (2);
		\draw (0) -- (3);
		\draw (0) -- (4);
		\draw (1) -- (0);
        \end{tikzpicture}
         &\quad
        \begin{tikzpicture}
		\node[draw, circle, scale=.4,fill] (2) at (0,1.2) {};
		\node[draw, circle, scale=.4,fill] (3) at (0.5,1) {};
		\node[draw, circle, scale=.4,fill] (4) at (0.7,0.5) {};
        \node[draw, circle, scale=.4,fill] (5) at (-0.7,0.5) {};
        \node (t) at (0,-0.5){$F_{5}$};
        \end{tikzpicture}
        \end{tabular}
       \caption{The trees and forests $Z_{4,4}, W_{4,4}, P_5, K_{1,5}$ and $F_5.$} \label{obZ}
\end{figure}

Notice that we can write $F_1 = K_{1,0}=P_1= W_{0,1} = W_{1,0} = Z_{0,1}$; while we are not considering them as distinct objects, this observation is important for describing the morphisms. Furthermore, $W_{1,1}=F_2,$ $K_{1,1}=P_2$ and while $K_{1,2}$ is the same tree as $P_3$, they are considered to be different since they have a different root; $P_i$ always has as a root a vertex of degree one and $K_{1,\ell}$ a vertex of degree $\ell$. Finally we also have $Z_{i,0}=K_{1,i}, Z_{0,i}=P_i=W_{0,i}, Z_{i,1} = K_{1,i+1}$ and $W_{i,0}=F_i$.

We know from the previously considered families of trees that we have morphisms:
\begin{itemize}
    \item $P_i \rightarrowtail P_j$ and $P_j \twoheadrightarrow P_i$ for $i \leq j$;
    
    \item $W_{i,j} \rightarrowtail W_{i,k}$ and $W_{i,k} \twoheadrightarrow W_{i,j}$ for $2 \leq j \leq k$; 
    
    \item $F_i \rightarrowtail K_{1,j}$ and $K_{1,j} \twoheadrightarrow K_{1,i} $ for $i \leq j$; and 
    
    \item $F_i \rightarrowtail F_j$ and $F_j \twoheadrightarrow F_i$ for every $i,j$.
\end{itemize}
Thus, it suffices to analyse morphisms associated to the subtree $Z_{k,\ell}$ with a nontrivial admissible cut for $1\leq k \leq m$ and $2 \leq \ell \leq n$. To do so, denote by $e_j$ the admissible cut that passes through the $j$th edge from the root of $P_\ell \subseteq Z_{k,\ell}$ for $1 \leq j \leq \ell$. Denote by $g_i$ the cut that passes through $i$ edges of $K_{1,k} \subseteq Z_{k,\ell}$ for $1 \leq k \leq \ell$, and denote by $g_ie_j$ the cut that is the union of the cuts $g_i$ and $e_j$; see, for example, Figure \ref{gecuts}.
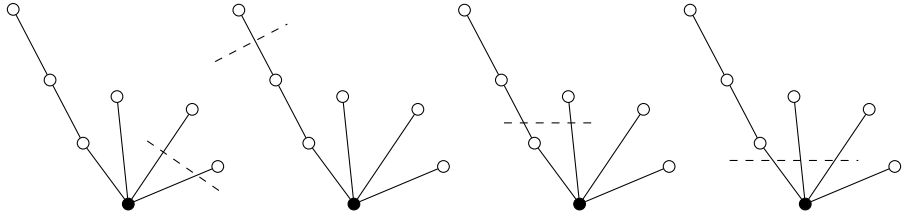
\begin{figure}[ht]
    \centering
    \begin{tikzpicture}[rotate = 90]
        \node[draw, circle, scale=.4,fill] (010) at ( 2.42, 0.85-4) {};
		\node[draw, circle, scale=.4] (020) at ( 2.92, 1.66-6) {};
		\node[draw, circle, scale=.4] (030) at ( 3.68, -4) {};
		\node[draw, circle, scale=.4] (040) at ( 3.85, -3) {};
		\node[draw, circle, scale=.4] (050) at ( 3.23, 0.45-3) {};
		\node[draw, circle, scale=.4] (060) at ( 4.07, 0.89-3) {};
		\node[draw, circle, scale=.4] (070) at ( 5.00, 1.38-3) {};
        \node (080) at (3, -4){};
        \node (090) at (3, -2){};
       \draw[-, dashed] (090)--(080);

		\draw (010) -- (020);
		\draw (010) -- (030);
		\draw (010) -- (040);
		\draw (010) -- (050);
		\draw (060) -- (050);
		\draw (070) -- (060);
    
        \node[draw, circle, scale=.4,fill] (01) at ( 2.42, 0.85-1) {};
		\node[draw, circle, scale=.4] (02) at ( 2.92, 1.66-3) {};
		\node[draw, circle, scale=.4] (03) at ( 3.68, -1) {};
		\node[draw, circle, scale=.4] (04) at ( 3.85, 0) {};
		\node[draw, circle, scale=.4] (05) at ( 3.23, 0.45) {};
		\node[draw, circle, scale=.4] (06) at ( 4.07, 0.89) {};
		\node[draw, circle, scale=.4] (07) at ( 5.00, 1.38) {};
        \node (80) at (3.5, -.5){};
        \node (90) at (3.5, 1){};
        \draw[-, dashed] (90)--(80);

		\draw (01) -- (02);
		\draw (01) -- (03);
		\draw (01) -- (04);
		\draw (01) -- (05);
		\draw (06) -- (05);
		\draw (07) -- (06);
    
        \node[draw, circle, scale=.4,fill] (1) at ( 2.42, 2.85) {};
		\node[draw, circle, scale=.4] (2) at ( 2.92, 1.66) {};
		\node[draw, circle, scale=.4] (3) at ( 3.68, 2) {};
		\node[draw, circle, scale=.4] (4) at ( 3.85, 3) {};
		\node[draw, circle, scale=.4] (5) at ( 3.23, 3.45) {};
		\node[draw, circle, scale=.4] (6) at ( 4.07, 3.89) {};
		\node[draw, circle, scale=.4] (7) at ( 5.00, 4.38) {};
		\node (8) at ( 4.89, 3.59) {};
		\node (9) at ( 4.24, 4.84) {};

		\draw (1) -- (2);
		\draw (1) -- (3);
		\draw (1) -- (4);
		\draw (1) -- (5);
		\draw (6) -- (5);
		\draw (7) -- (6);
		\draw[-, dashed] (9) -- (8);

        \node[draw, circle, scale=.4,fill] (10) at ( 2.42, 5.85) {};
		\node[draw, circle, scale=.4] (20) at ( 2.92, 4.66) {};
        \node (34) at (3.4, 5.8){};
        \node (00) at (2.5, 4.5){};
        \draw[-, dashed] (00)--(34);
		\node[draw, circle, scale=.4] (30) at ( 3.68, 5) {};
		\node[draw, circle, scale=.4] (40) at ( 3.85, 6) {};
		\node[draw, circle, scale=.4] (50) at ( 3.23, 6.45) {};
		\node[draw, circle, scale=.4] (60) at ( 4.07, 6.89) {};
		\node[draw, circle, scale=.4] (70) at ( 5.01, 7.38) {};

		\draw (10) -- (20);
		\draw (10) -- (30);
		\draw (10) -- (40);
		\draw (10) -- (50);
		\draw (60) -- (50);
		\draw (70) -- (60);
    \end{tikzpicture}
    \caption{The tree $Z_{3,3}$ with the cuts $g_2, e_3, g_1e_2$ and $g_2e_1 $ respectively. }
    \label{gecuts}
\end{figure}
Then:
\begin{itemize}
    \item from ($Z_{k,\ell}, e_j)$ we obtain morphisms $Z_{k,(j-1)}\rightarrowtail Z_{k,\ell}$ and $Z_{k,\ell}\twoheadrightarrow P_{\ell-j}$;
    
    \item from $(Z_{k,\ell},g_i)$ we obtain morphisms $Z_{(k-i),\ell} \rightarrowtail Z_{k,\ell}$ and $Z_{k,\ell}\twoheadrightarrow  F_i$; and 
    
    \item from $(Z_{k,\ell},g_ie_j)$ we obtain morphisms $Z_{(k-i),(j-1)} \rightarrowtail Z_{k,\ell}$ and $Z_{k,\ell}\twoheadrightarrow  W_{i,\ell-j}$.
\end{itemize}
For the squares, consider the tree $Z_{m,n}$ with the cuts $g_ie_j$ and $g_\ell e_k$ where $i + \ell \leq m$ and $k<j\leq n$. Then the corresponding square is 
\[ \begin{tikzpicture}
    \centering
    \node (23) at (0,1.5){$W_\ell$};
    \node (13) at (3,1.5){$Z_{m,n}$};
    \node (02) at (0,0){$W_{\ell, j-k}$};
    \node (12) at (3,0){$Z_{m-i,j-1}.$};
    \draw[>->] (23)--(13);
    \draw[->>] (23)--(02);
    \draw[->>] (13)--(12);
    \draw[>->] (02)--(12); 
\end{tikzpicture} \]

\begin{example}
    Consider the tree $Z_{2,2}$. Then in Figure \ref{z22im} we can see the objects of the associated double category that correspond to nonempty subtrees and forests together with the morphisms between them. The pink arrows are horizontal morphisms and the black arrows are vertical morphisms.
    \begin{figure}
        \centering
        \begin{tikzpicture}[scale=2,thick]
		\tikzstyle{every node}=[minimum width=0pt, inner sep=2pt, circle]
			\draw (-1.1,2.7) node[draw, fill=black] (0) {};
			\draw (-0.8,2.7) node[draw, fill=black] (1) {};
			\draw (-0.95,2.7) node[draw, scale=6] (2) {};
			\draw (-2.57,2.44) node[draw] (3) { };
			\draw (-2.18,2.43) node[draw, fill=black] (4) {};
			\draw (-2.38,2.42) node[draw, scale=6] (5) { };
			\draw (0.14,2.6) node[draw, fill=black] (6) {};
			\draw (0.54,2.61) node[draw, fill=black] (7) { };
			\draw (0.33,2.3) node[draw, fill=black] (8)  {};
			\draw (0.34,2.53) node[draw, scale=6] (9) { };
			\draw (0.98,1.85) node[draw] (10) {};
			\draw (1.37,1.87) node[draw] (11) {};
			\draw (1.17,1.47) node[draw, fill=black] (12) { };
			\draw (1.18,1.7) node[draw, scale=7] (13) {};
			\draw (-3.51,1.66) node[draw, fill=black] (14) {};
			\draw (-3.17,1.65) node[draw] (15) {  };
			\draw (-2.81,1.65) node[draw] (16) {  };
			\draw (-3.15,1.65) node[draw, ellipse, minimum width=80pt, minimum height=20pt] (17) {  };
			\draw (-0.95,1.3) node[draw, scale=6] (18) { };
			\draw (-0.95,1.3) node[draw, fill=black] (19) { };
			\draw (-3.32,0.82) node[draw, fill=black] (20) { };
            \draw (-2.95,0.82) node[draw, fill=black] (21) { };
			\draw (-3.35,0.43) node[draw, fill=black] (22) { };
			\draw (-2.95,0.43) node[draw] (23) { };
			\draw (-3.13,0.62) node[draw, scale=7.1] (24) { };
			\draw (-2.6,-0.06) node[draw, fill=black] (25) { };
			\draw (-2.8,-0.42) node[draw, fill=black] (26) { };
			\draw (-2.45,-0.42) node[draw, fill=black] (27) { };
			\draw (-2.63,-0.28) node[draw, scale=6] (28){};
			\draw (1.08,0.64) node[draw] (29) {};
			\draw (1.42,0.71) node[draw] (30) {};
			\draw (1.7,0.67) node[draw] (31) {};
			\draw (1.42,0.31) node[draw, fill=black] (32) {};
			\draw (1.42,0.55) node[draw, ellipse, minimum width=60pt, minimum height=40pt] (33) { };
			\draw (-1.78,-0.98) node[draw] (34) {};
			\draw (-1.46,-1.26) node[draw, fill=black] (35) {};
			\draw (-1.43,-0.84) node[draw] (36) { };
			\draw (-1.41,-0.48) node[draw] (37) {  };
			\draw (-1.55,-0.87) node[draw,ellipse, minimum width=40pt, minimum height=65pt] (38) {  };
			\draw (-0.49,-0.43) node[draw] (39) {  };
			\draw (-0.13,-0.64) node[draw, fill=black] (40) {};
			\draw (-0.13,-0.19) node[draw] (41) { };
			\draw (0.16,-0.3) node[draw] (42) { };
			\draw (0.3,0) node[draw] (43) { };
			\draw (-0.01,-0.3) node[draw, ellipse, minimum width=70pt, minimum height=50pt] (44) {};
			\draw  (3) edge (4);
			\draw  (10) edge (12);
			\draw  (11) edge (12);
			\draw  (14) edge (15);
			\draw  (15) edge (16);
			\draw  (22) edge (23);
			\draw  (26) edge (27);
			\draw  (29) edge (32);
			\draw  (30) edge (32);
			\draw  (31) edge (32);
			\draw  (35) edge (36);
			\draw  (36) edge (37);
			\draw  (34) edge (35);
			\draw  (40) edge (41);
			\draw  (39) edge (40);
			\draw  (40) edge (42);
			\draw  (42) edge (43);
            \draw [->] (44) edge (38);
            \draw [->] (44) edge (18);
            \draw [->] (5) edge (18);
            \draw [->] (33) edge (5);
            \draw [->] (38) edge (13);
            \draw [->] (38) edge (18);
            \draw [->] (28) edge (2);
            \draw [->, bend  right] (24) edge (9);
            \draw [->] (44) edge (17);
            \draw [->, bend left] (17) edge (18);
            \draw [->, bend left] (17) edge (5);
            \draw [->] (2) edge (18);
            \draw [->] (13) edge (5);
            \draw [->] (44) edge (33);
            \draw [->] (33) edge (13);
            \draw [->] (13) edge (18);
            \draw [->] (44) edge (13);
            \draw [->] (33) edge (18);
            \draw [->] (44) edge (5);
            \draw[->] (38) edge (17);
            \draw[->, magenta, thick, bend right] (18) edge (44);
            \draw[->, magenta, thick, bend right] (18) edge (5);
            \draw[->, magenta, thick, bend right] (5) edge (44);
            \draw[->, magenta, thick] (2) edge (44);
            \draw[->, magenta, thick] (9) edge (44);
            \draw[->, magenta, thick] (2) edge (13);
            \draw[->, magenta, thick, bend left] (18) edge (13);
            \draw[->, magenta, thick] (9) edge (33);
            \draw[->, magenta, thick] (2) edge (33);
            \draw[->, magenta, thick, bend right] (18) edge (33);
            \draw[->, magenta, thick] (5) edge (38);
            \draw[->, magenta, thick, bend left] (2) edge (38);
            \draw[->, magenta, thick, bend right] (18) edge (38);
            \draw[->, magenta, thick] (28) edge (44);
            \draw[->, magenta, thick] (28) edge (38);
            \draw[->, magenta, thick, bend right] (2) edge (28);
            \draw[->, magenta, thick] (24) edge (44);
            \draw[->, magenta, thick] (9) edge (24);
            \draw[->, magenta, thick] (18) edge (17);
            \draw[->, magenta, thick] (5) edge (17);
		\end{tikzpicture}
        \caption{Part of the double category associated to the tree $Z_{2,2}$.}
        \label{z22im}
    \end{figure}
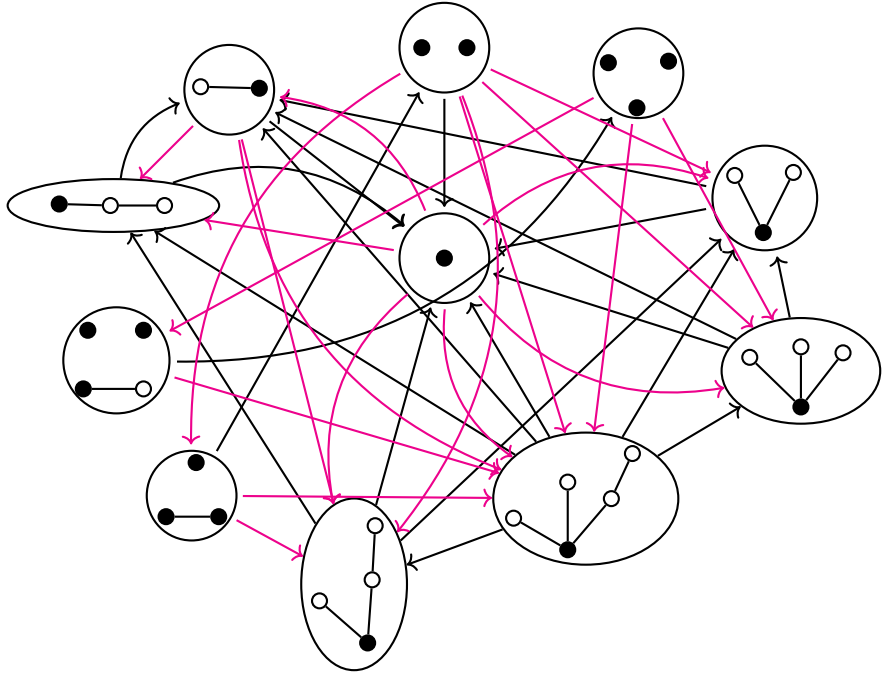
\end{example}

\section{Relationship with the 2-Segal set of the underlying graph} \label{relation}

In this section, we explore the relationship between the 2-Segal set arising from a rooted tree and that arising from the underlying graph. We begin by recalling the following definition of a 2-Segal set associated with a graph $G$ from \cite{2s}.

\begin{definition} 
Let $G$ be a finite graph.  We define a simplicial set $X^G$ associated to $G$ as follows.
\begin{enumerate}
    \item The set $X^G_0$ has a single element that we denote by $\varnothing$.
    
    \item The set $X^G_1$ is the set of all subgraphs of $G$.
    
    \item Any $X^G_n$ has elements $(H; S_1 , \dots, S_n)$ where $H$ is a subgraph of $G$ and the sets $S_1 , \dots , S_n$ form a partition of the set $V(H)$ of vertices into $n$ disjoint (but possibly empty) sets.
    
    \item The face maps $d_i \colon X^G_n \rightarrow X^G_{n-1}$ are defined as 
    $d_i(H; S_1 , \dots, S_n)= (H;, S_1,\dots, S_i\union S_{i+1}, S_{i+2}, \dots, S_n) $ for $1 \leq i \leq n-1$, $d_0(H; S_1 , \dots, S_n)=(H\setminus S_1; S_2, \dots, S_n)$ and $d_n(H; S_1 , \dots, S_n)=(H\setminus S_n; S_1, \dots, S_{n-1})$ where $H\setminus S$ denotes the subgraph of $H$ spanned by $V(H)\setminus S$ for $S = S_1, S_n$.
    
    \item The degeneracy maps  $s_i \colon X^G_n \rightarrow X^G_{n+1}$ are given by 
    \[ s_i(H; S_1 , \dots, S_n)= (H; S_1 , \dots, S_i, \varnothing, S_{i+1}, \dots, S_n). \]
\end{enumerate}
\end{definition}

While in \cite{2s}, only the labelled case is considered, just as in the case of rooted trees and their cuts, we can also consider the case where the graph is unlabelled and we take all unlabelled subgraphs.  In other words, when the graph is unlabelled we consider isomorphism classes of graphs with partitions. 

\begin{example}
    Consider the labelled graph $G$ with vertices $a,b,c$ and edges $(a,b), (b,c)$.  If we consider this graph together with its subgraphs, then $X^G_1$ has 13 subgraphs (one empty, three consisting of a single vertex, five with two vertices and four with three vertices). Then $X^G_2$ has 34 non-degenerate elements, i.e., elements in which both subsets in the partition are nonempty. In Figure \ref{labgraph2} we depict half of them; the remaining half is obtained by exchanging the two parts (that is, the colours) of the partition.
    \begin{figure}[ht]
    \begin{tikzpicture}
    \node[circle, draw, scale=.4, fill=yellow] (0) at (0,0){$a$};
    \node[circle, draw, scale=.4, fill=pink]   (1) at (0,.5){$b$};
    \draw(0)--(1);

    \node[circle, draw, scale=.4, fill=yellow] (0x) at (1,1){$c$};
    \node[circle, draw, scale=.4, fill=pink]   (1x) at (1,.5){$b$};
    \draw(0x)--(1x);

    \node[circle, draw, scale=.4, fill=yellow] (0) at (2,0){$a$};
    \node[circle, draw, scale=.4, fill=pink]   (1) at (2,.5){$b$};

    \node[circle, draw, scale=.4, fill=yellow] (0x) at (3,1){$c$};
    \node[circle, draw, scale=.4, fill=pink]   (1x) at (3,.5){$b$}; 

    \node[circle, draw, scale=.4, fill=yellow] (0x) at (4,0){$a$};
    \node[circle, draw, scale=.4, fill=pink]   (1x) at (4,1){$c$};

    \node[circle, draw, scale=.4, fill=yellow] (0y) at (5,0){$a$};
    \node[circle, draw, scale=.4, fill=pink]   (1y) at (5,0.5){$b$};
    \node[circle, draw, scale=.4, fill=pink]   (2) at (5,1){$c$};
    \draw(0y)--(1y); \draw(1y)--(2);

    \node[circle, draw, scale=.4, fill=yellow] (0yz) at (6,0){$a$};
    \node[circle, draw, scale=.4, fill=yellow] (1yz) at (6,0.5){$b$};
    \node[circle, draw, scale=.4, fill=pink]   (2z) at (6,1){$c$};
    \draw(0yz)--(1yz); \draw(1yz)--(2z);

    \node[circle, draw, scale=.4, fill=yellow] (0yzw) at (7,0){$a$};
    \node[circle, draw, scale=.4, fill=pink] (1yzw) at (7,0.5){$b$};
    \node[circle, draw, scale=.4, fill=yellow]   (2zw) at (7,1){$c$};
    \draw(0yzw)--(1yzw); \draw(1yzw)--(2zw);

    \node[circle, draw, scale=.4, fill=yellow] (30y) at (8,0){$a$};
    \node[circle, draw, scale=.4, fill=pink]   (31y) at (8,0.5){$b$};
    \node[circle, draw, scale=.4, fill=pink]   (32) at  (8,1){$c$};
    \draw(30y)--(31y); 

    \node[circle, draw, scale=.4, fill=yellow] (30yz) at (9,0){$a$};
    \node[circle, draw, scale=.4, fill=yellow] (31yz) at (9,0.5){$b$};
    \node[circle, draw, scale=.4, fill=pink]   (32z) at  (9,1){$c$};
    \draw(30yz)--(31yz); 

    \node[circle, draw, scale=.4, fill=yellow] (3) at (10,0){$a$};
    \node[circle, draw, scale=.4, fill=pink]   (4) at (10,0.5){$b$};
    \node[circle, draw, scale=.4, fill=yellow] (5) at (10,1){$c$};
    \draw(3)--(4);

    \node[circle, draw, scale=.4, fill=yellow] (30y) at (11,0){$a$};
    \node[circle, draw, scale=.4, fill=pink]   (31yw) at (11,0.5){$b$};
    \node[circle, draw, scale=.4, fill=pink]   (32w) at  (11,1){$c$};
    \draw(32w)--(31yw); 

    \node[circle, draw, scale=.4, fill=yellow] (30yz) at (12,0){$a$};
    \node[circle, draw, scale=.4, fill=yellow] (31w) at  (12,0.5){$b$};
    \node[circle, draw, scale=.4, fill=pink]   (32zw) at  (12,1){$c$};
    \draw(32zw)--(31w); 

    \node[circle, draw, scale=.4, fill=yellow] (3) at (13,0){$a$};
    \node[circle, draw, scale=.4, fill=pink]   (4w) at (13,0.5){$b$};
    \node[circle, draw, scale=.4, fill=yellow] (5w) at (13,1){$c$};
    \draw(5w)--(4w);

    \node[circle, draw, scale=.4, fill=yellow] (30y) at  (14,0){$a$};
    \node[circle, draw, scale=.4, fill=pink]   (31yw) at (14,0.5){$b$};
    \node[circle, draw, scale=.4, fill=pink]   (32w) at  (14,1){$c$};
    
    \node[circle, draw, scale=.4, fill=yellow] (30yz) at (15,0){$a$};
    \node[circle, draw, scale=.4, fill=yellow] (31w) at  (15,0.5){$b$};
    \node[circle, draw, scale=.4, fill=pink]   (32zw) at (15,1){$c$};

    \node[circle, draw, scale=.4, fill=yellow] (3) at  (16,0){$a$};
    \node[circle, draw, scale=.4, fill=pink]   (4w) at (16,0.5){$b$};
    \node[circle, draw, scale=.4, fill=yellow] (5w) at (16,1){$c$};
    \end{tikzpicture}
    \caption{Some elements in $X^G_2$, labelled case} \label{labgraph2}
\end{figure}
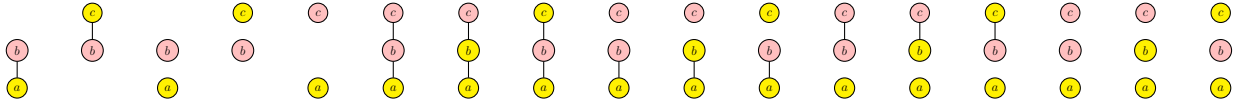
    
There are 24 non-degenerate elements $(H;S_1,S_2,S_3)$ in $X_3^G$, obtained by considering, for each of the four subgraphs $H$ of $G$ that contain all three vertices, the six partitions of the three vertices into three singleton sets $S_1,S_2,S_3$.

If we consider the associated tree we only have, depending on the choice of the root,  six or seven non-degenerate elements $X_1^T$, as we saw in Example \ref{convex}; four or seven in $X^T_2$, as we can see in Figure \ref{comp}, 
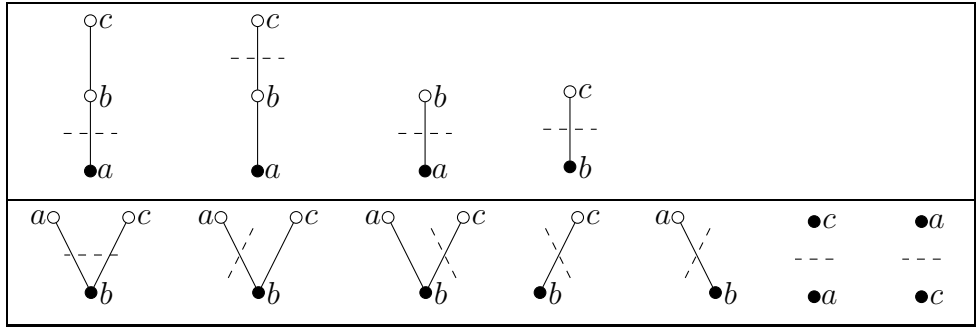
\begin{figure}[ht]
    \centering
    \begin{tabular}{|ccccccc|}
    \hline
    \begin{tikzpicture}
        \node[draw, circle, scale=.4, fill=black] (a) at (0,0){};
        \node[draw, circle, scale=.4] (b) at (0,1){};
        \node[draw, circle, scale=.4] (c) at (0,2){};
        \node at (0.2,0){$a$};
        \node at (0.2,1){$b$};
        \node at (0.2,2){$c$};
        \node (x) at (-0.5, 0.5){};
        \node (y) at (0.5, 0.5){};
            \draw[dashed] (x) to (y);
            \draw (a)--(b);
            \draw (b)--(c);
    \end{tikzpicture} 
    &
    \begin{tikzpicture}
        \node at (0.2,0){$a$};
        \node at (0.2,1){$b$};
        \node at (0.2,2){$c$};
        \node[draw, circle, scale=.4, fill=black] (a) at (0,0){};
        \node[draw, circle, scale=.4] (b) at (0,1){};
        \node[draw, circle, scale=.4] (c) at (0,2){};
        \node (x) at (-0.5, 1.5){};
        \node (y) at (0.5, 1.5){};
        \draw[dashed] (x) to (y);
        \draw (a)--(b);
        \draw (b)--(c);
    \end{tikzpicture} 
    &
    \begin{tikzpicture}
        \node at (0.2,0){$a$};
        \node at (0.2,1){$b$};
        \node[draw, circle, scale=.4, fill=black] (a) at (0,0){};
        \node[draw, circle, scale=.4] (b) at (0,1){};
        \node (x) at (-0.5, 0.5){};
        \node (y) at (0.5, 0.5){};
        \draw[dashed] (x) to (y);
        \draw (a)--(b);
    \end{tikzpicture} 
    &
    \begin{tikzpicture}
        \node at (0.2,1){$b$};
        \node at (0.2,2){$c$};
        \node[draw, circle, scale=.4, fill=black] (b) at (0,1){};
        \node[draw, circle, scale=.4] (c) at (0,2){};
        \node (x) at (-0.5, 1.5){};
        \node (y) at (0.5, 1.5){};
            \draw[dashed] (x) to (y);
            \draw (b)--(c);
    \end{tikzpicture} 
    & & & \\
    \hline
    \begin{tikzpicture}
        \node[draw, circle, scale=.4, fill=black] (b) at (0,0){};
        \node[draw, circle, scale=.4] (a) at (-0.5,1){};
        \node[draw, circle, scale=.4] (c) at (0.5,1){};
        \node at (0.2,0){$b$};
        \node at (-0.7,1){$a$};
        \node at (0.7,1){$c$};
            \draw (a)--(b);
            \draw (b)--(c);
        \node (x) at (0.5,0.5){};
        \node (y) at (-0.5,0.5){};
            \draw[dashed] (x) to (y);
    \end{tikzpicture} 
    &
    \begin{tikzpicture}
        \node at (0.2,0){$b$};
        \node at (-0.7,1){$a$};
        \node at (0.7,1){$c$};
        \node[draw, circle, scale=.4, fill=black] (b) at (0,0){};
        \node[draw, circle, scale=.4] (a) at (-0.5,1){};
        \node[draw, circle, scale=.4] (c) at (0.5,1){};
            \draw (a)--(b);
            \draw (b)--(c);
        \node (x) at (0,1){};
        \node (y) at (-0.5,0){};
            \draw[dashed] (x) to (y);
    \end{tikzpicture} 
    &
    \begin{tikzpicture}
        \node at (0.2,0){$b$};
        \node at (-0.7,1){$a$};
        \node at (0.7,1){$c$};
        \node[draw, circle, scale=.4, fill=black] (b) at (0,0){};
        \node[draw, circle, scale=.4] (a) at (-0.5,1){};
        \node[draw, circle, scale=.4] (c) at (0.5,1){};
            \draw (a)--(b);
            \draw (b)--(c);
        \node (x) at (0,1){};
        \node (y) at (0.5,0){};
            \draw[dashed] (x) to (y);
    \end{tikzpicture} 
    &
    \begin{tikzpicture}
        \node at (0.2,0){$b$};
        \node at (0.7,1){$c$};
        \node[draw, circle, scale=.4, fill=black] (b) at (0,0){};
        \node[draw, circle, scale=.4] (c) at (0.5,1){};
            \draw (b)--(c);
        \node (x) at (0,1){};
        \node (y) at (0.5,0){};
            \draw[dashed] (x) to (y);
    \end{tikzpicture} 
    &
    \begin{tikzpicture}
        \node at (0.2,0){$b$};
        \node at (-0.7,1){$a$};
        \node[draw, circle, scale=.4, fill=black] (b) at (0,0){};
        \node[draw, circle, scale=.4] (a) at (-0.5,1){};
            \draw (a)--(b);
        \node (x) at (0,1){};
        \node (y) at (-0.5,0){};
            \draw[dashed] (x) to (y);
    \end{tikzpicture} 
    &
    \begin{tikzpicture}
        \node at (0.2,0){$a$};
        \node at (0.2,1){$c$};
        \node[draw, circle, scale=.4, fill=black] (a) at (0,0){};
        \node[draw, circle, scale=.4, fill=black] (c) at (0,1){};
        \node (x) at (0.4,0.5){};
        \node (y) at (-0.4,0.5){};
            \draw[dashed] (x) to (y);
    \end{tikzpicture}
    &
    \begin{tikzpicture}
        \node at (0.2,1){$a$};
        \node at (0.2,0){$c$};
        \node[draw, circle, scale=.4, fill=black] (a) at (0,0){};
        \node[draw, circle, scale=.4, fill=black] (c) at (0,1){};
        \node (x) at (0.4,0.5){};
        \node (y) at (-0.4,0.5){};
        \draw[dashed] (x) to (y);
    \end{tikzpicture} \\ \hline 
    \end{tabular}
    \caption{The nondegenerate elements of $X_2^T$ when $T$ has a root of degree one (first row) and when $T$ has a root of degree two (second row).} \label{comp}
\end{figure}
and one or two in $X_3^T$.

Observe in particular that the 2-Segal set obtained from a rooted tree and its cuts is much smaller than the 2-Segal set obtained from the underlying graph and all of its subgraphs, reflecting both the fact that not all subgraphs can be obtained by admissible cuts, but also that we have an inherent ordering on a tree that is not present for the underlying graph.
\end{example}

We now consider an example of this construction applied to an unlabelled graph.

\begin{example}
Now consider the non-planar unlabelled tree $T=K_{1,3}$ with one vertex of degree three and three vertices of degree one. Then in the 2-Segal set $X^G$ associated to the underlying graph $G$ we see that the set $X^G_1$ contains eleven elements: the subgraphs
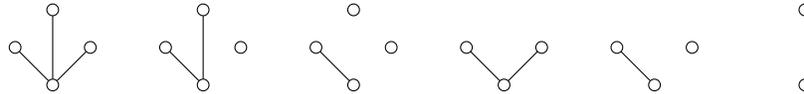
\[ \begin{tikzpicture}
    \node[draw, circle, scale=.4] (0) at (0,0){};
    \node[draw, circle, scale=.4] (1) at (-0.5,0.5){}; 
    \node[draw, circle, scale=.4] (3) at (0.5,0.5){}; 
    \node[draw, circle, scale=.4] (2) at (0,1){}; 
    \draw (0)--(1);
    \draw (0)--(2);
    \draw (0)--(3);

    \node[draw, circle, scale=.4] (01) at (2,0){};
    \node[draw, circle, scale=.4] (11) at (1.5,0.5){}; 
    \node[draw, circle, scale=.4] (31) at (2.5,0.5){}; 
    \node[draw, circle, scale=.4] (21) at (2,1){}; 
    \draw (01)--(11);
    \draw (01)--(21);
 
    \node[draw, circle, scale=.4] (02) at (4,0){};
    \node[draw, circle, scale=.4] (12) at (3.5,0.5){}; 
    \node[draw, circle, scale=.4] (32) at (4.5,0.5){}; 
    \node[draw, circle, scale=.4] (22) at (4,1){}; 
    \draw (02)--(12); 

    \node[draw, circle, scale=.4] (03) at (6,0){};
    \node[draw, circle, scale=.4] (13) at (5.5,0.5){}; 
    \node[draw, circle, scale=.4] (23) at (6.5,0.5){}; 
    \draw (03)--(13);
    \draw (03)--(23);

    \node[draw, circle, scale=.4] (031) at (8,0){};
    \node[draw, circle, scale=.4] (131) at (7.5,0.5){}; 
    \node[draw, circle, scale=.4] (331) at (8.5,0.5){}; 
    \draw (031)--(131);

    \node[draw, circle, scale=.4] (04) at (10,0){};
    \node[draw, circle, scale=.4] (14) at (10,1){}; 
    \draw (04)--(14);

\end{tikzpicture} \]
together with the five subgraphs having no edges and $0\leq i\leq 4$ vertices. In contrast, for the 2-Segal set $X^T$ one can check that $X^T_1$ has cardinality seven. In fact, in this case the cardinality is independent of which vertex we choose to be the root. 

In Figure \ref{unlabelled},
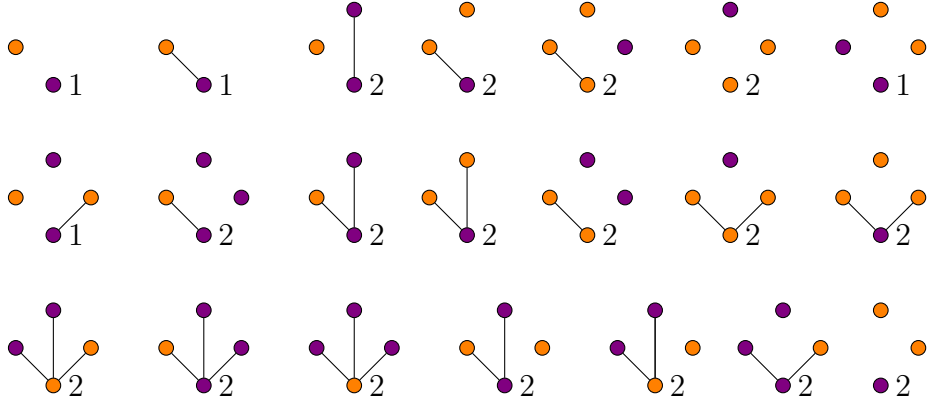
\begin{figure} 
\begin{tikzpicture}
    \node[draw, circle, scale=.5, fill=orange] (0) at (0,0){};
    \node (x) at  (0.3,0){2};
    \node[draw, circle, scale=.5, fill=violet] (1) at (-0.5,0.5){}; 
    \node[draw, circle, scale=.5, fill=orange] (3) at (0.5,0.5){}; 
    \node[draw, circle, scale=.5, fill=violet] (2) at (0,1){}; 
    \draw (0)--(1);
    \draw (0)--(2);
    \draw (0)--(3);

\node (x) at  (2.3,0){2};
    \node[draw, circle, scale=.5, fill=violet] (01) at (2,0){};
    \node[draw, circle, scale=.5, fill=orange] (11) at (1.5,0.5){}; 
    \node[draw, circle, scale=.5, fill=violet] (31) at (2.5,0.5){}; 
    \node[draw, circle, scale=.5, fill=violet] (21) at (2,1){}; 
    \draw (01)--(11);
    \draw (01)--(21);
    \draw (01)--(31);

 \node (x) at  (4.3,0){2};
    \node[draw, circle, scale=.5, fill=orange] (02) at (4,0){};
    \node[draw, circle, scale=.5, fill=violet] (12) at (3.5,0.5){}; 
    \node[draw, circle, scale=.5, fill=violet] (32) at (4.5,0.5){}; 
    \node[draw, circle, scale=.5, fill=violet] (22) at (4,1){}; 
    \draw (02)--(12); 
    \draw (02)--(22);
    \draw (02)--(32);

\node (x) at  (6.3,0){2};
    \node[draw, circle, scale=.5, fill=violet] (03) at (6,0){};
    \node[draw, circle, scale=.5, fill=orange] (13) at (5.5,0.5){}; 
    \node[draw, circle, scale=.5, fill=orange] (23) at (6.5,0.5){};
    \node[draw, circle, scale=.5, fill=violet] (33) at (6,1){}; 
    \draw (03)--(13);
    \draw (03)--(33);
    
\node (x) at  (8.3,0){2};
    \node[draw, circle, scale=.5, fill=orange] (031) at (8,0){};
    \node[draw, circle, scale=.5, fill=violet] (131) at (7.5,0.5){}; 
    \node[draw, circle, scale=.5, fill=orange] (331) at (8.5,0.5){}; 
    \node[draw, circle, scale=.5, fill=violet] (221) at (8,1){}; 
    \draw (031)--(131);
    \draw (031)--(221);
    
\node (x) at  (10,0){2};
    \node[draw, circle, scale=.5, fill=violet] (04) at (9.7,0){};
    \node[draw, circle, scale=.5, fill=violet] (14) at (9.2,0.5){}; 
    \node[draw, circle, scale=.5, fill=orange] (24) at (10.2,0.5){};
    \node[draw, circle, scale=.5, fill=violet] (34) at (9.7,1){}; 
    \draw (04)--(14);
    \draw (04)--(24);

\node (x) at  (11.3,0){2};
    \node[draw, circle, scale=.5, fill=violet] (04) at (11,0){};
    \node[draw, circle, scale=.5, fill=orange]  (14) at (11,1){}; 
    \node[draw, circle, scale=.5, fill=orange] (24) at (11.5,0.5){};

 \node (x1) at  (0.3,2){1};
    \node[draw, circle, scale=.5, fill=violet] (0x) at (0,2){};
    \node[draw, circle, scale=.5, fill=orange]  (1x) at (-0.5,2.5){}; 
    \node[draw, circle, scale=.5, fill=orange] (3x) at (0.5,2.5){}; 
    \node[draw, circle, scale=.5, fill=violet] (2x) at (0,3){}; 
    \draw (0x)--(3x);

\node (x2) at  (2.3,2){2};
    \node[draw, circle, scale=.5, fill=violet]  (01x) at (2,2){};
    \node[draw, circle, scale=.5, fill=orange] (11x) at (1.5,2.5){}; 
    \node[draw, circle, scale=.5, fill=violet] (31x) at (2.5,2.5){}; 
    \node[draw, circle, scale=.5, fill=violet] (21x) at (2,3){}; 
    \draw (01x)--(11x);

 \node (x3) at  (4.3,2){2};
    \node[draw, circle, scale=.5, fill=violet]  (02x) at (4,2){};
    \node[draw, circle, scale=.5, fill=orange] (12x) at (3.5,2.5){}; 
    \node[draw, circle, scale=.5, fill=violet] (22x) at (4,3){}; 
    \draw (02x)--(12x); 
    \draw (02x)--(22x);

\node (x4) at  (5.8,2){2};
    \node[draw, circle, scale=.5, fill=violet] (03x) at (5.5,2){};
    \node[draw, circle, scale=.5, fill=orange]  (13x) at (5,2.5){}; 
    \node[draw, circle, scale=.5, fill=orange] (23x) at (5.5,3){};
    \draw (03x)--(13x);
    \draw (03x)--(23x);

\node (x5) at  (7.4,2){2};
    \node[draw, circle, scale=.5, fill=orange]  (031x) at (7.1,2){};
    \node[draw, circle, scale=.5, fill=orange] (131x) at (6.6,2.5){}; 
    \node[draw, circle, scale=.5, fill=violet] (331x) at (7.1,3){}; 
    \node[draw, circle, scale=.5, fill=violet] (221x) at (7.6,2.5){}; 
    \draw (031x)--(131x);
    \draw (031)--(221);

\node (x6) at  (9.3,2){2};
    \node[draw, circle, scale=.5, fill=orange]  (3xy) at (9,2){};
    \node[draw, circle, scale=.5, fill=orange] (1xy) at (8.5,2.5){}; 
    \node[draw, circle, scale=.5, fill=violet] (31x) at (9,3){}; 
    \node[draw, circle, scale=.5, fill=orange] (21x) at (9.5,2.5){}; 
    \draw (1xy)--(3xy);
    \draw (3xy)--(21x);
    
\node (x7) at  (11.3,2){2};
    \node[draw, circle, scale=.5, fill=violet] (04x) at (11,2){};
    \node[draw, circle, scale=.5, fill=orange]  (14x) at (10.5,2.5){}; 
    \node[draw, circle, scale=.5, fill=orange] (24x) at (11.5,2.5){};
    \node[draw, circle, scale=.5, fill=orange] (34x) at (11,3){}; 
    \draw (04x)--(14x);
    \draw (04x)--(24x);

\node (yx1) at  (0.3,4){1};
    \node[draw, circle, scale=.5, fill=violet] (0xy) at (0,4){};
    \node[draw, circle, scale=.5, fill=orange] (1xy) at (-0.5,4.5){}; 

\node (yx2) at  (2.3,4){1};
    \node[draw, circle, scale=.5, fill=violet] (01xy) at (2,4){};
    \node[draw, circle, scale=.5, fill=orange] (11xy) at (1.5,4.5){}; 
    \draw (01xy)--(11xy);
   
 \node (yx3) at  (4.3,4){2};
    \node[draw, circle, scale=.5, fill=violet] (02xy) at (4,4){};
    \node[draw, circle, scale=.5, fill=orange] (12xy) at (3.5,4.5){}; 
    \node[draw, circle, scale=.5, fill=violet] (22xy) at (4,5){}; 
    \draw (02xy)--(22xy);

\node (yx4) at  (5.8,4){2};
    \node[draw, circle, scale=.5, fill=violet] (03xy) at (5.5,4){};
    \node[draw, circle, scale=.5, fill=orange] (13xy) at (5,4.5){}; 
    \node[draw, circle, scale=.5, fill=orange] (23xy) at (5.5,5){};
    \draw (03xy)--(13xy);

\node (yx5) at  (7.4,4){2};
    \node[draw, circle, scale=.5, fill=orange] (031xy) at (7.1,4){};
    \node[draw, circle, scale=.5, fill=orange] (131xy) at (6.6,4.5){}; 
    \node[draw, circle, scale=.5, fill=orange] (331xy) at (7.1,5){}; 
    \node[draw, circle, scale=.5, fill=violet] (221xy) at (7.6,4.5){}; 
    \draw (031xy)--(131xy);

\node (yx6) at  (9.3,4){2};
    \node[draw, circle, scale=.5, fill=orange] (y3xy) at (9,4){};
    \node[draw, circle, scale=.5, fill=orange] (y1xy) at (8.5,4.5){}; 
    \node[draw, circle, scale=.5, fill=violet] (y31x) at (9,5){}; 
    \node[draw, circle, scale=.5, fill=orange] (y21x) at (9.5,4.5){}; 
       
\node (yx7) at  (11.3,4){1};
    \node[draw, circle, scale=.5, fill=violet] (04xy) at (11,4){};
    \node[draw, circle, scale=.5, fill=violet] (14xy) at (10.5,4.5){}; 
    \node[draw, circle, scale=.5, fill=orange] (24xy) at (11.5,4.5){};
    \node[draw, circle, scale=.5, fill=orange] (34xy) at (11,5){}; 
\end{tikzpicture} 
\caption{Some elements in $X^G_2$, unlabelled case} \label{unlabelled}
\end{figure}
we depict the nondegenerate elements of $X_2^G$. The partition of the vertices is indicated by the two colours. The number below each element indicates how many distinct (that is, non-isomorphic) elements can by obtained by permuting the colours.  Thus, there are 38 elements for which both parts of the partition are non-empty. In  $X_2^T$, however, the number of elements corresponding to a subtree with a nontrivial cut is either six or eight, depending on the choice of the root.  
   
In Figure \ref{nondeg}, 
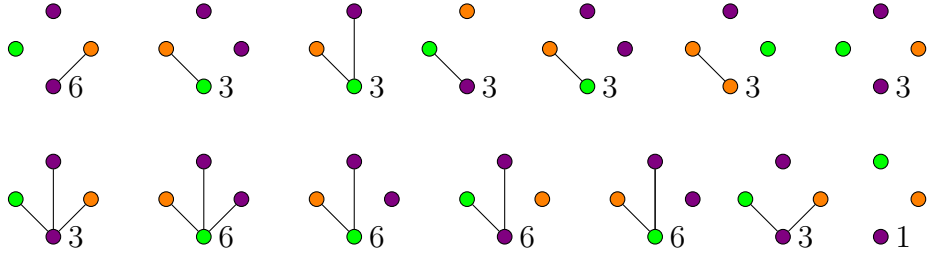
\begin{figure}[ht] \begin{tikzpicture}
    \node[draw, circle, scale=.5, fill=violet] (0) at (0,0){};
    \node (x) at  (0.3,0){3};
    \node[draw, circle, scale=.5, fill=green] (1) at (-0.5,0.5){}; 
    \node[draw, circle, scale=.5, fill=orange] (3) at (0.5,0.5){}; 
    \node[draw, circle, scale=.5, fill=violet] (2) at (0,1){}; 
    \draw (0)--(1);
    \draw (0)--(2);
    \draw (0)--(3);

\node (x) at  (2.3,0){6};
    \node[draw, circle, scale=.5, fill=green] (01) at (2,0){};
    \node[draw, circle, scale=.5, fill=orange] (11) at (1.5,0.5){}; 
    \node[draw, circle, scale=.5, fill=violet] (31) at (2.5,0.5){}; 
    \node[draw, circle, scale=.5, fill=violet] (21) at (2,1){}; 
    \draw (01)--(11);
    \draw (01)--(21);
    \draw (01)--(31);

 \node (x) at  (4.3,0){6};
    \node[draw, circle, scale=.5, fill=green] (02) at (4,0){};
    \node[draw, circle, scale=.5, fill=orange] (12) at (3.5,0.5){}; 
    \node[draw, circle, scale=.5, fill=violet] (32) at (4.5,0.5){}; 
    \node[draw, circle, scale=.5, fill=violet] (22) at (4,1){}; 
    \draw (02)--(12); 
    \draw (02)--(22);

\node (x) at  (6.3,0){6};
    \node[draw, circle, scale=.5, fill=violet] (03) at (6,0){};
    \node[draw, circle, scale=.5, fill=green] (13) at (5.5,0.5){}; 
    \node[draw, circle, scale=.5, fill=orange] (23) at (6.5,0.5){};
    \node[draw, circle, scale=.5, fill=violet] (33) at (6,1){}; 
    \draw (03)--(13);
    \draw (03)--(33);

\node (x) at  (8.3,0){6};
    \node[draw, circle, scale=.5, fill=green] (031) at (8,0){};
    \node[draw, circle, scale=.5, fill=orange] (131) at (7.5,0.5){}; 
    \node[draw, circle, scale=.5, fill=violet] (331) at (8.5,0.5){}; 
    \node[draw, circle, scale=.5, fill=violet] (221) at (8,1){}; 
    \draw (031)--(131);
    \draw (031)--(221);
    
\node (x) at  (10,0){3};
    \node[draw, circle, scale=.5, fill=violet] (04) at (9.7,0){};
    \node[draw, circle, scale=.5, fill=green] (14) at (9.2,0.5){}; 
    \node[draw, circle, scale=.5, fill=orange] (24) at (10.2,0.5){};
    \node[draw, circle, scale=.5, fill=violet] (34) at (9.7,1){}; 
    \draw (04)--(14);
    \draw (04)--(24);

\node (x) at  (11.3,0){1};
    \node[draw, circle, scale=.5, fill=violet] (04) at (11,0){};
    \node[draw, circle, scale=.5, fill=green]  (14) at (11,1){}; 
    \node[draw, circle, scale=.5, fill=orange] (24) at (11.5,0.5){};

 \node (x1) at  (0.3,2){6};
    \node[draw, circle, scale=.5, fill=violet] (0x) at (0,2){};
    \node[draw, circle, scale=.5, fill=green]  (1x) at (-0.5,2.5){}; 
    \node[draw, circle, scale=.5, fill=orange] (3x) at (0.5,2.5){}; 
    \node[draw, circle, scale=.5, fill=violet] (2x) at (0,3){}; 
    \draw (0x)--(3x);

\node (x2) at  (2.3,2){3};
    \node[draw, circle, scale=.5, fill=green]  (01x) at (2,2){};
    \node[draw, circle, scale=.5, fill=orange] (11x) at (1.5,2.5){}; 
    \node[draw, circle, scale=.5, fill=violet] (31x) at (2.5,2.5){}; 
    \node[draw, circle, scale=.5, fill=violet] (21x) at (2,3){}; 
    \draw (01x)--(11x);

 \node (x3) at  (4.3,2){3};
    \node[draw, circle, scale=.5, fill=green]  (02x) at (4,2){};
    \node[draw, circle, scale=.5, fill=orange] (12x) at (3.5,2.5){}; 
    \node[draw, circle, scale=.5, fill=violet] (22x) at (4,3){}; 
    \draw (02x)--(12x); 
    \draw (02x)--(22x);

\node (x4) at  (5.8,2){3};
    \node[draw, circle, scale=.5, fill=violet] (03x) at (5.5,2){};
    \node[draw, circle, scale=.5, fill=green]  (13x) at (5,2.5){}; 
    \node[draw, circle, scale=.5, fill=orange] (23x) at (5.5,3){};
    \draw (03x)--(13x);

\node (x) at  (7.4,2){3};
    \node[draw, circle, scale=.5, fill=green]  (031x) at (7.1,2){};
    \node[draw, circle, scale=.5, fill=orange] (131x) at (6.6,2.5){}; 
    \node[draw, circle, scale=.5, fill=violet] (331x) at (7.1,3){}; 
    \node[draw, circle, scale=.5, fill=violet] (221x) at (7.6,2.5){}; 
    \draw (031x)--(131x);
    \draw (031)--(221);

\node (x) at  (9.3,2){3};
    \node[draw, circle, scale=.5, fill=orange]  (3xy) at (9,2){};
    \node[draw, circle, scale=.5, fill=orange] (1xy) at (8.5,2.5){}; 
    \node[draw, circle, scale=.5, fill=violet] (31x) at (9,3){}; 
    \node[draw, circle, scale=.5, fill=green] (21x) at (9.5,2.5){}; 
    \draw (1xy)--(3xy);
    
\node (x) at  (11.3,2){3};
    \node[draw, circle, scale=.5, fill=violet] (04x) at (11,2){};
    \node[draw, circle, scale=.5, fill=green]  (14x) at (10.5,2.5){}; 
    \node[draw, circle, scale=.5, fill=orange] (24x) at (11.5,2.5){};
    \node[draw, circle, scale=.5, fill=violet] (34x) at (11,3){}; 
 \end{tikzpicture} 
 \caption{The non-degenerate elements of $X_3^G$.} \label{nondeg}
\end{figure}
we illustrate the non-degenerate elements of $X^G_3$, where as before the partition is indicated by the colouring and we record the number of distinct elements that can be obtained by permuting the parts. Thus $X^G_3$ has 55 non-degenerate elements. However,  if we consider the graph as a rooted tree, the set $X^T_3$ has only three or four non-degenerate elements (depending on the choice of the root). The 2-Segal set associated to a tree viewed as a graph is substantially larger than the 2-Segal set obtained from admissible cuts of the tree itself.
 
Finally observe that $X_4^G$ has just four non-degenerate elements, determined by which of the four singleton parts has the vertex of degree 3, while $X_4^T$ has just one.
\end{example}

\begin{lemma} \label{lem:tree-graph-simplicial}
Consider the 2-Segal sets $X^{T}$ and $X^{G}$ associated to a rooted tree $T$ and to the underlying graph $G=U(T)$.  There is a simplicial map $U \colon X^{T}\to X^{G}$ that sends a subforest $F$ with a layering of $n-1$ cuts to the underlying subgraph with the vertex partition defined by the layers
\[ U(F;\;L_{0}\supseteq\cdots\supseteq L_{n}) =(H;S_{1},\dots,S_{n}), \]
where $L_{0}=V(F)$, $L_{n}=\varnothing$, $H=U(F)$ and $S_{i}=L_{i-1}\setminus L_{i}$ for $i=1,\dots,n$.
\end{lemma}

\begin{proof}
We check that $U \colon X^{T}\to X^{G}$ preserves the simplicial degeneracy and face maps
\begin{equation} \label{simp-nat}
\begin{tikzcd}X^{T}_{n}\ar[d,"s_{i}"']\ar[r,"U_{n}"]&X^{G}_{n}\ar[d,"s_{i}"] \\ 
    X^{T}_{n+1} \ar[r,"U_{n+1}"]&X^{G}_{n+1}
\end{tikzcd}
\qquad\qquad\qquad
\begin{tikzcd}X^{T}_{n}\ar[d,"d_{i}"']\ar[r,"U_{n}"]&X^{G}_{n}\ar[d,"d_{i}"] \\ 
    X^{T}_{n-1} \ar[r,"U_{n-1}"]&X^{G}_{n-1}.
\end{tikzcd}
\end{equation}
For the degeneracy maps $s_{i}$ we have
\[ s_{i}U_{n}(F;L_{0}\supseteq\cdots\supseteq L_{n}) =s_{i}(H;S_{1},\dots,S_{n})=(H;S_{1},\dots,S_{i},\varnothing, S_{i+1},\dots,S_{n}) \]
and
\[ U_{n-1}s_{i}(F;L_{0}\supseteq\cdots\supseteq L_{n})=
    U_{n-1}(F;L_{0}\supseteq \cdots L_{i}\supseteq L_{i}\supseteq\cdots\supseteq L_{n}), \]
which agree.

Consider now the face maps $d_{i}$. Since $X^{T}_{0}$ and $X^{G}_{0}$ are singletons we can assume $n\geq 2$. For $1\le i\le n-1$ we have
\[ d_{i}U_{n}(F;L_{0}\supseteq\cdots\supseteq L_{n})=d_{i}(H;S_{1},\dots,S_{n})=(H;S_{1},\dots,S_{i}\union S_{i+1},\dots,S_{n}) \]
and    
\[ U_{n-1}d_{i}(F;L_{0}\supseteq\cdots\supseteq L_{n})= U_{n-1}(F;L_{0}\supseteq \cdots L_{i-1}\supseteq L_{i+1}\supseteq\cdots\supseteq L_{n}). \]
They agree since $S_{i}\union S_{i+1}=(L_{i-1}\setminus L_{i})\union(L_{i}\setminus L_{i+1})=L_{i-1}\setminus L_{i+1}$. In the case when $i=0$, we have
\[ d_{0}U_{n}(F;L_{0}\supseteq\cdots\supseteq L_{n})=d_{0}(H;S_{1},\dots,S_{n})=(H\setminus S_{1};\;S_{2},\dots,S_{n}) \]
and
\[ U_{n-1}d_{0}(F;L_{0}\supseteq\cdots\supseteq L_{n}) = U_{n-1}(F|L_1;\;L_{1} \supseteq \cdots\supseteq L_{n}), \]
which agree since $H\setminus S_{1}=U(F|L_{1})$: removing the vertices of $S_{1}=L_{0}\setminus L_{1}$ from $H$ gives the underlying graph of the forest $F|L_1$ defined by restricting $F$ to the vertices $L_{1}$.

In the case when $i=n$, we have
\[ d_{n}U_{n}(F;L_{0}\supseteq\cdots\supseteq L_{n})=d_{n}(H;S_{1},\dots,S_{n})=(H\!\setminus\! S_{n};\;S_{1},\dots,S_{n-1}) \]
and
\[ U_{n-1}d_{n}(F;L_{0}\supseteq\cdots\supseteq L_{n})= U_{n-1}(F|(L_{0}\!\setminus\! L_{n-1});\;(L_{0}\!\setminus\! L_{n-1})\supseteq\cdots\supseteq (L_{n-1}\!\setminus\! L_{n-1})), \]
which agree since $S_{i}=L_{i-1}\!\setminus\!L_{i}=(L_{i-1}\!\setminus\!L_{n-1})\setminus(L_{i}\!\setminus\!L_{n-1})$ for $i=1,\dots,n-1$, and removing the vertices of $S_{n}=L_{n-1}\!\setminus\! \varnothing$ from $H$ gives the underlying graph of $F$ restricted to the vertices $L_{0}\!\setminus\!L_{n-1}$.
\end{proof}

Observe that the simplicial map $U \colon X^{T}\to X^{G}$ is bijective in simplicial degree zero,  as $X_0^{T}$ and $X_0^{G}$ are both singletons.  However it is not CULF, nor it is relatively Segal, as we explain shortly. Intuitively these negative results say that decompositions of underlying graphs cannot always be uniquely lifted to decompositions of trees.

Let us define these two notions more formally.

\begin{definition} \label{culfdef}
A map of simplicial sets $X\to Y$ is \emph{conservative with unique lifting of factorizations}, or simply \emph{CULF}, if the squares
\[  \begin{tikzcd}X_{n}\ar[d,"s_j"']\ar[r]&Y_{n}\ar[d,"s_j"]
      \\ X_{n+1}      \ar[r]&Y_{n+1},
\end{tikzcd}
\qquad\qquad
\begin{tikzcd}X_{n}\ar[d,"d_{i}"']\ar[r]&Y_{n}\ar[d,"d_{i}"]
      \\ X_{n-1}      \ar[r]&Y_{n-1}
\end{tikzcd} \]
are pullbacks for each $0<i<n$ and each $0\leq j\leq n$.
\end{definition}

By \cite[Lemma 4.1]{dec} checking the CULF property for a map of simplicial sets $U:X^T\to X^G$ is equivalent to checking the following diagrams are pullbacks
\[ \begin{tikzcd}X^T_{n}\ar[d,"\lambda_n"']\ar[r,"U_n"]&X^G_{n}\ar[d,"\lambda_n"] \\ 
    X^T_{1} \ar[r,"U_{1}"]&X^G_{1}.
    \end{tikzcd} \]
for the vertical arrows $\lambda_0=s_0$ and $\lambda_n=d_1^{n-1}$ (where $n\geq2)$. The case $n=1$ is immediate as $\lambda_1=\id$.   Equivalently, we must check that the function $U_{n}\colon X^T_n\to X^G_n$ restricts to bijections between the fibres $\lambda_n^{-1}(H)$ and $\lambda_n^{-1}(U(H))$.  For $n=0$, the fibres are singletons if $H$ is the empty subforest, and are empty otherwise.  In other words, $U$ is conservative. For $n\geq2$ the map $\lambda_n=d_1^{n-1}$ is the function that forgets the cuts (or the partition of vertices) and returns just the subforest (or subgraph), and the condition says that any partition of a graph underlying an admissible subforest $H$ must arise from a unique layering of cuts of $H$. This idea is the motivation behind the terminology ``unique lifting of factorizations".

\begin{example}
Consider 
\begin{equation}\label{eq:TG-not-culf}
T=\begin{tikzpicture}[baseline=2ex,level distance=5mm, sibling distance=5mm, grow'=up, every node/.style={circle,draw,inner sep=1.5pt}]
\node[fill=black,"$b$" left ]{} child{node["$a$" left ]{}                            }
                     child{node["$c$" right]{} child{node["$d$" right]{}} };
\end{tikzpicture}\qquad\qquad\quad
G=U(T)=\begin{tikzpicture}[level distance=5mm,  grow'=right, every node/.style={draw,circle,inner sep=1.5pt}]
\node["$a$"]{}child{node["$b$"]{}child{node["$c$"]{}child{node["$d$"]{}}}};
\end{tikzpicture}.
\end{equation}
Then we see that $U$ fails to be CULF as the element $(G;\{a,c\}, \{b,d\})$ of $d_{1}^{-1}(G)$ is not in the image of $U_{2}$: the graph can be partitioned into two subgraphs containing no edges, but the tree cannot. 
\end{example}

The above example shows that $U \colon X^T\to X^G$ is not CULF in general; the same argument applies also to the labelled and unlabelled contexts.

\begin{prop}
   Let $T$ be a rooted tree and $G=U(T)$ its underlying graph. Then the induced map on 2-Segal sets $U\colon X^T \rightarrow X^G$ is CULF if and only if 
    \begin{itemize}
        \item $T$ has just one vertex, in the case that $T$ is a labelled tree, and
        
        \item $T$ has at most one edge, in the case that $T$ is an unlabelled tree.
    \end{itemize}
\end{prop}

\begin{proof}
If $T$ is the tree with just one vertex then $U\colon X^T\to X^G$ is CULF as it is an isomorphism. 
    
Suppose that $T$ is a labelled tree with more than one vertex. Then choosing an edge we have an admissible subforest
\[ H=\left(\;\begin{tikzpicture}[baseline=1.5mm,level distance=5mm, grow'=up, every node/.style={circle,draw,inner sep=1.5pt}]
\node[fill=black,"{$a$}" right ]{} child{node["$b$" right ]{} };
\end{tikzpicture}
\right)\in X_1^T.\]
The fibres of $d_1$ over $H$ and $UH$ have 3 and 4 elements respectively:
    \begin{align*}
    d_1^{-1}(H)&=\{ (H\supseteq \{a,b\}\supseteq \varnothing),
    (H\supseteq\{a\}\supseteq \varnothing),
    (H\supseteq\varnothing\supseteq \varnothing)
    \} \subseteq X^T_2,\\
    d_1^{-1}(UH)&=\{(UH;\varnothing,\{a,b\}),(UH;\{a\},\{b\}),(UH;\{b\},\{a\}),(UH;\{a,b\},\varnothing)\} \subseteq X^G_2.    
    \end{align*}
Therefore $U\colon X^T\to X^G$ is not CULF.
    
Now assume that $T$ is an unlabelled tree with at least two edges. Then $G=U(T)$ contains a subgraph $UH=
\begin{tikzpicture}
    \node[circle, scale=.4, draw] (0) at (0,0){}; \node[circle, scale=.4, draw] (1) at (0.7,0){}; \node[circle, scale=.4, draw] (2) at (1.4,0){};
    \draw(0)--(1); \draw(1)--(2); 
    \end{tikzpicture}$, 
and the  fibre $d_1^{-1}(UH)\subseteq X_2^G$ is given by six partitions 
    \def\colours#1#2#3{
    \begin{tikzpicture}
        \node[circle, scale=.4, draw,fill=#1] (0) at (0,0){}; \node[circle, scale=.4, draw,fill=#2] (1) at (0.7,0){}; \node[circle, scale=.4, draw,fill=#3] (2) at (1.4,0){};
        \draw(0)--(1); \draw(1)--(2); 
    \end{tikzpicture}}
\[\qquad
\colours{violet}{violet}{violet}
,\qquad\quad
\colours{orange}{orange}{orange}
,\qquad\quad
\colours{orange}{violet}{orange}
,\quad\;\;
\colours{violet}{orange}{violet}
,\quad\;\;
\colours{violet}{violet}{orange}
,\quad
\colours{violet}{orange}{orange}
   \quad \]
\[ (UH; V(UH), \varnothing), (UH; \varnothing, V(UH)), (UH; \circ, \circ \circ), (UH; \circ \circ, \circ), (UH;\begin{tikzpicture}
        \node[circle, scale=.4, draw] (0) at (0,0){}; \node[circle, scale=.4, draw] (1) at (0.5,0){}; \draw(0)--(1);  
    \end{tikzpicture}, \circ ), (UH; \circ, \begin{tikzpicture}
        \node[circle, scale=.4, draw] (0) at (0,0){}; \node[circle, scale=.4, draw] (1) at (0.5,0){}; \draw(0)--(1);  
    \end{tikzpicture} ). \] 
However, the fibre of $d_1^{-1}(H)\subseteq X_2^T$ contains either four elements,
\[ (H \supset \varnothing \supseteq \varnothing), 
    \qquad(H \supseteq H \supset \varnothing), 
    \qquad(H \supset \bullet \supset \varnothing), 
    \qquad(H \supset \begin{tikzpicture}[baseline=0.3ex]  \node[circle, scale=.4, draw, fill=black] (0) at (0,0){}; \node[circle, scale=.4, draw] (1) at (0,0.4){}; \draw(0)--(1);   \end{tikzpicture} \supset \varnothing), \]
or, in the case that $T$ is planar and the root of $H$ has degree two (when the last element listed can arise in two non-isomorphic ways) five elements. In either case we see that $U\colon X^T\to X^G$ is not CULF.

Finally we consider the unlabelled rooted tree $T$ with just one edge and two vertices and show that, for each admissible subtree $H$, the functions $U_n$ restrict to bijections between $\lambda_n^{-1}(H)$, the $n$-layerings of $H$, and $\lambda_n^{-1}(UH)$, the $n$-partitions of the vertices of the underlying graph $UH$.  In the case $H=T$ we observe that the two vertices of $UH$ are indistinguishable up to an isomorphism of the graph, and so there are just $n+\binom n2$ partitions of the vertex set into $n$ parts. These partitions correspond under $U_n$ to the $n+\binom n2$ layerings by $n-1$ cuts of $T$.  

In the cases where $H$ has just one vertex or $H=\varnothing$ the result is immediate.
\end{proof}

Let us now turn to the definition of relatively Segal.

\begin{definition} \label{relSegal}
A map of simplicial sets $X \rightarrow Y$ is \emph{relatively Segal} if the square
\[ \begin{tikzcd}X_{n} \ar[d] \ar[r] & Y_{n} \ar[d] \\ 
    X_{1} \times \dots \times X_{1} \ar[r] & Y_{1}\times\dots\times Y_{1},
\end{tikzcd} \]
in which the vertical arrows are the Segal maps, is a pullback for each $n \geq 2$.
\end{definition}

\begin{prop}
     Let $T$ be a rooted tree and $G=U(T)$ its underlying graph. Then the induced map on 2-Segal sets $U\colon X^T \rightarrow X^G$ is  relatively Segal if and only if $T$ has just one vertex.
\end{prop}

\begin{proof}
If $T$ is the tree with just one vertex then $U\colon X^T\to X^G$ is relatively Segal as it is an isomorphism. 
Suppose $T$ has more than one vertex, so that it has an admissible subforest   
\[ H=\left(\;\begin{tikzpicture}[baseline=1.5mm,level distance=5mm, grow'=up, every node/.style={circle,draw,inner sep=1.5pt}]
\node[fill=black,"{$a$}" right ]{} child{node["$b$" right ]{} };
\end{tikzpicture}\right)\in X_1^T.\]
Then in all cases (labelled or not, and planar or not) the diagram
\[ \begin{tikzcd}X^{T}_{2}\ar[d,"{(d_0,d_2)}"']\ar[rr,"U_{2}"]&&X^{G}_{2}\ar[d,"{(d_0,d_2)}"] \\ X^{T}_{1}\times X^{T}_{1}
    \ar[rr,"U_{1}\times U_1"]&& X^{G}_{1}\times X^{G}_{1}
    \end{tikzcd} \]
fails to be a pullback: observe that the subgraph of $S$ of $G$ containing both vertices of $H$, but not the edge, does not arise from an admissible subtree of $T$. Therefore the element $\sigma\in X_2^G$ given by a partition of $S$ into two non-empty parts is not in the image of $U_2$, though $(d_0\sigma,d_2\sigma)$ is clearly in the image of $U_1\times U_1$.  
\end{proof}

In the next section, we discuss the consequences of these results for Hall algebras.

\section{The associated Hall algebra}\label{hall}

In this section, we consider the Hall algebra associated with the 2-Segal set obtained from a rooted tree.  Let us begin by recalling the following definition from \cite{dk}; here a 2-Segal set $X$ is \emph{reduced} if $X_0$ consists of a single point.

\begin{definition} \label{defhall}
Let $X$ be a reduced 2-Segal set with finitely many non-degenerate simplices in each degree, and $k$ a field. The \emph{Hall algebra} $\mathcal H_k(X)=\mathcal  H(X)$ associated to $X$ has underlying $k$-vector space spanned by the elements of $X_1$.  If $1_b$ denotes the basis vector corresponding to $b \in X_1$, then the multiplication in $\mathcal H(X)$ is defined by
\[ 1_b \ast 1_{b'} = \sum_{b'' \in X_1} c_{bb'}^{b''} 1_{b''} \]
where $c_{bb'}^{b''}$ is the number of elements $c \in X_2$ such that $d_0(c)=b$, $d_2(c)=b'$, and $d_1(c)=b''$.  If $\ast$ denotes the single element of $X_0$, then the unit of $\mathcal H(X)$ is $1_{s_0(\ast)}$.
\end{definition}

Let us start by considering some specific examples of trees.  If $T$ is a rooted tree, we denote by $\mathcal H^T$ the Hall algebra associated to $X^T$.

\begin{example}
    Let $T=F_1$, the tree with a single vertex and no edges. Since the only nonempty subtree is the tree $T$ itself, $\mathcal H^T$ is a 2-dimensional vector space with basis $\{1_\varnothing,1_T\}$, multiplicative identity $1_\varnothing$ and $1_T\ast 1_T=0$.
\end{example}

\begin{example}
    Let $T$ be the labelled tree with root vertex $a$, one other vertex $b$, and a single edge between them.  Then $\mathcal H^T$ is the 4-dimensional vector space with basis $\{T, \bullet_a, \bullet_b, \varnothing\}$.  The only nontrivial multiplication is 
    \[ [\bullet_b][\bullet_a] = [T]. \]
    In particular, observe that 
    \[ [\bullet_a][\bullet_b]=0 \] 
    due to the ordering on the vertices inherent in the tree structure, so in particular $\mathcal H^T$ is not commutative.
\end{example}

By the same reasoning all nontrivial labelled trees give noncommutative Hall algebras.  For unlabelled trees, although the corresponding Hall algebras still need not be commutative, we can show that they are commutative for a certain family of trees classified by the following result.

\begin{prop} \label{comu}
  The Hall algebra associated to a unlabelled rooted tree $T$ is commutative if and only if $T$ has no vertices of degree higher than two and its root has degree one, in other words if $T$ is a path tree whose root vertex has degree one.  
\end{prop}

\begin{proof}
    By Remark \ref{rem1}, if $T$ is a tree with a vertex of degree $n \geq 3$ (or $n \geq 2$ if it is the root vertex, then the forest $F_n$ with $n$ vertices and no edges is an element in $X_1$. Notice that this forest is always the top layer of the tree, so in particular it is never the bottom layer and $F_n\neq d_0(C)$ for all $C \in X_2.$  It follows that $1_{F_n} \ast 1_Y =0$ for every $Y \in X_1.$ However since $F_n = d_2(Z)$ for some $Z \in X_1$, we have that $1_{d_0(Z)} \ast F_n= 1_Z$, showing that $\mathcal{H}^T$ is noncommutative.

    Now let $P_n$ denote the unlabelled path tree on $n$ vertices with root a vertex of degree one. Then the elements of $X_1$ are the path graphs $P_m$ for $m\leq n$. Denote by $1_m$ the basis vector corresponding to $P_m$. Then $1_i \ast 1_j = 1_{i+j}$ if $i+j \leq n$ and $1_i \ast 1_j = 0$ if $i+j >n$. Thus $\mathcal{H}^{P_n}$ is commutative.
\end{proof}

While the associated Hall algebra is not commutative in general for trees having vertices of degree higher than two, if the tree is symmetric enough there are elements that do commute with each other.

\begin{example}
Consider the tree $T$ depicted in Figure \ref{HT} 
\begin{figure}
    \centering
       \begin{tikzpicture}
           \node[circle, draw, scale=.5, fill=black] (0) at (0,0){};
           \node[circle, draw, scale=.5] (1) at (0,1){};
           \node[circle, draw, scale=.5] (2) at (0,2){};
           \node[circle, draw, scale=.5] (3) at (0,3){};
           \node[circle, draw, scale=.5] (4) at (0,4){};
           \node[circle, draw, scale=.5] (5) at (0,5){};
           \node[circle, draw, scale=.5] (01) at (1,0){};
           \node[circle, draw, scale=.5] (21) at (1,2){};
           \node[circle, draw, scale=.5] (41) at (1,4){};
           \node[circle, draw, scale=.5] (10) at (-1,0){};
           \node[circle, draw, scale=.5] (12) at (-1,2){};
           \node[circle, draw, scale=.5] (14) at (-1,4){};
           \node (x) at (1, 1.5){};
           \node (y) at (1,3.5){};
           \node (x0) at (-1, 1.5){$c_1$};
           \node (y0) at (-1,3.5){$c_2$};
        \draw (1)--(2);
        \draw (2)--(3);
        \draw (0)--(1);
        \draw (3)--(4);
        \draw (4)--(5);
        \draw (0)--(01);
        \draw (0)--(10);
        \draw (2)--(21);
        \draw (2)--(12);
        \draw (4)--(41);
        \draw (4)--(14);
        \draw[dashed] (y)--(y0);
        \draw[dashed] (x)--(x0);
       \end{tikzpicture}
       \caption{A tree $T$ such that $\mathcal{H}^T$ has two elements which commute.}
       \label{HT}
   \end{figure}
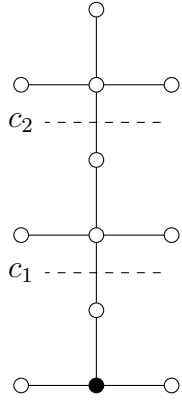
where the root is the black vertex. Let $c_1$ and $c_2$ be cuts of the tree and take the elements $e_i=(T,c_i)\in X_2$. Then $ d_0(e_1)= d_2(e_2) =T_1$ and $d_0(e_2)=d_2(e_1)=T_2$ thus $1_{T_1}\ast 1_{T_2}=T= 1_{T_2}\ast 1_{T_1}.$ 
\end{example}
   
\begin{prop}\label{01}
    For the Hall algebra $\mathcal H^T$ associated to a labelled rooted tree $T$, all multiplications are of the form $x \cdotp y = \varepsilon z$, where $\varepsilon = 0,1$.
\end{prop}

\begin{proof}
In the labelled context composites are unique if they exist, as was shown in Proposition \ref{prop:comp}.
\end{proof}

Since by Proposition \ref{01}, every coefficient in a multiplication is zero or one, we can depict the multiplication via a directed graph $G(\mathcal{H}^T)$, where the vertices are the basis vectors, and there exists a directed edge $(1_a, 1_b)$ if $1_a \ast 1_b\neq 0$. In this way, we can see which elements do not commute with each other, and in particular, for the Hall algebra to be commutative, every arrow must be symmetric. 

\begin{example}
    Recall the tree $T=Y_{3}$ that has one vertex of degree three and three vertices of degree one. Fix as a root a vertex of degree one. When the tree is unlabelled and nonplanar, the elements of $X^T_1$ are $T, F_1, P_2, P_3^1, P_3^2$, $F_2$, and $\varnothing$ where $P_i$ denotes the path on $i$ vertices, $F_i$ denotes the forest with no edges and $ P_3^i$ is the path on three vertices having as a root a vertex of degree $i$. Then on the left-hand side of Figure \ref{GH} 
    \begin{figure}[ht]
    \centering
    \begin{tikzpicture}
    \node[draw, circle] (01) at (4, 1.5){$1_{F_1}$};
    \node[draw, circle] (0) at (1.3, 1.5){$1_{F_2}$};
    \node[draw, circle] (23) at (0.5, 4){$1_{P_2}$};
    \node[draw, circle] (456) at (5, 4){$1_{P_3^2}$};
    \node[draw, circle] (78) at (1.5, 6){$1_{P_3^1}$};
    \node[draw, circle] (T) at (4,6){$1_T$};
    
    \path[loop below] (01) edge (01); 
    \draw[->](0) to (01); 
    \draw[->](23) to (01); 
    \draw[bend right,  ->](01) to (78); 
    \draw[->](456) to (01); 
    \draw[->] (0) to (23);
    \draw[bend right, ->](01) to (23);
    
    \node[draw, circle] (01x) at (11, 1.5){$1_{F_1}$};
    \node[draw, circle] (0x) at (8.3, 1.5){$1_{F_2}$};
    \node[draw, circle] (23x) at (7.5, 4){$1_{P_2}$};
    \node[draw, circle] (456x) at (12, 4){$1_{P_3^2}$};
    \node[draw, circle] (78x) at (8.5, 5){$1_{L}$};
    \node[draw, circle] (78y) at (9.5, 6){$1_{R}$};
    \node[draw, circle] (Tx) at (11,6){$1_T$};
    
    \path[loop below] (01x) edge (01x); 
    \draw[->](0x) to (01x); 
    \draw[->](23x) to (01x); 
    \draw[bend right,  ->](01x) to (78x); 
    \draw[bend right,  ->](01x) to (78y); 
    \draw[->](456x) to (01x); 
    \draw[->] (0x) to (23x);
    \draw[bend right, ->](01x) to (23x);
    \end{tikzpicture}
    \caption{The graphs $G(\mathcal{H}^T)$ and $G(\mathcal{H}^{T_p})$} \label{GH}
    \end{figure}
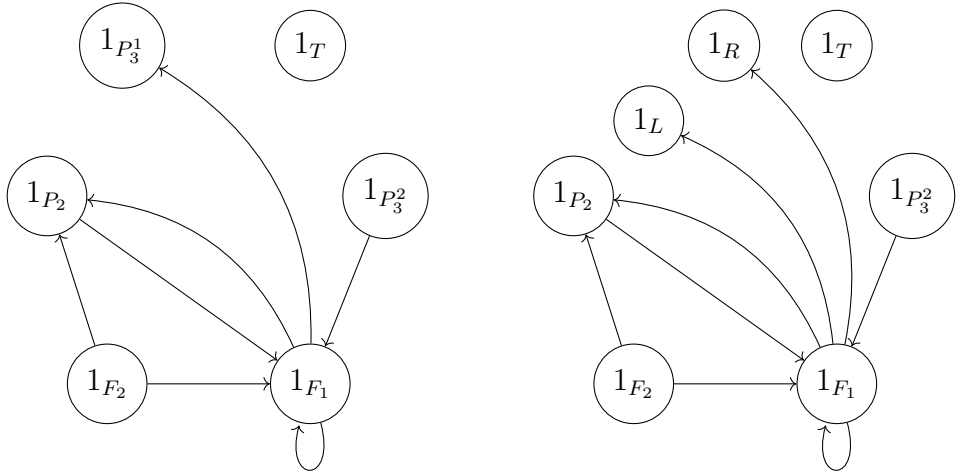
    we can see the graph $G(\mathcal{H}^T)$. If we consider the same unlabelled tree $T_p$ with its planar structure, we obtain the same elements in $X_1$, together with another tree which is also a path on three vertices with a root of degree one. Denote these two trees by $L$ and $R$ respectively. Then we can see on the right-hand side of Figure \ref{GH} the graph $G(\mathcal{H}^{T_p})$. As we can see, in both cases the only two elements different than $1_{\varnothing}$ which could commute with each other are $1_{F_1}$ and $1_{P_2}.$ However notice that $1_{P_2}\ast 1_{F_1}= 1_{P_3^1}+1_{P_3^2}$ while $1_{F_1}\ast 1_{P_2}=1_{P_{3}^1}$.
    
    If we consider the case where the tree is labelled, no two elements other than $1_\varnothing$ can commute.  
\end{example}

Let us now compare the Hall algebra $\mathcal H^T$ associated to a tree $T$ to the Hall algebra $\mathcal H^G$ associated to its underlying graph $G=U(T)$.  The following results are not difficult to check.  

\begin{prop}\label{thm:graph-hall}
Let $X^G$ be the 2-Segal set associated to a finite graph $G$.  The associated Hall algebra $\mathcal H^G=\mathcal H(X^G)$ is characterized as follows.
\begin{enumerate}
\item \label{hallunit} 
The element $1_\varnothing$ is the multiplicative identity, so $1_H \ast 1_\varnothing = 1_H = 1_\varnothing\ast 1_H$ for every subgraph $H$ of $G$.

\item \label{halldisjoint}
If $H \cap K \neq \varnothing$, then $1_H \ast 1_K = 0$.

\item \label{vertmult} 
If $a$ and $b$ are distinct vertices of $G$, then 
\[ 1_a \ast 1_b = \sum_H 1_H \]
where $H$ is a graph for which $v(H) = \{a, b\}$.  If there are $n$ edges between $a$ and $b$ in $G$, then there are $2^n$ such subgraphs $H$.

\item \label{generalmult} 
More generally, if $H \cap K = \varnothing$ for two nonempty subgraphs $H$ and $K$ of $G$, then 
\[ 1_H \ast 1_K = \sum_J 1_J \] 
where $J$ is a graph with:
\begin{itemize}
\item the disjoint union $H \union K$ as a subgraph, and

\item either $J= H \union K$ or the difference between $J$ and $H \union K$ consists of edges connecting vertices of $H$ with vertices of $K$.
\end{itemize}

\item \label{comm} It is commutative, so $1_H \ast 1_K = 1_K \ast 1_H$ for any subgraphs $H$ and $K$ of $G$.
\end{enumerate}
\end{prop}

\begin{example} 
    Consider the labelled graph $G$ on three vertices  with vertex set $\{a,b,c\}$ and edge set $\{(a,b), (b,c)\}$.  Then the non-zero products in the commutative Hall algebra $\mathcal H^G$ are $\varnothing*H=H$ and  
    \[ \begin{array}{rclrcl}
        \bullet_a \ast \bullet_b &=& 
        \bullet_a \bullet_b + \edge ab  &
        \bullet_a\bullet_c \ast \;\bullet_b &=&\bullet_a\bullet_b\bullet_c+\edge ab \bullet_c+{}_a\!\!\bullet\;\edge bc+\,G  \\ 
        \bullet_a \ast\bullet_c &=&\bullet_a \:\bullet_c
        & \bullet_a\bullet_b\ast \;\bullet_c&=&\bullet_a\bullet_b\bullet_c+\;{}_a\!\!\bullet\!\edge bc \\ 
        \bullet_b \ast \bullet_c &=& \bullet_b \bullet_c + \edge b c &   
        \bullet_a \ast \bullet_b  \,\bullet_c&=&
        \bullet_a\bullet_b\bullet_c + \edge ab \bullet_c \\
        \edge ab \ast \,\bullet_c&=& \!\!\!\!\edge ab \bullet_c+ G
        \qquad
        & \bullet_a \ast\!\! \edge bc &=&\!\!{}_a\!\!\bullet \edge bc+\,G.
    \end{array} \]
\end{example}

The following result and its linear dual follow from \cite[\S 8]{dec}.

\begin{prop}
Suppose $F\colon X\to Y$ is a simplicial map between reduced 2-Segal sets.
\begin{enumerate}
    \item If $F$ is CULF then the linear map $F^* \colon \mathcal H(Y)\to \mathcal H(X)$ given by 
    \[ F^*(1_y)\;=\sum_{F_1(x)=y}1_x\]
    defines a homomorphism of Hall algebras. 
    
    \item If $F$ is relatively Segal then the linear map ${F}_* \colon \mathcal H(X)\to \mathcal H(Y)$ given by
    \[{F}_*(1_x)\;=\;1_{F_1(x)}\]
    defines a homomorphism of Hall algebras. 
\end{enumerate} 
\end{prop}

That is, the Hall algebra construction is contravariantly functorial on CULF maps and covariantly functorial on relatively Segal maps of reduced 2-Segal sets.

Let $T$ be a rooted tree and $G$ its underlying undirected graph, and consider the associated simplicial map $U \colon X^T\to X^G$ as described in Lemma~\ref{lem:tree-graph-simplicial}.  This map does not necessarily induce an algebra homomorphism in either direction between the associated Hall algebras $\mathcal H^T$ and $\mathcal H^G$, as we saw above that $U$ is neither CULF nor relatively Segal. 

\begin{example}
If $T$ and $G$ are the tree and underlying graph displayed in \eqref{eq:TG-not-culf} then 
\[  
1_{\{b\}}*1_{\{d\}}
=0\;\text{ in }\;\mathcal H^T\qquad
1_{\{b\}}*1_{\{d\}}
=1_{\{b,d\}}\;\text{ in }\;\mathcal H^G, \]
since for the multiplication in $\mathcal H^T$ there are no 
admissible subforests of $T$ with vertex set $\{b,d\}$.
\end{example}

Let us examine the structure constants of the two Hall algebras $\mathcal H^T$ and $\mathcal H^G$ in more detail. For $F,F',F''\in X^T_1$ and $H,H',H''\in X^G_1$, write $T_{F,F'}^{F''}$ and $G_{H,H'}^{H''}$ for the subsets  (which in the labelled context must be either empty or singletons)  of those elements $\tau\in X^T_2$ and $\gamma\in X^G_2$ with $(d_0\tau,d_1\tau,d_2\tau)=(F,F'',F')$ and $(d_0\gamma,d_1\gamma,d_2\gamma)=(H,H'',H')$ respectively, so that
\[ 1_{F} \ast 1_{F'} = \sum_{F'' \in X^T_1} |T_{F,F'}^{F''}| \cdot 1_{F''} 
\qquad\text{ and }\qquad 1_{H} \ast 1_{H'} = \sum_{H'' \in X^G_1} |G_{H,H'}^{H''}|\cdot 1_{H''} \]
Since $U\colon X^G\to X^T$ is a simplicial map it defines functions 
\[U_{F,F'}^{F''}\colon T_{F,F'}^{F''}\to G_{UF,UF'}^{UF''}.\]
If these maps were bijections, then $U\colon \mathcal H^T\to\mathcal H^G$ would be an algebra homomorphism. 

\section{The associated invertible operad and cooperad} \label{operad}
	
There exists a deep connection between the theory of operads and 2-Segal sets.  In this section, we describe how to associate a coloured operad to the simplicial set $X^T$ obtained from a rooted tree $T$ . We begin with the notion of coloured operad from \cite[\S 3.6]{dk}.
	
\begin{definition}\label{defop} 
	Let $\mathfrak{C}$ be a set and let $(\mathcal{V}, \otimes, I)$ be a symmetric monoidal category.  A \emph{(non-symmetric) $\mathfrak{C}$-coloured operad} in $\mathcal{V}$ is given by the following data:
	\begin{enumerate}
		\item for each $n \geq 0$ and $(n+1)$-tuple of colours $(c_1,\dots,c_n|c_0)$ with $c_i\in \mathfrak{C}$, an object
		\[ O(c_1,\dots,c_n \mid c_0) \in \mathcal{V}; \]
        
        \item  for each colour $c$, a unit map $I_c \colon I \to O(c \mid c)$; and
        
		\item  for each $(n+1)$-tuple of colours $(c_1,\dots,c_n \mid c_0)$ and $k_i$-tuples $(c^i_1, \dots, c^i_{k_i})$ for $1\leq i \leq n$, an appropriately associative and unital composition map
        {\small
        \begin{align*}	
            O(c_1,\dots,c_n \mid c_0)\otimes O(c_{1}^1, \dots, c^{1}_{k_1} \mid c_1) \otimes \dots \otimes O(c^n_1, \dots, c^n_{k_n} \mid c_n) \to O(c_1^1, \dots, c^1_{k_1}, \dots, c^n_1, \dots, c^n_{k_n} \mid c_0).
        \end{align*} }
    \end{enumerate}
	
    A \emph{(non-symmetric) $\mathfrak{C}$-coloured cooperad} in $\mathcal{V}$ is the formal dual of an operad. More precisely, it is given by the data (1) as in Definition \ref{defop} along with 
    \begin{itemize}
    	\item[($2'$)] for each colour $c$, a counit map $\varepsilon_c \colon O(c \mid c) \to I$; 
     
    	\item[($3'$)] for each $(n+1)$-tuple of colours $(c_1,\dots,c_n)$ and $k_i$-tuples $(c'_{i,1}, \dots, c'_{i,k_i})$ for $1\leq i \leq n$, an appropriately coassociative and counital cocomposition map
            {\small
            \[  O(c_1^1, \dots,c^1_{k_1},\dots,c^n_1,\dots, c^n_{k_n} \mid c_0) \to  O(c_1, \dots, c_n \mid c_0) \otimes O(c_{1}^1, \dots, c^{1}_{k_1} \mid c_1) \otimes \dots \otimes O(c^n_1, \dots, c^n_{k_n} \mid c_n). \] } 
    \end{itemize}
\end{definition}
 
\begin{definition}
	A (co)operad is \emph{invertible} when the (co)unit and (co)composition maps are bijective.
\end{definition}
	
Given an invertible operad $\mathcal{O}$ the inverse of the unit and composition maps define a cooperadic structure on the same class of objects.  Similarly given an invertible cooperad we can use the inverse of the counit and cocomposition maps to view this structure as an operad.  In this way we may view invertible (co)operads as both operads and cooperads.  

To begin showing how the theory of operads connects to simplicial sets we construct a cooperad from the standard simplices $\Delta[n]$.

\begin{example} \cite[Example 3.6.3]{dk}
	The collection of standard simplices $\lbrace \Delta(n) \rbrace_{n \geq 0}$ forms a 1-coloured operad in $(\SSet, \union,\varnothing)$.  The operadic composition maps
	\[ \nu_{k_1, \dots ,k_n} \colon \Delta[n] \union \Delta[k_1] \union \dots \union \Delta[k_n] \to \Delta[k_1+\dots+k_n] \]
	are defined as follows. The $0$-th vertex of $\Delta[n]$ is mapped to the $0$-th vertex of $\Delta[k_1+\dots+k_n]$, the $i$-th vertex of $\Delta[n]$ is mapped to the vertex of $\Delta[k_1+\dots+k_n]$ labelled by $k_1+\dots +k_i$, and the $j$-th vertex of $\Delta[k_i]$ is mapped into the vertex of $\Delta[k_1+\dots+k_n]$ labelled by $k_1+\dots+k _{i-1}+j$.
\end{example}

This operad from standard simplices allows us to build a cooperad from any simplicial set $X$. Since $\Hom(\Delta[n],X)$ is naturally identified with $X_n$, the maps $\nu_{k_1, \dots, k_n}$ define a cooperadic cocomposition
\[ f_{k_1,\dots , k_n} \colon X_{k_1+\dots+k_n} \to X_n \times X_{k_1} \times \dots \times X_{k_n}, \]
and the collection of sets $(X_n)$ forms a 1-coloured operad in $(\Set, \times, 1)$.
Using this cocomposition we can build the following  $X_1$-coloured cooperad.

\begin{example}(\cite[Example 3.6.4]{dk})
For any simplicial set $X$, there is an $X_1$-coloured cooperad $\mathcal{Q}_X$ with
{\small
\begin{align*}
\mathcal{Q}_X(c_1, \dots, c_n | c_0)& = \lbrace  x  \in  X_n \mid \partial_{\lbrace 0,1 \rbrace}(x) = c_1,\partial_{\lbrace 1,2 \rbrace}(x) = c_2, \dots,\partial_{\lbrace n-1, n \rbrace}(x)= c_n, \partial_{\lbrace 0,n \rbrace}(x) = c_0\rbrace.
\end{align*} }
The cocomposition of this cooperad is inherited from the cocomposition maps $f_{k_1, \dots, k_n}$ obtained from the operad of standard simplices.
\end{example}
	
In \cite{dk}, Dyckerhoff and Kapranov show that the cocomposition maps of $\mathcal{Q}_X$ factor through the 2-Segal map and that the cooperad is invertible if and only if $X$ is 2-Segal, so that  the category of 2-Segal sets is equivalent to the category of invertible operads in $(\Set,\times, 1)$. In this sense we may view the cooperad associated to a 2-Segal set as either a cooperad or an operad.  In particular we may view the 2-Segal set $X^T$ associated to a rooted tree $T$ as an operad $\mathcal{Q}_{X^T}$.
 
\begin{example}
Let $T$ be a rooted tree. The \emph{invertible (co)operad associated to a rooted tree} $\mathcal{Q}_{X^T}$ is defined by the following data.  The colour set is $X^T_1$, and $\mathcal{Q}_{X^T}(c_1, \dots, c_n \mid c_0)$ is the set of layerings of admissible subforests by $n-1$ cuts, $H=L_0\supseteq L_1\supseteq\dots\supseteq L_n=\varnothing$, such that $c_0=H$ and $c_i = L_{i-1} \setminus L_i$ for $1\leq i \leq n$. The composition map 
\begin{align*}
    \nu_{k_1,\dots,k_n}\colon\mathcal{Q}_{X^T}(c_1, \dots, c_n \mid c_0) \times \mathcal{Q}_{X^T}(c^1_1, \dots,c^1_{k_1} \mid c_1) \times & \dots  \times \mathcal{Q}_{X^T}(c^n_1, \dots, c^n_{k_n} \mid c_n) \\
    &\to \mathcal{Q}_{X^T}(c_1^1, \dots, c^1_{k_1}, \dots, c^n_1, \dots, c^n_{k_n} \mid c_0)
\end{align*}
is defined by the following assignment, where each $H_i$ is the admissible subforest defined by  
$L_{i-1}\setminus L_i$:
{\small
 \begin{align*}
	& \left(\; (H \supseteq L_1\supseteq\cdots\supseteq L_n\!=\!\varnothing),\; \;
    (H_1 \supseteq L^1_1\supseteq\cdots\supseteq  L^1_{k_1}\!=\!\varnothing), \dots, (H_n \supseteq L^n_1\supseteq \cdots\supseteq L^n_{k_n}\!=\!\varnothing) \;\right) \\
	& \quad\mapsto\;(
    H\supseteq 
    L_1\!\union \!L^1_1\supseteq\cdots\supseteq 
    L_1 \supseteq 
    L_2\!\union \!L^2_1\supseteq\cdots\supseteq 
    L_2\supseteq\cdots \supseteq
    L_{n-1}\supseteq L^n_1\supseteq \cdots\supseteq L^n_{k_n}\!=\!\varnothing).
\end{align*}}
We illustrate the composition map in Figure \ref{treeoperad}. 
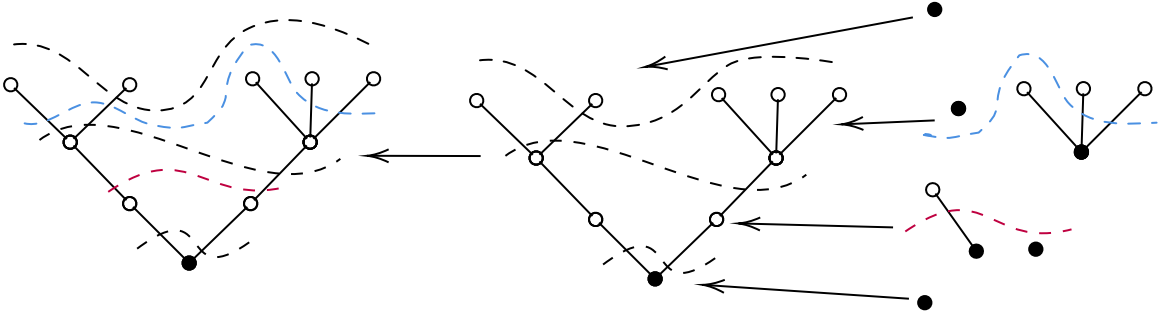
\begin{figure}
		\centering
\tikzset{every picture/.style={line width=0.75pt}} 

\begin{tikzpicture}[x=0.75pt,y=0.75pt,yscale=-1,xscale=1]

\draw    (85,123.79) -- (113.33,152.13) ;
\draw [shift={(113.33,152.13)}, rotate = 45] [color={rgb, 255:red, 0; green, 0; blue, 0 }  ][fill={rgb, 255:red, 0; green, 0; blue, 0 }  ][line width=0.75]      (0, 0) circle [x radius= 3.35, y radius= 3.35]   ;
\draw [shift={(83.33,122.13)}, rotate = 45] [color={rgb, 255:red, 0; green, 0; blue, 0 }  ][line width=0.75]      (0, 0) circle [x radius= 3.35, y radius= 3.35]   ;
Straight Lines [id:da8615763379870922] 
\draw    (142.64,123.76) -- (115.02,150.49) ;
\draw [shift={(113.33,152.13)}, rotate = 135.94] [color={rgb, 255:red, 0; green, 0; blue, 0 }  ][line width=0.75]      (0, 0) circle [x radius= 3.35, y radius= 3.35]   ;
\draw [shift={(144.33,122.13)}, rotate = 135.94] [color={rgb, 255:red, 0; green, 0; blue, 0 }  ][line width=0.75]      (0, 0) circle [x radius= 3.35, y radius= 3.35]   ;
Straight Lines [id:da7678790785475034] 
\draw    (54.97,92.81) -- (81.7,120.44) ;
\draw [shift={(83.33,122.13)}, rotate = 45.94] [color={rgb, 255:red, 0; green, 0; blue, 0 }  ][line width=0.75]      (0, 0) circle [x radius= 3.35, y radius= 3.35]   ;
\draw [shift={(53.33,91.13)}, rotate = 45.94] [color={rgb, 255:red, 0; green, 0; blue, 0 }  ][line width=0.75]      (0, 0) circle [x radius= 3.35, y radius= 3.35]   ;
Straight Lines [id:da7000329099658564] 
\draw    (55.02,89.49) -- (81.64,63.76) ;
\draw [shift={(83.33,62.13)}, rotate = 315.97] [color={rgb, 255:red, 0; green, 0; blue, 0 }  ][line width=0.75]      (0, 0) circle [x radius= 3.35, y radius= 3.35]   ;
\draw [shift={(53.33,91.13)}, rotate = 315.97] [color={rgb, 255:red, 0; green, 0; blue, 0 }  ][line width=0.75]      (0, 0) circle [x radius= 3.35, y radius= 3.35]   ;
Straight Lines [id:da6307303713716845] 
\draw    (51.64,89.49) -- (25.02,63.76) ;
\draw [shift={(23.33,62.13)}, rotate = 224.03] [color={rgb, 255:red, 0; green, 0; blue, 0 }  ][line width=0.75]      (0, 0) circle [x radius= 3.35, y radius= 3.35]   ;
\draw [shift={(53.33,91.13)}, rotate = 224.03] [color={rgb, 255:red, 0; green, 0; blue, 0 }  ][line width=0.75]      (0, 0) circle [x radius= 3.35, y radius= 3.35]   ;
Straight Lines [id:da0859668274701304] 
\draw    (145.97,120.44) -- (172.7,92.81) ;
\draw [shift={(174.33,91.13)}, rotate = 314.06] [color={rgb, 255:red, 0; green, 0; blue, 0 }  ][line width=0.75]      (0, 0) circle [x radius= 3.35, y radius= 3.35]   ;
\draw [shift={(144.33,122.13)}, rotate = 314.06] [color={rgb, 255:red, 0; green, 0; blue, 0 }  ][line width=0.75]      (0, 0) circle [x radius= 3.35, y radius= 3.35]   ;
Straight Lines [id:da9010792159735723] 
\draw    (146.91,60.87) -- (172.76,89.38) ;
\draw [shift={(174.33,91.13)}, rotate = 47.82] [color={rgb, 255:red, 0; green, 0; blue, 0 }  ][line width=0.75]      (0, 0) circle [x radius= 3.35, y radius= 3.35]   ;
\draw [shift={(145.33,59.13)}, rotate = 47.82] [color={rgb, 255:red, 0; green, 0; blue, 0 }  ][line width=0.75]      (0, 0) circle [x radius= 3.35, y radius= 3.35]   ;
Straight Lines [id:da6323590657247875] 
\draw    (204.67,60.79) -- (176,89.46) ;
\draw [shift={(174.33,91.13)}, rotate = 135] [color={rgb, 255:red, 0; green, 0; blue, 0 }  ][line width=0.75]      (0, 0) circle [x radius= 3.35, y radius= 3.35]   ;
---------
\draw [shift={(206.33,59.13)}, rotate = 135] [color={rgb, 255:red, 0; green, 0; blue, 0 }  ][line width=0.75]      (0, 0) circle [x radius= 3.35, y radius= 3.35]   ;
Straight Lines [id:da7339974626230119] 
\draw    (175.26,61.47) -- (174.41,88.78) ;
\draw [shift={(174.33,91.13)}, rotate = 91.79] [color={rgb, 255:red, 0; green, 0; blue, 0 }  ][line width=0.75]      (0, 0) circle [x radius= 3.35, y radius= 3.35]   ;
\draw [shift={(175.33,59.13)}, rotate = 91.79] [color={rgb, 255:red, 0; green, 0; blue, 0 }  ][line width=0.75]      (0, 0) circle [x radius= 3.35, y radius= 3.35]   ;
\draw  [dash pattern={on 4.5pt off 4.5pt}]  (87.08,144.88) .. controls (127.08,114.88) and (106.08,169.88) .. (146.08,139.88) ;
\draw [line width=0.75] [line join = round][line cap = round] [dash pattern={on 4.5pt off 4.5pt}]  (25,41.79) .. controls (56.27,38) and (68.68,75.7) .. (97.17,75.13) .. controls (125.65,74.55) and (120.32,44.69) .. (143.33,33.68) .. controls (166.35,22.67) and (193,36.01) .. (207.67,43.46) ;
\draw [color={rgb, 255:red, 74; green, 144; blue, 226 }  ,draw opacity=1 ][line width=0.75] [line join = round][line cap = round] [dash pattern={on 4.5pt off 4.5pt}]  (30.33,81.63) .. controls (42.33,84.38) and (53.41,70.71) .. (65.75,71.04) .. controls (76.3,71.33) and (84.86,80.7) .. (95.25,82.54) .. controls (111.03,85.34) and (106.58,83.63) .. (122.5,81.29) .. controls (129.83,74.38) and (130.5,73.29) .. (131.75,68.54) .. controls (132,57.04) and (136.25,51.04) .. (143,42.29) .. controls (157.92,37.65) and (161.58,59.48) .. (167.75,66.29) .. controls (179.14,78.87) and (197.25,76.93) .. (212.25,76.29) ;
\draw  [dash pattern={on 4.5pt off 4.5pt}]  (37.83,90.04) .. controls (77.83,60.04) and (149.58,129.54) .. (189.58,99.54) ;
\draw [color={purple}  ,draw opacity=1 ] [dash pattern={on 4.5pt off 4.5pt}]  (72.5,116.21) .. controls (114.08,87.71) and (121.42,124.21) .. (163.08,113.46) ;
\draw    (260.33,98.13) -- (205,98) ;
\draw [shift={(203,98)}, rotate = 0.12] [color={rgb, 255:red, 0; green, 0; blue, 0 }  ][line width=0.75]    (10.93,-3.29) .. controls (6.95,-1.4) and (3.31,-0.3) .. (0,0) .. controls (3.31,0.3) and (6.95,1.4) .. (10.93,3.29)   ;
\draw    (320,131.79) -- (346.67,158.46) ;
\draw [shift={(348.33,160.13)}, rotate = 45] [color={rgb, 255:red, 0; green, 0; blue, 0 }  ][line width=0.75]      (0, 0) circle [x radius= 3.35, y radius= 3.35]   ;
\draw [shift={(318.33,130.13)}, rotate = 45] [color={rgb, 255:red, 0; green, 0; blue, 0 }  ][line width=0.75]      (0, 0) circle [x radius= 3.35, y radius= 3.35]   ;
\draw    (377.64,131.76) -- (348.33,160.13) ;
\draw [shift={(348.33,160.13)}, rotate = 135.94] [color={rgb, 255:red, 0; green, 0; blue, 0 }  ][fill={rgb, 255:red, 0; green, 0; blue, 0 }  ][line width=0.75]      (0, 0) circle [x radius= 3.35, y radius= 3.35]   ;
\draw [shift={(379.33,130.13)}, rotate = 135.94] [color={rgb, 255:red, 0; green, 0; blue, 0 }  ][line width=0.75]      (0, 0) circle [x radius= 3.35, y radius= 3.35]   ;
\draw    (289.97,100.81) -- (316.7,128.44) ;
\draw [shift={(318.33,130.13)}, rotate = 45.94] [color={rgb, 255:red, 0; green, 0; blue, 0 }  ][line width=0.75]      (0, 0) circle [x radius= 3.35, y radius= 3.35]   ;
\draw [shift={(288.33,99.13)}, rotate = 45.94] [color={rgb, 255:red, 0; green, 0; blue, 0 }  ][line width=0.75]      (0, 0) circle [x radius= 3.35, y radius= 3.35]   ;
\draw    (290.02,97.49) -- (316.64,71.76) ;
\draw [shift={(318.33,70.13)}, rotate = 315.97] [color={rgb, 255:red, 0; green, 0; blue, 0 }  ][line width=0.75]      (0, 0) circle [x radius= 3.35, y radius= 3.35]   ;
\draw [shift={(288.33,99.13)}, rotate = 315.97] [color={rgb, 255:red, 0; green, 0; blue, 0 }  ][line width=0.75]      (0, 0) circle [x radius= 3.35, y radius= 3.35]   ;
\draw    (286.64,97.49) -- (260.02,71.76) ;
\draw [shift={(258.33,70.13)}, rotate = 224.03] [color={rgb, 255:red, 0; green, 0; blue, 0 }  ][line width=0.75]      (0, 0) circle [x radius= 3.35, y radius= 3.35]   ;
\draw [shift={(288.33,99.13)}, rotate = 224.03] [color={rgb, 255:red, 0; green, 0; blue, 0 }  ][line width=0.75]      (0, 0) circle [x radius= 3.35, y radius= 3.35]   ;
\draw    (380.97,128.44) -- (407.7,100.81) ;
\draw [shift={(409.33,99.13)}, rotate = 314.06] [color={rgb, 255:red, 0; green, 0; blue, 0 }  ][line width=0.75]      (0, 0) circle [x radius= 3.35, y radius= 3.35]   ;
\draw [shift={(379.33,130.13)}, rotate = 314.06] [color={rgb, 255:red, 0; green, 0; blue, 0 }  ][line width=0.75]      (0, 0) circle [x radius= 3.35, y radius= 3.35]   ;
\draw    (381.91,68.87) -- (407.76,97.38) ;
\draw [shift={(409.33,99.13)}, rotate = 47.82] [color={rgb, 255:red, 0; green, 0; blue, 0 }  ][line width=0.75]      (0, 0) circle [x radius= 3.35, y radius= 3.35]   ;
\draw [shift={(380.33,67.13)}, rotate = 47.82] [color={rgb, 255:red, 0; green, 0; blue, 0 }  ][line width=0.75]      (0, 0) circle [x radius= 3.35, y radius= 3.35]   ;
\draw    (439.67,68.79) -- (411,97.46) ;
\draw [shift={(409.33,99.13)}, rotate = 135] [color={rgb, 255:red, 0; green, 0; blue, 0 }  ][line width=0.75]      (0, 0) circle [x radius= 3.35, y radius= 3.35]   ;
\draw [shift={(441.33,67.13)}, rotate = 135] [color={rgb, 255:red, 0; green, 0; blue, 0 }  ][line width=0.75]      (0, 0) circle [x radius= 3.35, y radius= 3.35]   ;
\draw    (410.26,69.47) -- (409.41,96.78) ;
\draw [shift={(409.33,99.13)}, rotate = 91.79] [color={rgb, 255:red, 0; green, 0; blue, 0 }  ][line width=0.75]      (0, 0) circle [x radius= 3.35, y radius= 3.35]   ;
\draw [shift={(410.33,67.13)}, rotate = 91.79] [color={rgb, 255:red, 0; green, 0; blue, 0 }  ][line width=0.75]      (0, 0) circle [x radius= 3.35, y radius= 3.35]   ;
\draw  [dash pattern={on 4.5pt off 4.5pt}]  (322.08,152.88) .. controls (362.08,122.88) and (341.08,177.88) .. (381.08,147.88) ;
\draw [line width=0.75] [line join = round][line cap = round] [dash pattern={on 4.5pt off 4.5pt}]  (260,49.79) .. controls (291.27,46) and (303.68,83.7) .. (332.17,83.13) .. controls (360.65,82.55) and (370.33,61.14) .. (385.33,52.14) .. controls (400.33,43.14) and (443,51.78) .. (442.67,51.46) ;
\draw  [dash pattern={on 4.5pt off 4.5pt}]  (272.83,98.04) .. controls (312.83,68.04) and (384.58,137.54) .. (424.58,107.54) ;
\draw    (489.33,24.14) ;
\draw [shift={(489.33,24.14)}, rotate = 0] [color={rgb, 255:red, 0; green, 0; blue, 0 }  ][fill={rgb, 255:red, 0; green, 0; blue, 0 }  ][line width=0.75]      (0, 0) circle [x radius= 3.35, y radius= 3.35]   ;
\draw    (478.33,28.14) -- (345.3,52.77) ;
\draw [shift={(343.33,53.14)}, rotate = 349.51] [color={rgb, 255:red, 0; green, 0; blue, 0 }  ][line width=0.75]    (10.93,-3.29) .. controls (6.95,-1.4) and (3.31,-0.3) .. (0,0) .. controls (3.31,0.3) and (6.95,1.4) .. (10.93,3.29)   ;
\draw    (484.33,172.14) ;
\draw [shift={(484.33,172.14)}, rotate = 0] [color={rgb, 255:red, 0; green, 0; blue, 0 }  ][fill={rgb, 255:red, 0; green, 0; blue, 0 }  ][line width=0.75]      (0, 0) circle [x radius= 3.35, y radius= 3.35]   ;
\draw    (476.33,170.14) -- (374.33,163.27) ;
\draw [shift={(372.33,163.14)}, rotate = 3.85] [color={rgb, 255:red, 0; green, 0; blue, 0 }  ][line width=0.75]    (10.93,-3.29) .. controls (6.95,-1.4) and (3.31,-0.3) .. (0,0) .. controls (3.31,0.3) and (6.95,1.4) .. (10.93,3.29)   ;
\draw    (501.33,74.14) ;
\draw [shift={(501.33,74.14)}, rotate = 0] [color={rgb, 255:red, 0; green, 0; blue, 0 }  ][line width=0.75]      (0, 0) circle [x radius= 3.35, y radius= 3.35]   ;
\draw [shift={(501.33,74.14)}, rotate = 0] [color={rgb, 255:red, 0; green, 0; blue, 0 }  ][fill={rgb, 255:red, 0; green, 0; blue, 0 }  ][line width=0.75]      (0, 0) circle [x radius= 3.35, y radius= 3.35]   ;
\draw    (535.91,65.87) -- (563.33,96.13) ;
\draw [shift={(563.33,96.13)}, rotate = 47.82] [color={rgb, 255:red, 0; green, 0; blue, 0 }  ][fill={rgb, 255:red, 0; green, 0; blue, 0 }  ][line width=0.75]      (0, 0) circle [x radius= 3.35, y radius= 3.35]   ;
\draw [shift={(534.33,64.13)}, rotate = 47.82] [color={rgb, 255:red, 0; green, 0; blue, 0 }  ][line width=0.75]      (0, 0) circle [x radius= 3.35, y radius= 3.35]   ;
\draw    (593.67,65.79) -- (565,94.46) ;
\draw [shift={(563.33,96.13)}, rotate = 135] [color={rgb, 255:red, 0; green, 0; blue, 0 }  ][line width=0.75]      (0, 0) circle [x radius= 3.35, y radius= 3.35]   ;
\draw [shift={(595.33,64.13)}, rotate = 135] [color={rgb, 255:red, 0; green, 0; blue, 0 }  ][line width=0.75]      (0, 0) circle [x radius= 3.35, y radius= 3.35]   ;
\draw    (564.26,66.47) -- (563.41,93.78) ;
\draw [shift={(563.33,96.13)}, rotate = 91.79] [color={rgb, 255:red, 0; green, 0; blue, 0 }  ][line width=0.75]      (0, 0) circle [x radius= 3.35, y radius= 3.35]   ;
\draw [shift={(564.33,64.13)}, rotate = 91.79] [color={rgb, 255:red, 0; green, 0; blue, 0 }  ][line width=0.75]      (0, 0) circle [x radius= 3.35, y radius= 3.35]   ;
\draw [color={rgb, 255:red, 74; green, 144; blue, 226 }  ,draw opacity=1 ][line width=0.75] [line join = round][line cap = round] [dash pattern={on 4.5pt off 4.5pt}]  (484.25,87.54) .. controls (494.8,87.83) and (473.86,85.7) .. (484.25,87.54) .. controls (500.03,90.34) and (495.58,88.63) .. (511.5,86.29) .. controls (518.83,79.38) and (519.5,78.29) .. (520.75,73.54) .. controls (521,62.04) and (525.25,56.04) .. (532,47.29) .. controls (546.92,42.65) and (550.58,64.48) .. (556.75,71.29) .. controls (568.14,83.87) and (586.25,81.93) .. (601.25,81.29) ;
\draw    (489.33,80.14) -- (444.33,82.05) ;
\draw [shift={(442.33,82.14)}, rotate = 357.56] [color={rgb, 255:red, 0; green, 0; blue, 0 }  ][line width=0.75]    (10.93,-3.29) .. controls (6.95,-1.4) and (3.31,-0.3) .. (0,0) .. controls (3.31,0.3) and (6.95,1.4) .. (10.93,3.29)   ;
\draw    (540.33,145.14) ;
\draw [shift={(540.33,145.14)}, rotate = 0] [color={rgb, 255:red, 0; green, 0; blue, 0 }  ][fill={rgb, 255:red, 0; green, 0; blue, 0 }  ][line width=0.75]      (0, 0) circle [x radius= 3.35, y radius= 3.35]   ;
\draw    (510.33,146.14) -- (489.69,117.05) ;
\draw [shift={(488.33,115.14)}, rotate = 234.64] [color={rgb, 255:red, 0; green, 0; blue, 0 }  ][line width=0.75]      (0, 0) circle [x radius= 3.35, y radius= 3.35]   ;
\draw [shift={(510.33,146.14)}, rotate = 234.64] [color={rgb, 255:red, 0; green, 0; blue, 0 }  ][fill={rgb, 255:red, 0; green, 0; blue, 0 }  ][line width=0.75]      (0, 0) circle [x radius= 3.35, y radius= 3.35]   ;
\draw [color={purple}  ,draw opacity=1 ] [dash pattern={on 4.5pt off 4.5pt}]  (474.5,136.21) .. controls (516.08,107.71) and (516.67,145.89) .. (558.33,135.14) ;
\draw    (468.33,134.14) -- (392.33,132.19) ;
\draw [shift={(390.33,132.14)}, rotate = 1.47] [color={rgb, 255:red, 0; green, 0; blue, 0 }  ][line width=0.75]    (10.93,-3.29) .. controls (6.95,-1.4) and (3.31,-0.3) .. (0,0) .. controls (3.31,0.3) and (6.95,1.4) .. (10.93,3.29)   ;
\end{tikzpicture}
\caption{Composition in the operad associated to a tree.}
		\label{treeoperad}
\end{figure}
\end{example}
	 
\begin{remark}
As previously mentioned, there are three different choices of 2-Segal set associated to a tree $T$, according to whether $T$ is considered labelled, planar or neither.  Only in the unlabelled planar situation do we obtain sets $\mathcal{Q}_{X^T}(c_1, \dots, c_n \mid c_0)$ of $n$-ary operations of the coloured operad that are neither empty nor singletons; see Figure \ref{2ary}.
    \begin{figure}
		\centering
    \tikzset{every picture/.style={line width=0.75pt}} 
\begin{tikzpicture}[x=0.75pt,y=0.75pt,yscale=-1,xscale=1]

\draw    (105.33,35.33) ;
\draw [shift={(105.33,35.33)}, rotate = 0] [color={rgb, 255:red, 0; green, 0; blue, 0 }  ][fill={rgb, 255:red, 0; green, 0; blue, 0 }  ][line width=0.75]      (0, 0) circle [x radius= 3.35, y radius= 3.35]   ;
\draw    (204.91,20.87) -- (230.76,49.38) ;
\draw [shift={(232.33,51.13)}, rotate = 47.82] [color={rgb, 255:red, 0; green, 0; blue, 0 }  ][line width=0.75]      (0, 0) circle [x radius= 3.35, y radius= 3.35]   ;
\draw [shift={(203.33,19.13)}, rotate = 47.82] [color={rgb, 255:red, 0; green, 0; blue, 0 }  ][line width=0.75]      (0, 0) circle [x radius= 3.35, y radius= 3.35]   ;
\draw    (262.67,20.79) -- (232.33,51.13) ;
\draw [shift={(232.33,51.13)}, rotate = 135] [color={rgb, 255:red, 0; green, 0; blue, 0 }  ][fill={rgb, 255:red, 0; green, 0; blue, 0 }  ][line width=0.75]      (0, 0) circle [x radius= 3.35, y radius= 3.35]   ;
\draw [shift={(264.33,19.13)}, rotate = 135] [color={rgb, 255:red, 0; green, 0; blue, 0 }  ][line width=0.75]      (0, 0) circle [x radius= 3.35, y radius= 3.35]   ;
\draw    (155.33,22.35) -- (155.33,50) ;
\draw [shift={(155.33,50)}, rotate = 90] [color={rgb, 255:red, 0; green, 0; blue, 0 }  ][fill={rgb, 255:red, 0; green, 0; blue, 0 }  ][line width=0.75]      (0, 0) circle [x radius= 3.35, y radius= 3.35]   ;
\draw [shift={(155.33,20)}, rotate = 90] [color={rgb, 255:red, 0; green, 0; blue, 0 }  ][line width=0.75]      (0, 0) circle [x radius= 3.35, y radius= 3.35]   ;
\draw    (353.91,22.87) -- (379.76,51.38) ;
\draw [shift={(381.33,53.13)}, rotate = 47.82] [color={rgb, 255:red, 0; green, 0; blue, 0 }  ][line width=0.75]      (0, 0) circle [x radius= 3.35, y radius= 3.35]   ;
\draw [shift={(352.33,21.13)}, rotate = 47.82] [color={rgb, 255:red, 0; green, 0; blue, 0 }  ][line width=0.75]      (0, 0) circle [x radius= 3.35, y radius= 3.35]   ;
\draw    (411.67,22.79) -- (381.33,53.13) ;
\draw [shift={(381.33,53.13)}, rotate = 135] [color={rgb, 255:red, 0; green, 0; blue, 0 }  ][fill={rgb, 255:red, 0; green, 0; blue, 0 }  ][line width=0.75]      (0, 0) circle [x radius= 3.35, y radius= 3.35]   ;
\draw [shift={(413.33,21.13)}, rotate = 135] [color={rgb, 255:red, 0; green, 0; blue, 0 }  ][line width=0.75]      (0, 0) circle [x radius= 3.35, y radius= 3.35]   ;
\draw    (441.91,21.87) -- (467.76,50.38) ;
\draw [shift={(469.33,52.13)}, rotate = 47.82] [color={rgb, 255:red, 0; green, 0; blue, 0 }  ][line width=0.75]      (0, 0) circle [x radius= 3.35, y radius= 3.35]   ;
\draw [shift={(440.33,20.13)}, rotate = 47.82] [color={rgb, 255:red, 0; green, 0; blue, 0 }  ][line width=0.75]      (0, 0) circle [x radius= 3.35, y radius= 3.35]   ;
\draw    (499.67,21.79) -- (469.33,52.13) ;
\draw [shift={(469.33,52.13)}, rotate = 135] [color={rgb, 255:red, 0; green, 0; blue, 0 }  ][fill={rgb, 255:red, 0; green, 0; blue, 0 }  ][line width=0.75]      (0, 0) circle [x radius= 3.35, y radius= 3.35]   ;
\draw [shift={(501.33,20.13)}, rotate = 135] [color={rgb, 255:red, 0; green, 0; blue, 0 }  ][line width=0.75]      (0, 0) circle [x radius= 3.35, y radius= 3.35]   ;
\draw [line width=0.75] [line join = round][line cap = round] [dash pattern={on 4.5pt off 4.5pt}]  (352.33,42.47) .. controls (352.33,27.59) and (375.33,35.92) .. (375.33,16.47) ;
\draw [line width=0.75] [line join = round][line cap = round] [dash pattern={on 4.5pt off 4.5pt}]  (473.75,27.8) .. controls (476.42,26.89) and (477.75,26.3) .. (481.5,26.55) .. controls (488.13,28.89) and (492.45,35.04) .. (499.25,36.8) .. controls (504.28,38.11) and (504.67,36.64) .. (509.75,35.55) ;
\draw  [draw opacity=0] (273.43,54.56) .. controls (277,49.26) and (279.17,42.3) .. (279.17,34.67) .. controls (279.17,27.53) and (277.27,20.97) .. (274.1,15.82) -- (256.33,34.67) -- cycle ; \draw   (273.43,54.56) .. controls (277,49.26) and (279.17,42.3) .. (279.17,34.67) .. controls (279.17,27.53) and (277.27,20.97) .. (274.1,15.82) ;  
\draw  [draw opacity=0] (96.26,54.8) .. controls (91.44,49.48) and (88.5,42.41) .. (88.5,34.67) .. controls (88.5,27.27) and (91.17,20.51) .. (95.61,15.28) -- (118.5,34.67) -- cycle ; \draw   (96.26,54.8) .. controls (91.44,49.48) and (88.5,42.41) .. (88.5,34.67) .. controls (88.5,27.27) and (91.17,20.51) .. (95.61,15.28) ;  
\draw    (299.85,33.67) -- (313.68,33.83)(299.82,36.67) -- (313.65,36.83) ;
\draw   (342.67,16) .. controls (338,16.06) and (335.7,18.42) .. (335.76,23.09) -- (335.79,25.71) .. controls (335.88,32.38) and (333.59,35.74) .. (328.92,35.8) .. controls (333.59,35.74) and (335.96,39.04) .. (336.04,45.71)(336.01,42.71) -- (336.08,48.59) .. controls (336.14,53.26) and (338.5,55.56) .. (343.17,55.5) ;
\draw   (519.67,54.5) .. controls (524.34,54.56) and (526.7,52.26) .. (526.76,47.59) -- (526.79,44.69) .. controls (526.88,38.02) and (529.25,34.72) .. (533.92,34.78) .. controls (529.25,34.72) and (526.96,31.36) .. (527.05,24.69)(527.01,27.69) -- (527.08,22.09) .. controls (527.14,17.42) and (524.84,15.06) .. (520.17,15) ;
\draw    (188.33,23.33) -- (188.33,47.33) ;

\draw (126.92,37.92) node [anchor=north west][inner sep=0.75pt]   [align=left] {\textbf{,}};
\draw (424,41) node [anchor=north west][inner sep=0.75pt]   [align=left] {\textbf{,}};
\end{tikzpicture}
\caption{A non-empty, non-singleton set of binary operations in the coloured operad associated to a planar unlabelled tree.}
		\label{2ary}
	\end{figure}
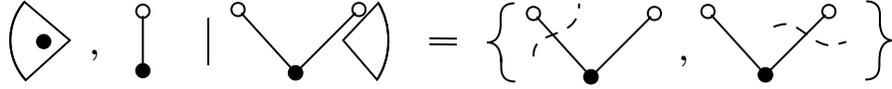
\end{remark}
	
\begin{example}
Let $G$ be a graph, and let $X^G$ be its associated 2-Segal set.  Following \cite[Example 2.3]{2s}, the invertible (co)operad $\mathcal{Q}_{X^{G}}$ associated to $G$ is defined as follows.  The colour set $\mathfrak{C} = X^G_1$, the set of all subgraphs of $G$,
\[ \mathcal{Q}_{X^G}(c_1,\dots,c_n \mid c_0) = \lbrace (H;S_1,\dots,S_n) \in X^G_n  \mid S_1 = c_1, \dots, S_n = c_n, H = c_0 \rbrace, \]
and the composition map $\nu_{k_1,\dots,k_n}$ is given by 
\[ \left((H; S_1, \dots, S_n), (S_1; S^1_1, \dots, S^1_{k_1}), \dots, (S_n; S^n_1,\dots, S^n_{k_n}) \right)\mapsto (H; S^1_1,\dots, S^1_{k_1}, \dots, S^n_1,\dots, S^n_{k_n}). \]
The composition map is illustrated in Figure \ref{graphoperad}.
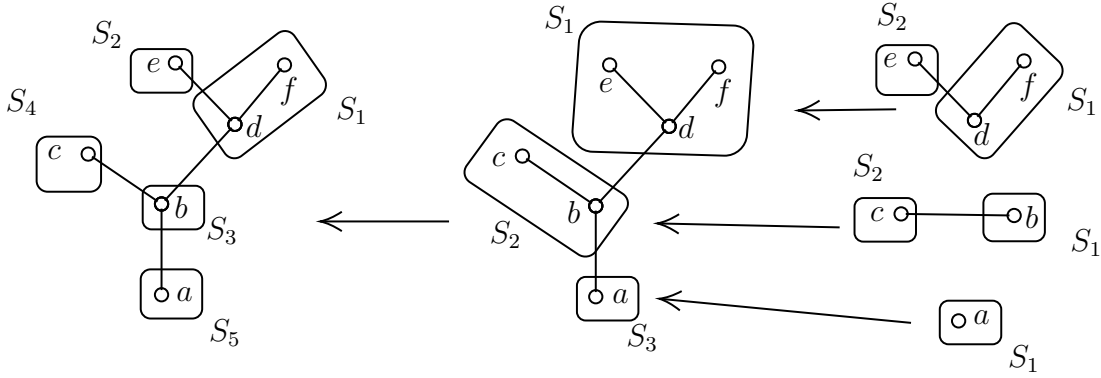
\begin{figure}
    \centering

\tikzset{every picture/.style={line width=0.75pt}} 

\begin{tikzpicture}[x=0.75pt,y=0.75pt,yscale=-1,xscale=1]

\draw    (100,121.35) -- (100,162.46) ;
\draw [shift={(100,164.81)}, rotate = 90] [color={rgb, 255:red, 0; green, 0; blue, 0 }  ][line width=0.75]      (0, 0) circle [x radius= 3.35, y radius= 3.35]   ;
\draw [shift={(100,119)}, rotate = 90] [color={rgb, 255:red, 0; green, 0; blue, 0 }  ][line width=0.75]      (0, 0) circle [x radius= 3.35, y radius= 3.35]   ;
\draw    (64.94,95.13) -- (98.06,117.68) ;
\draw [shift={(100,119)}, rotate = 34.24] [color={rgb, 255:red, 0; green, 0; blue, 0 }  ][line width=0.75]      (0, 0) circle [x radius= 3.35, y radius= 3.35]   ;
\draw [shift={(63,93.81)}, rotate = 34.24] [color={rgb, 255:red, 0; green, 0; blue, 0 }  ][line width=0.75]      (0, 0) circle [x radius= 3.35, y radius= 3.35]   ;
\draw    (135.41,80.54) -- (101.59,117.27) ;
\draw [shift={(100,119)}, rotate = 132.64] [color={rgb, 255:red, 0; green, 0; blue, 0 }  ][line width=0.75]      (0, 0) circle [x radius= 3.35, y radius= 3.35]   ;
\draw [shift={(137,78.81)}, rotate = 132.64] [color={rgb, 255:red, 0; green, 0; blue, 0 }  ][line width=0.75]      (0, 0) circle [x radius= 3.35, y radius= 3.35]   ;
\draw    (108.63,49.5) -- (135.37,77.12) ;
\draw [shift={(137,78.81)}, rotate = 45.94] [color={rgb, 255:red, 0; green, 0; blue, 0 }  ][line width=0.75]      (0, 0) circle [x radius= 3.35, y radius= 3.35]   ;
\draw [shift={(107,47.81)}, rotate = 45.94] [color={rgb, 255:red, 0; green, 0; blue, 0 }  ][line width=0.75]      (0, 0) circle [x radius= 3.35, y radius= 3.35]   ;
\draw    (160.5,50.62) -- (138.5,77.01) ;
\draw [shift={(137,78.81)}, rotate = 129.81] [color={rgb, 255:red, 0; green, 0; blue, 0 }  ][line width=0.75]      (0, 0) circle [x radius= 3.35, y radius= 3.35]   ;
\draw [shift={(162,48.81)}, rotate = 129.81] [color={rgb, 255:red, 0; green, 0; blue, 0 }  ][line width=0.75]      (0, 0) circle [x radius= 3.35, y radius= 3.35]   ;
\draw    (245,127.81) -- (182,127.81) ;
\draw [shift={(180,127.81)}, rotate = 360] [color={rgb, 255:red, 0; green, 0; blue, 0 }  ][line width=0.75]    (10.93,-4.9) .. controls (6.95,-2.3) and (3.31,-0.67) .. (0,0) .. controls (3.31,0.67) and (6.95,2.3) .. (10.93,4.9)   ;
\draw   (84.5,58.41) .. controls (84.5,60.84) and (86.47,62.81) .. (88.9,62.81) -- (111.1,62.81) .. controls (113.53,62.81) and (115.5,60.84) .. (115.5,58.41) -- (115.5,45.21) .. controls (115.5,42.78) and (113.53,40.81) .. (111.1,40.81) -- (88.9,40.81) .. controls (86.47,40.81) and (84.5,42.78) .. (84.5,45.21) -- cycle ;
\draw   (449.5,133.6) .. controls (449.5,136.03) and (451.47,138) .. (453.9,138) -- (476.1,138) .. controls (478.53,138) and (480.5,136.03) .. (480.5,133.6) -- (480.5,120.4) .. controls (480.5,117.97) and (478.53,116) .. (476.1,116) -- (453.9,116) .. controls (451.47,116) and (449.5,117.97) .. (449.5,120.4) -- cycle ;
\draw   (90.5,127.41) .. controls (90.5,129.84) and (92.47,131.81) .. (94.9,131.81) -- (117.1,131.81) .. controls (119.53,131.81) and (121.5,129.84) .. (121.5,127.41) -- (121.5,114.21) .. controls (121.5,111.78) and (119.53,109.81) .. (117.1,109.81) -- (94.9,109.81) .. controls (92.47,109.81) and (90.5,111.78) .. (90.5,114.21) -- cycle ;
\draw   (89.5,171.85) .. controls (89.5,174.59) and (91.72,176.81) .. (94.46,176.81) -- (115.54,176.81) .. controls (118.28,176.81) and (120.5,174.59) .. (120.5,171.85) -- (120.5,156.96) .. controls (120.5,154.22) and (118.28,152) .. (115.54,152) -- (94.46,152) .. controls (91.72,152) and (89.5,154.22) .. (89.5,156.96) -- cycle ;
\draw   (131.91,93.39) .. controls (133.87,96.18) and (137.66,96.77) .. (140.37,94.69) -- (180.42,64.01) .. controls (183.13,61.93) and (183.74,57.98) .. (181.77,55.18) -- (167.09,34.24) .. controls (165.13,31.44) and (161.34,30.86) .. (158.63,32.93) -- (118.58,63.62) .. controls (115.87,65.69) and (115.26,69.65) .. (117.23,72.45) -- cycle ;
\draw   (37,107.25) .. controls (37,110.32) and (39.49,112.81) .. (42.56,112.81) -- (63.94,112.81) .. controls (67.01,112.81) and (69.5,110.32) .. (69.5,107.25) -- (69.5,90.56) .. controls (69.5,87.49) and (67.01,85) .. (63.94,85) -- (42.56,85) .. controls (39.49,85) and (37,87.49) .. (37,90.56) -- cycle ;
\draw    (319,122.35) -- (319,163.46) ;
\draw [shift={(319,165.81)}, rotate = 90] [color={rgb, 255:red, 0; green, 0; blue, 0 }  ][line width=0.75]      (0, 0) circle [x radius= 3.35, y radius= 3.35]   ;
\draw [shift={(319,120)}, rotate = 90] [color={rgb, 255:red, 0; green, 0; blue, 0 }  ][line width=0.75]      (0, 0) circle [x radius= 3.35, y radius= 3.35]   ;
\draw    (283.94,96.13) -- (317.06,118.68) ;
\draw [shift={(319,120)}, rotate = 34.24] [color={rgb, 255:red, 0; green, 0; blue, 0 }  ][line width=0.75]      (0, 0) circle [x radius= 3.35, y radius= 3.35]   ;
\draw [shift={(282,94.81)}, rotate = 34.24] [color={rgb, 255:red, 0; green, 0; blue, 0 }  ][line width=0.75]      (0, 0) circle [x radius= 3.35, y radius= 3.35]   ;
\draw    (354.41,81.54) -- (320.59,118.27) ;
\draw [shift={(319,120)}, rotate = 132.64] [color={rgb, 255:red, 0; green, 0; blue, 0 }  ][line width=0.75]      (0, 0) circle [x radius= 3.35, y radius= 3.35]   ;
\draw [shift={(356,79.81)}, rotate = 132.64] [color={rgb, 255:red, 0; green, 0; blue, 0 }  ][line width=0.75]      (0, 0) circle [x radius= 3.35, y radius= 3.35]   ;
\draw    (327.63,50.5) -- (354.37,78.12) ;
\draw [shift={(356,79.81)}, rotate = 45.94] [color={rgb, 255:red, 0; green, 0; blue, 0 }  ][line width=0.75]      (0, 0) circle [x radius= 3.35, y radius= 3.35]   ;
\draw [shift={(326,48.81)}, rotate = 45.94] [color={rgb, 255:red, 0; green, 0; blue, 0 }  ][line width=0.75]      (0, 0) circle [x radius= 3.35, y radius= 3.35]   ;
\draw    (379.5,51.62) -- (357.5,78.01) ;
\draw [shift={(356,79.81)}, rotate = 129.81] [color={rgb, 255:red, 0; green, 0; blue, 0 }  ][line width=0.75]      (0, 0) circle [x radius= 3.35, y radius= 3.35]   ;
\draw [shift={(381,49.81)}, rotate = 129.81] [color={rgb, 255:red, 0; green, 0; blue, 0 }  ][line width=0.75]      (0, 0) circle [x radius= 3.35, y radius= 3.35]   ;
\draw   (492.5,185.41) .. controls (492.5,187.84) and (494.47,189.81) .. (496.9,189.81) -- (519.1,189.81) .. controls (521.53,189.81) and (523.5,187.84) .. (523.5,185.41) -- (523.5,172.21) .. controls (523.5,169.78) and (521.53,167.81) .. (519.1,167.81) -- (496.9,167.81) .. controls (494.47,167.81) and (492.5,169.78) .. (492.5,172.21) -- cycle ;
\draw   (514.5,133.01) .. controls (514.5,135.66) and (516.65,137.81) .. (519.3,137.81) -- (540.7,137.81) .. controls (543.35,137.81) and (545.5,135.66) .. (545.5,133.01) -- (545.5,118.61) .. controls (545.5,115.96) and (543.35,113.81) .. (540.7,113.81) -- (519.3,113.81) .. controls (516.65,113.81) and (514.5,115.96) .. (514.5,118.61) -- cycle ;
\draw   (460.5,58.96) .. controls (460.5,61.74) and (462.76,64) .. (465.54,64) -- (486.46,64) .. controls (489.24,64) and (491.5,61.74) .. (491.5,58.96) -- (491.5,43.85) .. controls (491.5,41.07) and (489.24,38.81) .. (486.46,38.81) -- (465.54,38.81) .. controls (462.76,38.81) and (460.5,41.07) .. (460.5,43.85) -- cycle ;
\draw   (309.5,173.41) .. controls (309.5,175.84) and (311.47,177.81) .. (313.9,177.81) -- (336.1,177.81) .. controls (338.53,177.81) and (340.5,175.84) .. (340.5,173.41) -- (340.5,160.21) .. controls (340.5,157.78) and (338.53,155.81) .. (336.1,155.81) -- (313.9,155.81) .. controls (311.47,155.81) and (309.5,157.78) .. (309.5,160.21) -- cycle ;
\draw    (481.63,47.5) -- (508.37,75.12) ;
\draw [shift={(510,76.81)}, rotate = 45.94] [color={rgb, 255:red, 0; green, 0; blue, 0 }  ][line width=0.75]      (0, 0) circle [x radius= 3.35, y radius= 3.35]   ;
\draw [shift={(480,45.81)}, rotate = 45.94] [color={rgb, 255:red, 0; green, 0; blue, 0 }  ][line width=0.75]      (0, 0) circle [x radius= 3.35, y radius= 3.35]   ;
\draw    (533.5,48.62) -- (511.5,75.01) ;
\draw [shift={(510,76.81)}, rotate = 129.81] [color={rgb, 255:red, 0; green, 0; blue, 0 }  ][line width=0.75]      (0, 0) circle [x radius= 3.35, y radius= 3.35]   ;
\draw [shift={(535,46.81)}, rotate = 129.81] [color={rgb, 255:red, 0; green, 0; blue, 0 }  ][line width=0.75]      (0, 0) circle [x radius= 3.35, y radius= 3.35]   ;
\draw   (511.03,94.25) .. controls (513.5,96.61) and (517.33,96.45) .. (519.58,93.88) -- (552.86,55.96) .. controls (555.11,53.39) and (554.94,49.4) .. (552.46,47.04) -- (533.97,29.37) .. controls (531.5,27.01) and (527.67,27.18) .. (525.42,29.74) -- (492.14,67.66) .. controls (489.89,70.23) and (490.06,74.23) .. (492.54,76.59) -- cycle ;
\draw   (307.3,81.61) .. controls (306.87,87.52) and (311.33,92.46) .. (317.25,92.63) -- (383.77,94.61) .. controls (389.69,94.78) and (394.84,90.13) .. (395.27,84.22) -- (398.45,39.96) .. controls (398.88,34.05) and (394.43,29.11) .. (388.5,28.93) -- (321.98,26.96) .. controls (316.06,26.78) and (310.91,31.44) .. (310.49,37.35) -- cycle ;
\draw   (253.96,99.53) .. controls (251.94,102.29) and (252.61,106.06) .. (255.46,107.96) -- (310.21,144.47) .. controls (313.05,146.37) and (316.99,145.67) .. (319,142.91) -- (334.08,122.25) .. controls (336.09,119.49) and (335.42,115.71) .. (332.58,113.82) -- (277.83,77.3) .. controls (274.99,75.41) and (271.05,76.11) .. (269.04,78.87) -- cycle ;
\draw    (502,178) ;
\draw [shift={(502,178)}, rotate = 0] [color={rgb, 255:red, 0; green, 0; blue, 0 }  ][line width=0.75]      (0, 0) circle [x radius= 3.35, y radius= 3.35]   ;
\draw    (475.35,124.03) -- (527.65,124.78) ;
\draw [shift={(530,124.81)}, rotate = 0.82] [color={rgb, 255:red, 0; green, 0; blue, 0 }  ][line width=0.75]      (0, 0) circle [x radius= 3.35, y radius= 3.35]   ;
\draw [shift={(473,124)}, rotate = 0.82] [color={rgb, 255:red, 0; green, 0; blue, 0 }  ][line width=0.75]      (0, 0) circle [x radius= 3.35, y radius= 3.35]   ;
\draw    (470.5,71.21) -- (423,71.79) ;
\draw [shift={(421,71.81)}, rotate = 359.31] [color={rgb, 255:red, 0; green, 0; blue, 0 }  ][line width=0.75]    (10.93,-4.9) .. controls (6.95,-2.3) and (3.31,-0.67) .. (0,0) .. controls (3.31,0.67) and (6.95,2.3) .. (10.93,4.9)   ;
\draw    (478,179) -- (352.99,167) ;
\draw [shift={(351,166.81)}, rotate = 5.48] [color={rgb, 255:red, 0; green, 0; blue, 0 }  ][line width=0.75]    (10.93,-4.9) .. controls (6.95,-2.3) and (3.31,-0.67) .. (0,0) .. controls (3.31,0.67) and (6.95,2.3) .. (10.93,4.9)   ;
\draw    (442,130.81) -- (352,128.86) ;
\draw [shift={(350,128.81)}, rotate = 1.25] [color={rgb, 255:red, 0; green, 0; blue, 0 }  ][line width=0.75]    (10.93,-4.9) .. controls (6.95,-2.3) and (3.31,-0.67) .. (0,0) .. controls (3.31,0.67) and (6.95,2.3) .. (10.93,4.9)   ;

\draw (106,159.4) node [anchor=north west][inner sep=0.75pt]    {$a$};
\draw (105,113.4) node [anchor=north west][inner sep=0.75pt]    {$b$};
\draw (41,87.96) node [anchor=north west][inner sep=0.75pt]    {$c$};
\draw (141,72.4) node [anchor=north west][inner sep=0.75pt]    {$d$};
\draw (90.9,44.21) node [anchor=north west][inner sep=0.75pt]    {$e$};
\draw (158,53.21) node [anchor=north west][inner sep=0.75pt]    {$f$};
\draw (326,160.4) node [anchor=north west][inner sep=0.75pt]    {$a$};
\draw (303.06,116.2) node [anchor=north west][inner sep=0.75pt]    {$b$};
\draw (265,91.4) node [anchor=north west][inner sep=0.75pt]    {$c$};
\draw (359.06,73.09) node [anchor=north west][inner sep=0.75pt]    {$d$};
\draw (319,54.21) node [anchor=north west][inner sep=0.75pt]    {$e$};
\draw (377,54.21) node [anchor=north west][inner sep=0.75pt]    {$f$};
\draw (508,76.4) node [anchor=north west][inner sep=0.75pt]    {$d$};
\draw (462.5,41.25) node [anchor=north west][inner sep=0.75pt]    {$e$};
\draw (531,51.21) node [anchor=north west][inner sep=0.75pt]    {$f$};
\draw (455.9,119.4) node [anchor=north west][inner sep=0.75pt]    {$c$};
\draw (534,118.4) node [anchor=north west][inner sep=0.75pt]    {$b$};
\draw (508,171.4) node [anchor=north west][inner sep=0.75pt]    {$a$};
\draw (122.5,175.81) node [anchor=north west][inner sep=0.75pt]    {$S_{5}$};
\draw (63,25.4) node [anchor=north west][inner sep=0.75pt]    {$S_{2}$};
\draw (121,124.4) node [anchor=north west][inner sep=0.75pt]    {$S_{3}$};
\draw (20,59.4) node [anchor=north west][inner sep=0.75pt]    {$S_{4}$};
\draw (525.5,188.81) node [anchor=north west][inner sep=0.75pt]    {$S_{1}$};
\draw (333,178.4) node [anchor=north west][inner sep=0.75pt]    {$S_{3}$};
\draw (186.5,63.81) node [anchor=north west][inner sep=0.75pt]    {$S_{1}$};
\draw (291.5,17.81) node [anchor=north west][inner sep=0.75pt]    {$S_{1}$};
\draw (554.86,59.36) node [anchor=north west][inner sep=0.75pt]    {$S_{1}$};
\draw (264,126.4) node [anchor=north west][inner sep=0.75pt]    {$S_{2}$};
\draw (459,16.4) node [anchor=north west][inner sep=0.75pt]    {$S_{2}$};
\draw (557,130.4) node [anchor=north west][inner sep=0.75pt]    {$S_{1}$};
\draw (447,92.4) node [anchor=north west][inner sep=0.75pt]    {$S_{2}$};

\end{tikzpicture}
    \caption{Composition in the operad associated to a graph.}
	\label{graphoperad}
\end{figure}
\end{example}

\end{document}